\documentclass[11pt]{article}

\usepackage{amsmath,amssymb,amsthm}
\usepackage[margin=1in]{geometry}
\usepackage[hidelinks]{hyperref}
\allowdisplaybreaks
\numberwithin{equation}{section}

\hypersetup{
  pdftitle={Componentwise and Measure-Sensitive Bounds for One-Dimensional Time-Frequency Localization},
  pdfauthor={Ahmadreza Azimifard},
  pdfsubject={Time-frequency localization, residual energy, measurable windows,
    soft masks, moments, and separated-interaction kernels},
  pdfkeywords={time-frequency localization, plunge region, defect trace,
    residual energy, Schatten bounds, Chebyshev-Bessel approximation,
    finite-rank approximation, measurable windows, soft masks,
    separated interactions}
}

\newtheorem{theorem}{Theorem}[section]
\newtheorem{proposition}[theorem]{Proposition}
\newtheorem{corollary}[theorem]{Corollary}
\newtheorem{lemma}[theorem]{Lemma}
\newtheorem{example}[theorem]{Example}
\theoremstyle{definition}

\theoremstyle{remark}
\newtheorem{remark}[theorem]{Remark}

\newcommand{\R}{\mathbb R}
\newcommand{\F}{\mathcal F}
\newcommand{\Sp}{\mathcal S}
\newcommand{\one}{\mathbf 1}
\newcommand{\eps}{\epsilon}
\newcommand{\Lpl}{\Lambda}
\newcommand{\Csharp}{C_{\sharp}}
\newcommand{\logplus}{\log_{2,+}}

\title{Componentwise and Measure-Sensitive Bounds for\\
One-Dimensional Time--Frequency Localization}

\author{Ahmadreza Azimifard}

\date{}

\begin{document}
\maketitle

\begin{abstract}
We prove explicit upper bounds for the number of eigenvalues of a
one-dimensional time--frequency localization operator in an open transition
window.  An abstract residual-energy principle combines a rank-$N$ approximant
with the Hilbert--Schmidt energy of its remainder and retains the strict integer
correction imposed by the open threshold.  For the Fourier kernel, centered
least-squares polynomials, optimized moments, and Chebyshev--Bessel truncations
improve or match the diameter--Taylor certificate; on interval blocks the
moment envelope gains the exact factor $1/(2N+1)$.  Finite measurable
partitions yield best-polynomial matrix and residual-energy bounds, together
with two-level threshold allocations that exploit spatial or Fourier-side
orthogonality and remain stable under large empty gaps.

A complementary defect channel gives a componentwise variational Schatten
envelope for finite interval unions.  For arbitrary finite-measure hard
windows, exact cross-boundary and symmetric-difference formulas give strict
transition counting and a certified near-one cluster that can be subtracted
from direct counts.  We also treat integrable soft masks, a root-sum-square
core--tail decomposition, and unbounded windows controlled by moments.  An
explicit unbounded finite-measure set has all polynomial moments and infinite
fractional translation perimeter for every order in $(0,1]$, yet admits an
$O(\log(1/\epsilon)/\log\log(1/\epsilon))$ certificate.  Finally, kernels with
finitely many bounded separated phase interactions and finite separated
amplitude rank admit factorial, least-squares, Chebyshev--Bessel,
span-compressed, and anisotropic-degree bounds.  For bounded hard Fourier
windows, the final hybrid is no larger than the corresponding earlier
certificates proved here; no universal comparison with all regular-domain
estimates is asserted.
\end{abstract}

\section{Introduction}
\label{sec:introduction}

Let $P_A$ denote multiplication by the indicator of a measurable set
$A\subset\mathbb R$, and let
\[
  Q_B:=\mathcal F^{-1}P_B\mathcal F,
  \qquad
  (\mathcal Ff)(\xi)=\int_{\mathbb R}e^{-2\pi i x\xi}f(x)\,dx.
\]
For $c>0$ the one-dimensional time--frequency localization operator is
\begin{equation}
  K_{A,B,c}:=P_{cB}\mathcal FP_A,
  \qquad
  S_{A,B,c}:=K_{A,B,c}^*K_{A,B,c}=P_AQ_{cB}P_A.
  \label{eq:intro-KS}
\end{equation}
The operator $K_{A,B,c}$ is a contraction, and
$S_{A,B,c}=K_{A,B,c}^*K_{A,B,c}$ is a positive contraction.  For
$0<\epsilon<1/2$ we study the open-window count
\begin{equation}
  \Lambda_\epsilon(A,cB)
  :=\operatorname{rank}\mathbf1_{(\epsilon,1-\epsilon)}(S_{A,B,c}).
  \label{eq:intro-plunge-count}
\end{equation}
When a general positive contraction $S$ replaces $S_{A,B,c}$, we write
$\Lambda_\epsilon(S)$ for the same spectral count.
For a compact operator $L$, singular values are ordered
$s_1(L)\ge s_2(L)\ge\cdots$, repeated with multiplicity, and
\[
  n(\tau;L):=\#\{k:s_k(L)>\tau\}.
\]

Classical work identifies the asymptotic transition profile for interval
windows, while recent quantitative results control the count under interval,
boundary-size, or fractional-regularity hypotheses
\cite{LandauWidom,KarnikRombergDavenport,KulikovLarsen,
HughesIsraelMayeli,AzimifardIndependent,AzimifardDeterminant}.
Those results answer difficult questions that are not displaced here.  In
particular, this paper does not improve the Landau--Widom coefficient, the
moving-threshold determinant asymptotics of
\cite{AzimifardDeterminant}, or the best known logarithmic order in their
respective regular regimes.  Our question is different:

\begin{quote}
Can one obtain explicit finite-parameter transition bounds that retain local
measure and local geometry, continue to work for soft masks and certain
unbounded rough windows, and extend beyond the Fourier phase itself?
\end{quote}

We answer this question by separating the spectral argument from the kernel
approximation.  Two complementary engines result.

\paragraph{The defect/Schatten engine.}
For a positive contraction $S$, the transition count is exactly the strict
singular-value count of $(S-S^2)^{1/2}$ at
$\sqrt{\epsilon(1-\epsilon)}$.  For finite unions of intervals, the
one-sided block estimate from \cite[Proposition~5.1]{AzimifardIndependent}
can then be assembled without replacing all component lengths by their
maxima.  An independent dilation variable for every spatial--frequency pair
produces the explicit two-branch envelope $G_*$ in
Theorem~\ref{thm:sch-componentwise}.  This engine is useful across a
broad threshold range and preserves interval-component data.
For arbitrary finite-measure hard windows, we use the cross-boundary
defect-trace identity of \cite[Proposition~1.2]{HughesIsraelMayeli} and its
frequency-dual and symmetric-difference forms.  The defect also certifies part
of the spectral cluster near one, which can be subtracted from any valid
direct upper count.

\paragraph{The direct finite-rank engine.}
If $S=K^*K$, then
$\Lambda_\epsilon(S)\le n(\sqrt\epsilon;K)$.  A rank-at-most-$N$
approximation of
$K$ with error at most $\sqrt\epsilon$ therefore gives an integer count of at
most $N$.  More generally, a Hilbert--Schmidt residual of squared energy
$H_N$ gives the sharper strict count
$N+\kappa(H_N/\epsilon)$ even before the entire remainder falls below the
threshold.  On the imaginary axis the exponential Taylor remainder is exactly
bounded by $|t|^N/N!$, with no generic $e^{|t|}$ loss.  If either window has
measure zero, then $K=0$ and the count vanishes.  For positive-measure bounded
measurable windows, this yields
\begin{equation}
  s_{N+1}(P_{cB}\mathcal FP_A)
  \le \sqrt{c|A||B|}\,
       \frac{\left(\frac{\pi c}{2}
       \operatorname{diam}_{\rm ess}(A)
       \operatorname{diam}_{\rm ess}(B)\right)^N}{N!}
  \label{eq:intro-main-envelope}
\end{equation}
for arbitrary bounded measurable $A$ and $B$, not only intervals.  The exact
factorial certificate is primary; a principal-branch Lambert--$W_0$
expression is a convenient explicit surrogate.

The distinction between the two engines matters.  The defect identity uses
the smaller threshold $\sqrt{\epsilon(1-\epsilon)}$ but is exact for the
two-sided spectral window.  The direct method uses the larger threshold
$\sqrt\epsilon$ and discards the upper endpoint, but can be much sharper when
the same kernel approximation controls both routes.  We retain the minimum of
available certificates rather than claim that one construction dominates the
other.

\subsection{Main contributions and exact scope}
\label{sec:intro-contributions}

The paper establishes the following results.  Statements imported from
earlier work are identified at the point of use.

\begin{enumerate}
  \item Section~\ref{sec:abstract-spectral} gives an operator-independent
  transfer theorem, a finite-rank approximation rule, and a residual-energy
  counting lemma with exact strict-threshold rounding.  It also gives direct
  and weighted-defect master bounds, the exact discrete factorial certificate,
  and a careful Lambert--$W_0$ inversion including zero-parameter and rounding
  cases.  Thus Fourier analysis is isolated as one source of approximation
  errors, not built into the spectral mechanism.

  \item Section~\ref{sec:componentwise-schatten} develops the componentwise
  variational Schatten theorem for finite interval unions.  Every product
  $c|I_i||J_j|$ is retained, every block receives its own dilation scale, and
  the remaining Schatten exponent is optimized only after the finite sum is
  assembled.

  \item Section~\ref{sec:fourier-extensions} proves
  \eqref{eq:intro-main-envelope}, extends the bounded-support envelope to
  integrable masks $0\le u,v\le1$, and gives an exact two-channel defect
  factorization for soft effects.  For bounded hard windows it replaces the
  fixed Taylor polynomial by a computable least-squares polynomial in the
  centered product, optimized moment centres, or a Chebyshev--Bessel
  truncation.  Their residual-energy minimum is formally no worse than the
  original factorial certificate.  A moment version treats unbounded hard or
  soft windows at every order for which the displayed moments are finite:
  \[
    s_{N+1}(P_{cB}\mathcal FP_A)
    \le \frac{\sqrt c(2\pi c)^N}{N!}
       \left(\int_A|x-\alpha|^{2N}\,dx
             \int_B|\eta-\beta|^{2N}\,d\eta\right)^{1/2}.
  \]
  A bounded-core plus tail formula supplies a second route when global high
  moments are inconvenient; its three disjoint kernel errors combine by a
  root-sum-square identity rather than a triangle inequality.

  \item For bounded measurable windows, Section~\ref{sec:partitions} proves a
  finite-partition theorem in which local Taylor errors form a scalar matrix
  whose $\ell^2\!\to\ell^2$ norm controls the global operator error.  This
  retains local measures and diameters and is invariant under enlarging empty
  gaps.  Local least-squares polynomials and optimized moments give a formal
  hierarchy below the diameter matrix; for an interval--interval block, the
  moment remainder is exactly the diameter--Taylor remainder divided by
  $2N+1$.  Separately, the weighted Ky--Fan theorem applies to arbitrary
  finite-measure windows and distributes the exact defect threshold among
  measurable pieces.  A stronger two-level allocation uses an outer
  $\ell^2$ budget over orthogonal spatial inputs and an inner $\ell^1$ budget
  over frequency pieces, together with an independently proved Fourier-dual
  column version.

  \item Section~\ref{sec:rough-example} constructs an explicit unbounded set
  $A_*$ of finite measure having all polynomial moments but infinite
  translation fractional perimeter for every $0<\gamma\le1$.  The moment
  theorem still gives
  \[
     \Lambda_\epsilon(A_*,cA_*)
     =O_c\!\left(\frac{\log(1/\epsilon)}
                        {\log\log(1/\epsilon)}\right).
  \]
  Thus the assumptions of being a finite interval union, having finite
  essential diameter or the relevant finite upper-Minkowski boundary content,
  and having finite $\operatorname{Per}_\gamma$ for some
  $0<\gamma\le1$ all fail.  Generic Hilbert--Schmidt estimates and qualitative
  finite-measure results may still apply, and we say so explicitly.

  \item Section~\ref{sec:analytic-kernels} replaces the Fourier interaction by
  a finite sum of separated nonlinear interactions.  A total-degree Taylor
  expansion gives rank at most
  $p\binom{N+r-1}{r}$ approximation with factorial error.  Joint
  half-oscillation, weighted least-squares polynomials, span-compressed ranks,
  Chebyshev--Bessel alternatives, and anisotropic interaction degrees sharpen
  the same finite-interaction class.  This supplies a precise non-Fourier
  extension without asserting that arbitrary integral transforms enjoy the
  same envelope.

  \item Section~\ref{sec:sharp-defect-energy} starts from the
  cross-boundary defect-trace identity of
  \cite[Proposition~1.2]{HughesIsraelMayeli}, records its one-dimensional
  normalization, and derives the frequency-dual, symmetric-difference, and
  scaled forms used here.  It also proves a strict defect-energy count, a
  high-cluster subtraction from any direct count, and higher defect-moment
  certificates.  These retain information discarded by a cutoff-only
  approximation argument.

  \item Section~\ref{sec:synthesis-strengthened-hybrid} combines the new
  entries.  The resulting bounded-window and interval-union envelopes are
  formally no larger than the earlier factorial, matrix, flat Ky--Fan, and
  rounded Schatten certificates proved in this paper.  This is an internal
  dominance theorem, not a claim of numerical superiority to every estimate
  in the regular-domain literature.
\end{enumerate}

Two finite examples clarify what the new parameters detect.  First, two
widely separated copies of an arbitrary bounded measurable component have a
partition bound independent of the separating gap.  Second, a full-span
fat-Cantor set and its enclosing interval have the same hulls but different
plunge counts at explicit parameters: the former count is zero and the latter
is nonzero.  Thus local measure is not cosmetic data, and an enclosing-interval
bound can miss a qualitative distinction.

\subsection{Basic compactness and endpoint conventions}
\label{sec:intro-compactness}

We record the elementary analytic foundation once.  It will also make clear
that boundedness of the sets is not needed for compactness.

\begin{lemma}[Finite-measure Fourier windows]
\label{lem:intro-finite-measure}
If $A,B\subset\mathbb R$ are measurable with finite measure, then
$K_{A,B,c}$ is Hilbert--Schmidt and
\begin{equation}
  \|K_{A,B,c}\|_{\mathcal S_2}^2=c|A||B|.
  \label{eq:intro-HS-mass}
\end{equation}
Consequently $S_{A,B,c}$ is a positive trace-class contraction and
$\operatorname{tr}S_{A,B,c}=c|A||B|$.
\end{lemma}

\begin{proof}
After identifying the domain with $L^2(A)$ and the range with $L^2(cB)$,
$K_{A,B,c}$ has kernel
$\mathbf1_{cB}(\xi)e^{-2\pi ix\xi}\mathbf1_A(x)$.  Its modulus is one on
$cB\times A$ and zero elsewhere.  Therefore
\[
 \|K_{A,B,c}\|_{\mathcal S_2}^2
 =\int_{cB}\int_A1\,dx\,d\xi
 =|cB|\,|A|=c|A||B|,
\]
which proves the Hilbert--Schmidt assertion.  Since $\mathcal F$ and the two
projections are contractions, $K_{A,B,c}$ is a contraction.  Hence
$S_{A,B,c}=K_{A,B,c}^*K_{A,B,c}$ is positive and bounded above by the
identity.  For a Hilbert--Schmidt operator $K$, the product $K^*K$ is
trace class, and the standard Hilbert--Schmidt identity gives
$\operatorname{tr}(K^*K)=\|K\|_{\mathcal S_2}^2$, proving the last claim.
\end{proof}

All counts in this paper are strict and all transition windows are open.  This
convention makes the endpoint treatment explicit: equality at the threshold is
excluded on both sides of every spectral--singular-value correspondence.

\section{Abstract spectral transfer and inversion of approximation errors}
\label{sec:abstract-spectral}

This section isolates the operator-theoretic mechanism used throughout the
paper.  It is independent of the Fourier transform and, in particular, makes
clear which parts of the later arguments depend only on a finite-rank
approximation estimate.

Let $\mathcal H$ and $\mathcal G$ be separable Hilbert spaces.  For a bounded
operator $A\colon\mathcal H\to\mathcal G$ and $\tau>0$, write
\begin{equation}
  n(\tau;A)
  :=\operatorname{rank}\mathbf 1_{(\tau,\infty)}(|A|),
  \qquad |A|=(A^*A)^{1/2}.
  \label{eq:abstract-singular-count}
\end{equation}
The value is allowed to be $+\infty$.  When $A$ is compact, this is the number
of singular values of $A$ that are strictly larger than $\tau$, counted with
multiplicity.  The use of a strict inequality is important: it matches the
open transition interval used below and makes all endpoint conventions
unambiguous.

For a positive contraction $0\leq S\leq I$ and
$0<\varepsilon<\tfrac12$, set
\begin{equation}
  \Lambda_\varepsilon(S)
  :=\operatorname{rank}\mathbf 1_{(\varepsilon,1-\varepsilon)}(S),
  \qquad
  t_\varepsilon:=\sqrt{\varepsilon(1-\varepsilon)},
  \qquad
  D_S:=(S-S^2)^{1/2}.
  \label{eq:abstract-transition-defect}
\end{equation}

\begin{theorem}[Exact defect transfer]
\label{thm:abstract-defect-transfer}
For every positive contraction $S$,
\begin{equation}
  \Lambda_\varepsilon(S)=n(t_\varepsilon;D_S).
  \label{eq:abstract-exact-transfer}
\end{equation}
Consequently, $S-S^2$ being compact is sufficient for
$\Lambda_\varepsilon(S)<\infty$ for every fixed
$0<\varepsilon<\tfrac12$; compactness of $S$ itself is not required.
\end{theorem}

\begin{proof}
We proceed directly through the spectral calculus.

\emph{Step 1: a scalar identity.}
For $\lambda\in[0,1]$,
\begin{align}
 \lambda(1-\lambda)-\varepsilon(1-\varepsilon)
 &= (\lambda-\varepsilon)(1-\varepsilon-\lambda).
 \label{eq:abstract-scalar-factorization}
\end{align}
Because the quadratic $\lambda\mapsto\lambda(1-\lambda)$ is positive
between its two level-set roots, \eqref{eq:abstract-scalar-factorization}
gives the equivalence
\begin{equation}
 \sqrt{\lambda(1-\lambda)}>t_\varepsilon
 \quad\Longleftrightarrow\quad
 \varepsilon<\lambda<1-\varepsilon.
 \label{eq:abstract-scalar-equivalence}
\end{equation}

\emph{Step 2: apply Borel functional calculus.}
The defect is the function
$D_S=d(S)$ with $d(\lambda)=\sqrt{\lambda(1-\lambda)}$.  Hence
\eqref{eq:abstract-scalar-equivalence} implies the equality of spectral
projections
\begin{equation}
 \mathbf 1_{(t_\varepsilon,\infty)}(D_S)
 =\mathbf 1_{(\varepsilon,1-\varepsilon)}(S).
 \label{eq:abstract-projection-identity}
\end{equation}
Taking ranks proves \eqref{eq:abstract-exact-transfer}.

\emph{Step 3: compactness.}
If $S-S^2$ is compact, then its positive square root $D_S$ is compact.
Every spectral projection of a compact positive operator corresponding to an
interval bounded away from zero has finite rank.  Since
$t_\varepsilon>0$, the right-hand side of
\eqref{eq:abstract-exact-transfer} is therefore finite.
\end{proof}

The preceding hypothesis is genuinely weaker than compactness of $S$.  For
example, on $\mathcal H_0\oplus\ell^2(\mathbb N)$ the operator
\begin{equation}
  S=I_{\mathcal H_0}\oplus\operatorname{diag}(e^{-1},e^{-2},\ldots)
  \label{eq:abstract-noncompact-example}
\end{equation}
is noncompact whenever $\mathcal H_0$ is infinite dimensional, whereas
$S-S^2$ is compact.  Its identity summand contributes no transition
eigenvalues, and a direct count gives
$\Lambda_\varepsilon(S)=\max\{0,\lceil\log(1/\varepsilon)\rceil-1\}$.

\begin{theorem}[Factorization through a contraction]
\label{thm:abstract-contraction-factorization}
Let $K\colon\mathcal H\to\mathcal G$ be a contraction, put $S=K^*K$, and
define
\begin{equation}
  R:=(I_{\mathcal G}-KK^*)^{1/2}K.
  \label{eq:abstract-R-definition}
\end{equation}
Then
\begin{equation}
  R^*R=S-S^2,
  \qquad |R|=D_S,
  \qquad
  \Lambda_\varepsilon(S)=n(t_\varepsilon;R).
  \label{eq:abstract-R-identities}
\end{equation}
There is also the generally weaker but often more convenient direct bound
\begin{equation}
  \Lambda_\varepsilon(S)\leq n(\sqrt\varepsilon;K).
  \label{eq:abstract-direct-count}
\end{equation}
\end{theorem}

\begin{proof}
Because $K$ is a contraction, $I_{\mathcal G}-KK^*$ is positive and its
square root in \eqref{eq:abstract-R-definition} is well defined.  Expanding
$R^*R$ and using associativity yields
\begin{align}
 R^*R
 &=K^*(I_{\mathcal G}-KK^*)K \\
 &=K^*K-K^*KK^*K \\
 &=S-S^2.
 \label{eq:abstract-R-computation}
\end{align}
Thus $|R|=(R^*R)^{1/2}=D_S$, and the exact count follows from
Theorem~\ref{thm:abstract-defect-transfer}.

For the direct estimate, the inclusion of scalar intervals
$(\varepsilon,1-\varepsilon)\subset(\varepsilon,\infty)$ gives
\begin{equation}
 \mathbf 1_{(\varepsilon,1-\varepsilon)}(S)
 \leq \mathbf 1_{(\varepsilon,\infty)}(S).
 \label{eq:abstract-projection-inclusion}
\end{equation}
Since $|K|=(K^*K)^{1/2}$, functional calculus gives the exact projection
identity
\[
 \mathbf1_{(\varepsilon,\infty)}(K^*K)
 =\mathbf1_{(\sqrt\varepsilon,\infty)}(|K|).
\]
Therefore the rank of the right-hand projection in
\eqref{eq:abstract-projection-inclusion} is $n(\sqrt\varepsilon;K)$, even if
$K$ is not compact.  Taking ranks proves
\eqref{eq:abstract-direct-count}.
\end{proof}

The direct threshold is larger because
$\sqrt\varepsilon>t_\varepsilon$ for $0<\varepsilon<\tfrac12$.
Accordingly, if the same approximation error is available for both $K$ and
$R$, the direct estimate can be sharper.  The value of the exact defect
identity is that a separate, better approximation of $R$ can exploit the
two-sided spectral condition rather than discard its upper endpoint.

\begin{lemma}[Finite-rank approximation rule]
\label{lem:abstract-finite-rank-rule}
Let $A\colon\mathcal H\to\mathcal G$ be bounded, let $r\in\mathbb N_0$,
and let $E\geq0$.  Suppose that $A_0$ has rank at most $r$ and
\begin{equation}
  \|A-A_0\|\leq E.
  \label{eq:abstract-approximation-hypothesis}
\end{equation}
Then, for every $\tau>0$ with $\tau\geq E$,
\begin{equation}
  n(\tau;A)\leq r.
  \label{eq:abstract-finite-rank-conclusion}
\end{equation}
In particular, the $(r+1)$st approximation number, and hence the
$(r+1)$st singular value when $A$ is compact, is at most $E$.
\end{lemma}

\begin{proof}
Assume for contradiction that the spectral subspace
\begin{equation}
  \mathcal M:=\operatorname{Ran}\mathbf 1_{(\tau,\infty)}(|A|)
\end{equation}
has dimension at least $r+1$.  Since $A_0\vert_{\mathcal M}$ has range of
dimension at most $r$, its kernel contains a unit vector $f\in\mathcal M$.
For such a vector, the spectral theorem and the strict interval
$(\tau,\infty)$ give
\begin{equation}
  \|Af\|^2=\langle |A|^2f,f\rangle>\tau^2\|f\|^2=\tau^2.
\end{equation}
On the other hand, $A_0f=0$, and therefore
\begin{equation}
  \|Af\|=\|(A-A_0)f\|\leq E\leq\tau,
\end{equation}
a contradiction.  Thus $\dim\mathcal M\leq r$, which is exactly
\eqref{eq:abstract-finite-rank-conclusion}.
\end{proof}

Combining the last two results gives two reusable routes.  If
$\operatorname{rank}K_N\leq r_N$ and $\|K-K_N\|\leq E_N$, then
\begin{equation}
 E_N\leq\sqrt\varepsilon
 \quad\Longrightarrow\quad
 \Lambda_\varepsilon(K^*K)\leq r_N.
 \label{eq:abstract-direct-approx-route}
\end{equation}
Moreover, with
$R_N=(I-KK^*)^{1/2}K_N$, one has
$\operatorname{rank}R_N\leq r_N$ and
$\|R-R_N\|\leq E_N$; hence
\begin{equation}
 E_N\leq t_\varepsilon
 \quad\Longrightarrow\quad
 \Lambda_\varepsilon(K^*K)\leq r_N.
 \label{eq:abstract-defect-approx-route}
\end{equation}
The second statement follows because the left factor
$(I-KK^*)^{1/2}$ is a contraction.

\subsection{Residual-energy counting and weighted defect residuals}
\label{sec:sharp-residual-energy}

The operator-norm finite-rank rule uses only the largest residual singular
value.  When the residual is Hilbert--Schmidt, retaining its entire singular
value energy gives a strictly endpoint-correct refinement.  For $r\geq0$ set
\begin{equation}
 \kappa(r):=\max\{0,\lceil r\rceil-1\}.
 \label{eq:sharp-kappa-definition}
\end{equation}
Thus $\kappa(r)$ is the largest nonnegative integer strictly smaller than
$r$ when $r>0$; in particular, $\kappa(L)=L-1$ at every positive integer
$L$.  This one-unit endpoint correction is forced by the strict spectral
counts used in this paper.

\begin{lemma}[Hilbert--Schmidt residual-energy count]
\label{lem:sharp-residual-energy}
Let $T\colon\mathcal H\to\mathcal G$ be bounded, let $T_N$ have rank at most
$N\in\mathbb N_0$, and suppose that
\[
 E_N:=T-T_N\in\mathcal S_2.
\]
Then, for every $\tau>0$,
\begin{equation}
 n(\tau;T)
 \leq
 N+\kappa\!\left(\frac{\|E_N\|_{\mathcal S_2}^2}{\tau^2}\right).
 \label{eq:sharp-residual-energy-count}
\end{equation}
\end{lemma}

\begin{proof}
We give the dimension argument in detail, including the strict endpoint.

\emph{Step 1: isolate the spectral subspace above $\tau$.}
Since $T=T_N+E_N$ is the sum of a finite-rank operator and a compact
operator, $T$ is compact.  Therefore
\[
 \mathcal M:=\operatorname{Ran}\one_{(\tau,\infty)}(|T|)
\]
has finite dimension $q=n(\tau;T)$.  If $q\leq N$, the asserted estimate is
immediate, so suppose $q>N$.

\emph{Step 2: remove the rank-$N$ part.}
The restriction $T_N|_{\mathcal M}$ has rank at most $N$.  Rank--nullity on
the finite-dimensional space $\mathcal M$ gives
\[
 \dim\bigl(\mathcal M\cap\ker T_N\bigr)
 =\dim\ker(T_N|_{\mathcal M})
 \geq q-N.
\]
Choose orthonormal vectors
$f_1,\ldots,f_{q-N}\in\mathcal M\cap\ker T_N$.

\emph{Step 3: use the open spectral endpoint.}
For every nonzero $f\in\mathcal M$, the spectral measure $\mu_f$ of $|T|$
is supported in $(\tau,\infty)$.  Consequently,
\[
 \|Tf\|^2
 =\langle |T|^2f,f\rangle
 =\int_{(\tau,\infty)}s^2\,d\mu_f(s)
 >\tau^2\int_{(\tau,\infty)}d\mu_f(s)
 =\tau^2\|f\|^2.
\]
The inequality is strict because $s^2-\tau^2$ is strictly positive
throughout the supporting interval.  Since $T_Nf_j=0$, we have
$E_Nf_j=Tf_j$, and hence
\begin{equation}
 (q-N)\tau^2
 <\sum_{j=1}^{q-N}\|E_Nf_j\|^2
 \leq\|E_N\|_{\mathcal S_2}^2.
 \label{eq:sharp-residual-strict-energy}
\end{equation}
The last inequality follows by extending the orthonormal family to an
orthonormal basis of $\mathcal H$ and using the definition of the
Hilbert--Schmidt norm.

\emph{Step 4: convert the strict real inequality to an integer bound.}
Put $r=\|E_N\|_{\mathcal S_2}^2/\tau^2$.  Equation
\eqref{eq:sharp-residual-strict-energy} says that the positive integer
$q-N$ is strictly smaller than $r$.  If $r$ is a positive integer, this
gives $q-N\leq r-1$; if $r$ is not an integer, it gives
$q-N\leq\lfloor r\rfloor=\lceil r\rceil-1$.  The case $r=0$ cannot occur
under the standing assumption $q>N$.  In all cases
$q-N\leq\kappa(r)$, which proves
\eqref{eq:sharp-residual-energy-count}.
\end{proof}

When a Hilbert--Schmidt envelope is used also as an operator-norm envelope,
the lemma refines the rule that stops only after that envelope falls below
the threshold.  It still gives information before that happens, charging
only the number of additional singular values that the total squared
residual energy can support.  A separately available, sharper operator-norm
estimate remains complementary and should be retained.

\begin{theorem}[Direct and weighted-defect residual master bounds]
\label{thm:sharp-residual-master}
Let $K\colon\mathcal H\to\mathcal G$ be a contraction and put
\begin{equation}
 S:=K^*K,
 \qquad
 W:=(I_{\mathcal G}-KK^*)^{1/2},
 \qquad
 R:=WK.
 \label{eq:sharp-W-R-definition}
\end{equation}
Fix $0<\epsilon<1/2$.  Suppose that $K_N$ has rank at most
$N\in\mathbb N_0$ and that $K-K_N\in\mathcal S_2$.  Define the unweighted
and weighted residual energies
\begin{equation}
 e_N^2:=\|K-K_N\|_{\mathcal S_2}^2,
 \qquad
 d_N^2:=\|W(K-K_N)\|_{\mathcal S_2}^2.
 \label{eq:sharp-two-residual-energies}
\end{equation}
Then
\begin{align}
 Q_\epsilon(S)
 &:=\operatorname{rank}\one_{(\epsilon,\infty)}(S)
 \leq N+\kappa\!\left(\frac{e_N^2}{\epsilon}\right),
 \label{eq:sharp-direct-Q-bound}\\
 \Lambda_\epsilon(S)
 &\leq
 \min\left\{
 N+\kappa\!\left(\frac{e_N^2}{\epsilon}\right),
 N+\kappa\!\left(
       \frac{d_N^2}{\epsilon(1-\epsilon)}\right)
 \right\}.
 \label{eq:sharp-direct-defect-master}
\end{align}
Moreover,
\begin{equation}
 d_N^2
 =\operatorname{Tr}\!\left((K-K_N)^*(I-KK^*)(K-K_N)\right)
 \leq e_N^2.
 \label{eq:sharp-weighted-energy-trace}
\end{equation}
The inequalities may be minimized over every available rank and every
available approximant $K_N$.
\end{theorem}

\begin{proof}
\emph{Step 1: the direct spectral channel.}
Functional calculus for $S=K^*K=|K|^2$ gives
\[
 \one_{(\epsilon,\infty)}(S)
 =\one_{(\sqrt\epsilon,\infty)}(|K|).
\]
Consequently,
\[
 Q_\epsilon(S)=n(\sqrt\epsilon;K).
\]
Apply Lemma~\ref{lem:sharp-residual-energy} to $T=K$, $T_N=K_N$, and
$\tau=\sqrt\epsilon$.  This proves
\eqref{eq:sharp-direct-Q-bound}.  Since
$(\epsilon,1-\epsilon)\subset(\epsilon,\infty)$, it also gives the first
entry on the right of \eqref{eq:sharp-direct-defect-master}.

\emph{Step 2: the exact defect factorization.}
Because $K$ is a contraction, $I-KK^*$ is positive, so $W$ is well defined
and $\|W\|\leq1$.  Direct multiplication gives
\begin{align*}
 R^*R
 &=K^*W^2K
 =K^*(I-KK^*)K\\
 &=K^*K-K^*KK^*K
 =S-S^2.
\end{align*}
For a scalar $\lambda\in[0,1]$,
\[
 \lambda(1-\lambda)-\epsilon(1-\epsilon)
 =(\lambda-\epsilon)(1-\epsilon-\lambda).
\]
It follows, with both endpoints audited, that
\[
 \lambda\in(\epsilon,1-\epsilon)
 \quad\Longleftrightarrow\quad
 \sqrt{\lambda(1-\lambda)}
   >\sqrt{\epsilon(1-\epsilon)}.
\]
Applying this scalar equivalence to $S$ and using
$|R|=(R^*R)^{1/2}=(S-S^2)^{1/2}$ yields the exact identity
\begin{equation}
 \Lambda_\epsilon(S)
 =n\!\left(\sqrt{\epsilon(1-\epsilon)};R\right).
 \label{eq:sharp-exact-defect-count}
\end{equation}

\emph{Step 3: approximate the defect operator with the weighted residual.}
Set $R_N:=WK_N$.  Then
\[
 \operatorname{rank}R_N\leq\operatorname{rank}K_N\leq N,
 \qquad
 R-R_N=W(K-K_N).
\]
The ideal property of the Hilbert--Schmidt class shows that the latter
operator is Hilbert--Schmidt.  Lemma~\ref{lem:sharp-residual-energy}, now
with threshold $\sqrt{\epsilon(1-\epsilon)}$, gives
\[
 n\!\left(\sqrt{\epsilon(1-\epsilon)};R\right)
 \leq
 N+\kappa\!\left(
       \frac{d_N^2}{\epsilon(1-\epsilon)}\right).
\]
Together with \eqref{eq:sharp-exact-defect-count}, this proves the second
entry in \eqref{eq:sharp-direct-defect-master}.

\emph{Step 4: compare the two residual energies without discarding the
weight.}
Since $W^2=I-KK^*$,
\begin{align*}
 d_N^2
 &=\operatorname{Tr}\bigl((W(K-K_N))^*W(K-K_N)\bigr)\\
 &=\operatorname{Tr}\bigl((K-K_N)^*(I-KK^*)(K-K_N)\bigr).
\end{align*}
Also $0\leq W^2\leq I$, whence
$d_N^2\leq e_N^2$.  Notice that the defect threshold has the smaller square
$\epsilon(1-\epsilon)$.  Therefore $d_N^2\leq e_N^2$ alone does not order
the two entries in \eqref{eq:sharp-direct-defect-master}; the defect entry
improves the direct one only when the weight $W$ removes enough residual
energy.  This is exactly why both certified bounds are retained.
\end{proof}

If $K$ is Hilbert--Schmidt, the zero approximant is always admissible.  In
that case, writing
\begin{equation}
 m:=\operatorname{Tr}S=\|K\|_{\mathcal S_2}^2,
 \qquad
 \mathcal D:=\operatorname{Tr}(S-S^2)=\|WK\|_{\mathcal S_2}^2,
 \label{eq:sharp-m-D-definition}
\end{equation}
Theorem~\ref{thm:sharp-residual-master} at $N=0$ gives
\begin{equation}
 Q_\epsilon(S)\leq\kappa\!\left(\frac m\epsilon\right),
 \qquad
 \Lambda_\epsilon(S)
 \leq\min\left\{
 \kappa\!\left(\frac m\epsilon\right),
 \kappa\!\left(\frac{\mathcal D}{\epsilon(1-\epsilon)}\right)
 \right\}.
 \label{eq:sharp-zero-rank-two-counts}
\end{equation}
The second entry retains the exact defect energy rather than bounding it by
$m$.

\subsection{Factorial envelopes and the Lambert--\texorpdfstring{$W$}{W} function}
\label{sec:abstract-factorial}

The next proposition performs the discrete inversion that will recur in the
finite-interaction kernel estimates.  The minimum below is the exact order
certified
by the factorial error envelope; it need not be the actual singular-value
count.

\begin{proposition}[Exact factorial certificate and explicit surrogate]
\label{prop:abstract-factorial-inversion}
Let $A\colon\mathcal H\to\mathcal G$ be bounded.  Assume that, for every
$N\in\mathbb N_0$, there is an operator $A_N$ of rank at most $N$ such that
\begin{equation}
  \|A-A_N\|\leq E_N,
  \qquad
  E_N=C\frac{\zeta^N}{N!},
  \label{eq:abstract-factorial-envelope}
\end{equation}
where $C>0$ and $\zeta\geq0$.  At $N=0$ we use the convention
$\zeta^0=1$, so $E_0=C$.  For $\tau>0$, define
\begin{equation}
 \mathcal N_{\mathrm{ex}}(C,\zeta;\tau)
 :=\min\left\{N\in\mathbb N_0:
 C\frac{\zeta^N}{N!}\leq\tau\right\}.
 \label{eq:abstract-N-exact}
\end{equation}
Then
\begin{equation}
  n(\tau;A)\leq \mathcal N_{\mathrm{ex}}(C,\zeta;\tau).
  \label{eq:abstract-factorial-count}
\end{equation}
More explicitly:
\begin{enumerate}
\item If $C\leq\tau$, then $\mathcal N_{\mathrm{ex}}=0$.
\item If $C>\tau$ and $\zeta=0$, then
      $\mathcal N_{\mathrm{ex}}=1$.
\item If $C>\tau$ and $\zeta>0$, put
\begin{equation}
  \rho:=\log\frac C\tau,
  \qquad
  \mathcal N_W(C,\zeta;\tau)
  :=\left\lceil
  \frac{\rho}{W_0\!\left(\rho/(e\zeta)\right)}
  \right\rceil,
  \label{eq:abstract-N-W}
\end{equation}
where $W_0$ is the principal real branch of the Lambert function.  Then
\begin{equation}
  \mathcal N_{\mathrm{ex}}(C,\zeta;\tau)
  \leq \mathcal N_W(C,\zeta;\tau),
  \qquad
  n(\tau;A)\leq\mathcal N_W(C,\zeta;\tau).
  \label{eq:abstract-W-conclusion}
\end{equation}
\end{enumerate}
\end{proposition}

\begin{proof}
We verify the certificate and then invert it.

\emph{Step 1: use the finite-rank rule.}
For every $N$ satisfying $E_N\leq\tau$,
Lemma~\ref{lem:abstract-finite-rank-rule}, applied to $A_N$, gives
$n(\tau;A)\leq N$.  The factorial sequence tends to zero: when
$\zeta>0$ its consecutive ratio is
$E_{N+1}/E_N=\zeta/(N+1)$, and when $\zeta=0$ all terms with $N\geq1$
vanish.  Thus the set in \eqref{eq:abstract-N-exact} is nonempty, and
choosing its least element proves \eqref{eq:abstract-factorial-count}.

\emph{Step 2: settle the two endpoint cases.}
If $C\leq\tau$, then $E_0=C\leq\tau$, so the minimum is zero.  If
$C>\tau$ and $\zeta=0$, then $E_0=C>\tau$ whereas $E_1=0\leq\tau$;
hence the minimum is one.  Separating the latter case is necessary because
the expression in \eqref{eq:abstract-N-W} contains division by $\zeta$.

\emph{Step 3: solve the continuous comparison equation.}
Assume now that $C>\tau$ and $\zeta>0$, so $\rho>0$.  Write
\begin{equation}
  w:=W_0\!\left(\frac{\rho}{e\zeta}\right)>0,
  \qquad
  N_*:=\frac{\rho}{w}.
  \label{eq:abstract-N-star}
\end{equation}
The defining identity $we^w=\rho/(e\zeta)$ gives
\begin{equation}
  N_*=e\zeta e^w,
  \qquad
  \log\frac{N_*}{e\zeta}=w,
  \qquad
  N_*\log\frac{N_*}{e\zeta}=\rho.
  \label{eq:abstract-N-star-identity}
\end{equation}
In particular $N_*>e\zeta$.

\emph{Step 4: pass to an integer without losing the inequality.}
Let $N=\lceil N_*\rceil$.  The function
\begin{equation}
  g(x):=x\log\frac{x}{e\zeta}
\end{equation}
has derivative
$g'(x)=\log(x/(e\zeta))+1>0$ for $x\geq N_*>e\zeta$.
Consequently,
\begin{equation}
  N\log\frac{N}{e\zeta}=g(N)\geq g(N_*)=\rho.
  \label{eq:abstract-rounded-comparison}
\end{equation}

\emph{Step 5: invoke the elementary factorial lower bound.}
For every integer $N\geq1$,
\begin{equation}
 \log N!=\sum_{k=1}^N\log k
 \geq\int_1^N\log x\,dx
 =N\log N-N+1
 \geq N\log N-N.
 \label{eq:abstract-factorial-lower-proof}
\end{equation}
Thus $N!\geq(N/e)^N$.  Combining this inequality with
\eqref{eq:abstract-rounded-comparison} gives
\begin{align}
 C\frac{\zeta^N}{N!}
 &\leq C\left(\frac{e\zeta}{N}\right)^N \\
 &=C\exp\left(-N\log\frac{N}{e\zeta}\right) \\
 &\leq Ce^{-\rho}=\tau.
 \label{eq:abstract-W-final-estimate}
\end{align}
Therefore $N$ is admissible in the exact minimum
\eqref{eq:abstract-N-exact}.  This proves
$\mathcal N_{\mathrm{ex}}\leq N=\mathcal N_W$ and completes the proof.
\end{proof}

For a transition count, Proposition~\ref{prop:abstract-factorial-inversion}
is used with $(A,\tau)=(K,\sqrt\varepsilon)$ through
\eqref{eq:abstract-direct-approx-route}, or with
$(A,\tau)=(R,t_\varepsilon)$ through the exact defect route
\eqref{eq:abstract-defect-approx-route}.  In either application the discrete
factorial minimum is at least as sharp as its Lambert--$W$ surrogate; the
latter is valuable because it is closed form.

\begin{corollary}[Other approximation laws]
\label{cor:abstract-other-laws}
Let $A\colon\mathcal H\to\mathcal G$ be bounded.  Suppose that for every
$N\in\mathbb N_0$ there is an $A_N$ of rank at most $N$.  Let $C>0$,
$\tau>0$, and put $\rho_+:=\max\{0,\log(C/\tau)\}$.  If one of the
following error estimates holds for every $N\in\mathbb N_0$, then the
corresponding singular-value count follows:
\begin{align}
 \|A-A_N\|\leq Ce^{-aN},\ a>0
 &\quad\Longrightarrow\quad
 n(\tau;A)\leq\left\lceil\frac{\rho_+}{a}\right\rceil,
 \label{eq:abstract-exponential-law}\\
 \|A-A_N\|\leq Ce^{-aN^q},\ a,q>0
 &\quad\Longrightarrow\quad
 n(\tau;A)\leq
 \left\lceil\left(\frac{\rho_+}{a}\right)^{1/q}\right\rceil,
 \label{eq:abstract-stretched-law}\\
 \|A-A_N\|\leq C(N+1)^{-a},\ a>0
 &\quad\Longrightarrow\quad
 n(\tau;A)\leq
 \left\lceil\max\left\{0,
 \left(\frac C\tau\right)^{1/a}-1\right\}\right\rceil.
 \label{eq:abstract-algebraic-law}
\end{align}
If the approximants have rank at most $r_N$ instead, define
\[
 N_{\rm exp}=\left\lceil\frac{\rho_+}{a}\right\rceil,
 \quad
 N_{\rm str}=\left\lceil\left(\frac{\rho_+}{a}\right)^{1/q}\right\rceil,
 \quad
 N_{\rm alg}=\left\lceil\max\left\{0,
       \left(\frac C\tau\right)^{1/a}-1\right\}\right\rceil.
\]
The corresponding conclusions are, respectively,
$n(\tau;A)\le r_{N_{\rm exp}}$,
$n(\tau;A)\le r_{N_{\rm str}}$, and
$n(\tau;A)\le r_{N_{\rm alg}}$; no monotonicity of $r_N$ is required.
\end{corollary}

\begin{proof}
If $C\leq\tau$, the rank-zero approximant at $N=0$ already suffices in all
three cases.  Assume $C>\tau$, so $\rho_+=\log(C/\tau)>0$.

For \eqref{eq:abstract-exponential-law}, the choice
$N=\lceil\rho_+/a\rceil$ gives $aN\geq\rho_+$ and hence
$Ce^{-aN}\leq Ce^{-\rho_+}=\tau$.  For
\eqref{eq:abstract-stretched-law}, choosing
$N=\lceil(\rho_+/a)^{1/q}\rceil$ gives $aN^q\geq\rho_+$ and the same
conclusion.  Finally,
$N=\lceil(C/\tau)^{1/a}-1\rceil$ satisfies
$N+1\geq(C/\tau)^{1/a}$, so $C(N+1)^{-a}\leq\tau$.
In each case Lemma~\ref{lem:abstract-finite-rank-rule} converts the error
inequality into the asserted count.  The version with rank $r_N$ is the
same argument with that rank inserted in
\eqref{eq:abstract-finite-rank-conclusion}.
\end{proof}

\section{The componentwise variational Schatten engine}
\label{sec:componentwise-schatten}

This section develops the all-parameter estimate for finite unions of intervals.
The only non-elementary ingredient is the one-sided Schatten estimate from
\cite[Proposition~5.1]{AzimifardIndependent}, restated as
Proposition~\ref{prop:sch-published-block}.  Every subsequent reduction and
optimization is proved below.

Throughout this section,
\[
 A=\bigcup_{i=1}^{M}I_i,
 \qquad
 B=\bigcup_{j=1}^{K}J_j,
\]
where the intervals within each union are pairwise disjoint, bounded, and of
positive length.  We retain the Fourier normalization
\[
 (\F f)(\xi)=\int_{\R}e^{-2\pi i x\xi}f(x)\,dx,
 \qquad Q_E=\F^{-1}P_E\F .
\]
For \(c>0\), set
\[
 S=P_AQ_{cB}P_A,
 \qquad
 T=P_{A^c}Q_{cB}P_A,
 \qquad
 t_\eps=\sqrt{\eps(1-\eps)}.
\]
By Lemma~\ref{lem:intro-finite-measure}, \(T\) is Hilbert--Schmidt, and the
abstract defect identity gives
\begin{equation}
 \Lpl_\eps(A,cB)
 =\operatorname{Tr}\one_{(\eps,1-\eps)}(S)
 =n(t_\eps;T),
 \qquad 0<\eps<\tfrac12 .
 \label{eq:sch-defect-count}
\end{equation}

For \(0<p<\infty\), write
\(\|X\|_p^p:=\sum_{k\geq1}s_k(X)^p\), possibly \(+\infty\), and write
\(X\in\Sp_p\) when the sum is finite.  We shall use two standard
singular-value facts.  If \(X,Y\) are compact and \(m,n\) are nonnegative
integers, then
\begin{equation}
 s_{m+n+1}(X+Y)\leq s_{m+1}(X)+s_{n+1}(Y).
 \label{eq:sch-kyfan}
\end{equation}
If \(0<p\leq1\) and \(X,Y\in\Sp_p\), Rotfel'd's inequality gives
\begin{equation}
 \|X+Y\|_p^p\leq\|X\|_p^p+\|Y\|_p^p.
 \label{eq:sch-rotfeld}
\end{equation}
Both statements follow first for finite matrices and then for compact operators
by approximation; see, for example, \cite{BhatiaMatrix,SimonTrace}.

\subsection{The one-sided input and its dilation invariance}

For \(\ell,b>0\), define an operator between two explicitly different Hilbert
spaces by
\begin{equation}
 \Gamma_{\ell,b}:L^2(0,\ell)\longrightarrow L^2(0,\infty),
 \qquad
 (\Gamma_{\ell,b}f)(x)
 =\int_0^\ell
 \frac{\sin\!\bigl(\pi b(x+y)\bigr)}{\pi(x+y)}f(y)\,dy .
 \label{eq:sch-gamma}
\end{equation}
Translations, a reflection of the input interval, and unimodular
multiplications identify this operator with either one-sided tail of an interval
localization block.  The quantitative input is the following previously proved
estimate.

\begin{proposition}[Imported one-sided block estimate]
\label{prop:sch-published-block}
For \(\ell,b>0\) and \(0<p\leq1\),
\begin{equation}
 \|\Gamma_{\ell,b}\|_p^p
 \leq \frac1p\left[
    \pi e\,b+8+
    \Csharp\bigl(1+\logplus(\ell p)\bigr)
 \right],
 \qquad
 \Csharp=3+\frac1{\log2},
 \label{eq:sch-published-block}
\end{equation}
where \(\logplus u=\max\{0,\log_2u\}\).
\end{proposition}

Proposition~\ref{prop:sch-published-block} is exactly
\cite[Proposition~5.1]{AzimifardIndependent}, with the domain and codomain
written explicitly; \cite[Lemma~4.2]{AzimifardIndependent} defines
\(\Csharp=3+1/\log2<4.5\).  Its kernel and coefficient \(\pi e\) agree with
the Fourier convention used above.

\begin{lemma}[Half-line normalization]
\label{lem:sch-halfline-normalization}
For every \(\ell,b>0\), the operators
\(\Gamma_{\ell,b}\) and \(\Gamma_{b\ell,1}\) have identical singular values.
\end{lemma}

\begin{proof}
The proof uses one unitary on the input space and another on the output space.
Define
\[
 U_b^{\mathrm{in}}:L^2(0,\ell)\longrightarrow L^2(0,b\ell),
 \qquad
 (U_b^{\mathrm{in}}f)(u)=b^{-1/2}f(u/b),
\]
and
\[
 U_b^{\mathrm{out}}:L^2(0,\infty)\longrightarrow L^2(0,\infty),
 \qquad
 (U_b^{\mathrm{out}}g)(u)=b^{-1/2}g(u/b).
\]
For example,
\[
 \|U_b^{\mathrm{in}}f\|_2^2
 =\int_0^{b\ell}b^{-1}|f(u/b)|^2\,du
 =\int_0^\ell|f(y)|^2\,dy,
\]
and the same change of variables proves that \(U_b^{\mathrm{out}}\) is
unitary.

Let \(h\in L^2(0,b\ell)\).  Since
\((U_b^{\mathrm{in}})^{-1}h(y)=b^{1/2}h(by)\), direct substitution gives
\begin{align*}
 &\bigl(U_b^{\mathrm{out}}\Gamma_{\ell,b}
       (U_b^{\mathrm{in}})^{-1}h\bigr)(u)\\
 &\quad=b^{-1/2}\int_0^\ell
 \frac{\sin\!\bigl(\pi b(u/b+y)\bigr)}
      {\pi(u/b+y)}\,b^{1/2}h(by)\,dy\\
 &\quad=\int_0^\ell
 b\frac{\sin\!\bigl(\pi(u+by)\bigr)}{\pi(u+by)}h(by)\,dy\\
 &\quad=\int_0^{b\ell}
 \frac{\sin\!\bigl(\pi(u+v)\bigr)}{\pi(u+v)}h(v)\,dv,
\end{align*}
where the last line uses \(v=by\).  This is precisely
\((\Gamma_{b\ell,1}h)(u)\).  Hence
\[
 U_b^{\mathrm{out}}\Gamma_{\ell,b}
 (U_b^{\mathrm{in}})^{-1}=\Gamma_{b\ell,1}.
\]
Left and right multiplication by unitaries leaves every singular value
unchanged, proving the claim.
\end{proof}

\begin{lemma}[Free component dilation]
\label{lem:sch-free-dilation}
Let \(\ell,b,\beta>0\) and put \(w=\ell b\).  Then
\[
 \Gamma_{\ell,b}\sim\Gamma_{w/\beta,\beta},
\]
where \(\sim\) means equality of all singular values.
\end{lemma}

\begin{proof}
Lemma~\ref{lem:sch-halfline-normalization} first gives
\(\Gamma_{\ell,b}\sim\Gamma_{\ell b,1}=\Gamma_{w,1}\).  Applying the same
lemma with \((\ell,b)=(w/\beta,\beta)\) gives
\(\Gamma_{w/\beta,\beta}\sim\Gamma_{w,1}\).  Symmetry and transitivity of
unitary equivalence complete the proof.
\end{proof}

The preceding lemma supplies the variational parameter: within this family the
singular values depend on \((\ell,b)\) only through the product \(w=\ell b\).
Each component pair may therefore select its own \(\beta\).

\subsection{Canonical normalization of an interval pair}

The same calculation also records why a single interval pair depends on only
one dimensionless product.

\begin{lemma}[Canonical one-interval localization]
\label{lem:sch-one-interval}
Let \(I,J\subset\R\) be bounded intervals of positive length and let
\(c>0\).  If
\(w=c|I||J|\), then
\[
 P_IQ_{cJ}P_I
 \simeq
 P_{(0,w)}Q_{(-1/2,1/2)}P_{(0,w)}.
\]
Consequently every unitarily invariant spectral count of the one-interval
operator depends on \(I,J,c\) only through \(w\).
\end{lemma}

\begin{proof}
Write, modulo endpoints and null sets,
\[
 I=x_I+(-|I|/2,|I|/2),
 \qquad
 cJ=c\xi_J+(-c|J|/2,c|J|/2),
\]
where \(x_I\) and \(\xi_J\) are the centers of \(I\) and \(J\).  Introduce
the unitaries
\[
 (T_af)(x)=f(x-a),
 \qquad
 (M_\eta f)(x)=e^{2\pi i\eta x}f(x),
 \qquad
 (D_\gamma f)(x)=\gamma^{-1/2}f(x/\gamma).
\]

First, translation sends the centered copy of \(I\) to \(I\).  Modulation by
\(c\xi_J\) sends the centered frequency interval to \(cJ\).  The extra Fourier
phase caused by spatial translation commutes with the frequency projection.
Therefore, with \(U_0=M_{c\xi_J}T_{x_I}\),
\[
 U_0^{-1}(P_IQ_{cJ}P_I)U_0
 =P_{(-|I|/2,|I|/2)}
  Q_{(-c|J|/2,c|J|/2)}
  P_{(-|I|/2,|I|/2)}.
\]

Second, direct calculation gives
\[
 D_\gamma P_ED_\gamma^{-1}=P_{\gamma E},
 \qquad
 \F D_\gamma f(\xi)=\gamma^{1/2}(\F f)(\gamma\xi),
\]
and hence
\[
 D_\gamma Q_ED_\gamma^{-1}=Q_{E/\gamma}.
\]
Choose \(\gamma=c|J|\).  Since \(\gamma|I|=w\), conjugation by \(D_\gamma\)
turns the preceding centered operator into
\[
 P_{(-w/2,w/2)}Q_{(-1/2,1/2)}P_{(-w/2,w/2)}.
\]
Finally, conjugation by \(T_{w/2}\) sends the spatial interval
\((-w/2,w/2)\) to \((0,w)\) and leaves the centered frequency projection
unchanged.  Composing the three unitary conjugations proves the assertion.
\end{proof}

\subsection{Exact optimization of the block scale}

\begin{lemma}[Two-branch optimized envelope]
\label{lem:sch-Gstar}
For \(u>0\), define
\[
 G_*(u)=\inf_{\beta>0}
 \left\{\pi e\,\beta+
 \Csharp\logplus\!\left(\frac{u}{\beta}\right)\right\},
 \qquad
 \beta_0=\frac{\Csharp}{\pi e\log2}.
\]
Then
\begin{equation}
 G_*(u)=
 \begin{cases}
  \pi e\,u,
     &0<u\leq\beta_0,\\[3pt]
  \displaystyle\frac{\Csharp}{\log2}
     +\Csharp\log_2\!\left(\frac{u}{\beta_0}\right),
     &u\geq\beta_0.
 \end{cases}
 \label{eq:sch-Gstar-formula}
\end{equation}
\end{lemma}

\begin{proof}
Put
\[
 F_u(\beta)=\pi e\,\beta+
 \Csharp\logplus(u/\beta).
\]
We minimize on the two regions separated by \(\beta=u\).

If \(\beta\geq u\), then \(u/\beta\leq1\), so
\(\logplus(u/\beta)=0\).  Thus \(F_u(\beta)=\pi e\beta\), and its minimum on
\([u,\infty)\) occurs at \(\beta=u\), with value \(\pi eu\).

If \(0<\beta<u\), then
\[
 F_u(\beta)=\pi e\beta+\Csharp\log_2(u/\beta),
\]
so
\[
 F_u'(\beta)=\pi e-\frac{\Csharp}{\beta\log2},
 \qquad
 F_u''(\beta)=\frac{\Csharp}{\beta^2\log2}>0.
\]
There is exactly one critical point, namely \(\beta_0\), and strict convexity
makes it the unique minimum whenever it lies in the region \(0<\beta<u\).

If \(u\leq\beta_0\), then \(F_u'(\beta)\leq0\) throughout \((0,u)\).
The infimum on that interval is therefore its boundary value at \(\beta=u\),
which agrees with the first-region minimum and equals \(\pi eu\).

If \(u\geq\beta_0\), the critical point is admissible (with equality at the
common boundary when \(u=\beta_0\)).  For \(u>\beta_0\), the function decreases
on \((0,\beta_0)\), increases on \((\beta_0,u)\), and remains increasing on
\([u,\infty)\); it is continuous at \(\beta=u\).  Thus \(\beta_0\) is the
global minimizer.  Since \(\pi e\beta_0=\Csharp/\log2\), evaluation at
\(\beta_0\) yields the second line of \eqref{eq:sch-Gstar-formula}.  At
\(u=\beta_0\), both formulas equal \(\Csharp/\log2\); hence there is no gap
between the cases.
\end{proof}

\subsection{The componentwise finite-union theorem}

\begin{theorem}[Componentwise variational Schatten envelope]
\label{thm:sch-componentwise}
Let \(A=\bigcup_{i=1}^{M}I_i\) and
\(B=\bigcup_{j=1}^{K}J_j\) be finite pairwise disjoint unions of bounded
intervals of positive length.  Put
\[
 w_{ij}=c|I_i||J_j|,
 \qquad t_\eps=\sqrt{\eps(1-\eps)}.
\]
Then, for \(0<\eps<1/2\),
\begin{equation}
 \boxed{
 \Lpl_\eps(A,cB)
 \leq
 2\inf_{0<p\leq1}\frac{t_\eps^{-p}}p
 \sum_{i=1}^{M}\sum_{j=1}^{K}
 \left[8+\Csharp+G_*(w_{ij}p)\right].}
 \label{eq:sch-componentwise}
\end{equation}
Each occurrence of \(G_*(w_{ij}p)\) is optimized with an independent block
dilation.
\end{theorem}

\begin{proof}
We give the reduction in five steps.

\emph{Step 1: decompose the global defect block.}
Because the \(J_j\) are disjoint,
\(Q_{cB}=\sum_{j=1}^{K}Q_{cJ_j}\).  Likewise
\(P_A=\sum_{i=1}^{M}P_{I_i}\).  The sums are finite, and hence
\begin{equation}
 T=P_{A^c}Q_{cB}P_A
   =\sum_{i=1}^{M}\sum_{j=1}^{K}T_{ij},
 \qquad
 T_{ij}=P_{A^c}Q_{cJ_j}P_{I_i}.
 \label{eq:sch-block-sum}
\end{equation}

\emph{Step 2: enlarge the output of each block.}
Since \(A^c\subset I_i^c\),
\[
 T_{ij}=P_{A^c}\widetilde T_{ij},
 \qquad
 \widetilde T_{ij}=P_{I_i^c}Q_{cJ_j}P_{I_i}.
\]
The projection \(P_{A^c}\) is a contraction.  The ideal property of singular
values therefore gives
\begin{equation}
 s_k(T_{ij})\leq s_k(\widetilde T_{ij})
 \quad(k\geq1),
 \qquad
 \|T_{ij}\|_p^p\leq\|\widetilde T_{ij}\|_p^p.
 \label{eq:sch-output-compression}
\end{equation}

\emph{Step 3: identify the two one-sided tails.}
To see this without suppressing a geometric step, write
\(I_i=(a,a+\ell)\), where \(\ell=|I_i|\), and let
\(b=c|J_j|\).  After shifting the frequency interval to be centered at zero,
the kernel of \(Q_{cJ_j}\), up to a unimodular factor in each variable, is
\[
 q_b(x-y)=\frac{\sin(\pi b(x-y))}{\pi(x-y)}.
\]
On the right tail put \(r=x-a-\ell>0\) and \(s=a+\ell-y\in(0,\ell)\).
Then \(x-y=r+s\), so a reflection of the input turns this tail exactly into
\(\Gamma_{\ell,b}\).  On the left tail put \(r=a-x>0\) and
\(s=y-a\in(0,\ell)\).  Now \(x-y=-(r+s)\); since \(q_b\) is even, this is
again \(\Gamma_{\ell,b}\).  If \(L_{ij}\) and \(R_{ij}\) denote the two tail
operators, embedded by zero extension into the common codomain
$L^2(I_i^c)$, then
\(\widetilde T_{ij}=L_{ij}+R_{ij}\), and each of \(L_{ij},R_{ij}\) has the
same singular values as \(\Gamma_{\ell,b}\).  (Orthogonality of the output
half-lines alone does not justify equality of their Schatten \(p\)-powers.)
Rotfel'd's inequality gives the estimate that is needed:
\begin{equation}
 \|\widetilde T_{ij}\|_p^p
 \leq2\|\Gamma_{|I_i|,c|J_j|}\|_p^p.
 \label{eq:sch-two-tails}
\end{equation}

\emph{Step 4: optimize each component scale.}
Fix \(p\in(0,1]\).  Lemma~\ref{lem:sch-free-dilation} gives, for every
\(\beta_{ij}>0\),
\[
 \Gamma_{|I_i|,c|J_j|}
 \sim\Gamma_{w_{ij}/\beta_{ij},\beta_{ij}}.
\]
Apply Proposition~\ref{prop:sch-published-block} with
\(\ell=w_{ij}/\beta_{ij}\) and \(b=\beta_{ij}\), and then use
\eqref{eq:sch-two-tails}:
\begin{align*}
 \|\widetilde T_{ij}\|_p^p
 &\leq\frac2p\left[
   \pi e\beta_{ij}+8+
   \Csharp\left(1+
   \logplus\!\left(\frac{w_{ij}p}{\beta_{ij}}\right)\right)
 \right].
\end{align*}
The parameter \(\beta_{ij}\) occurs nowhere else.  We may therefore minimize
it separately for every pair \((i,j)\).  Lemma~\ref{lem:sch-Gstar} yields
\begin{equation}
 \|\widetilde T_{ij}\|_p^p
 \leq\frac2p\left[8+\Csharp+G_*(w_{ij}p)\right].
 \label{eq:sch-local-envelope}
\end{equation}

\emph{Step 5: assemble the blocks and count.}
Repeated application of Rotfel'd's inequality
\eqref{eq:sch-rotfeld} to the finite sum \eqref{eq:sch-block-sum}, followed by
\eqref{eq:sch-output-compression} and \eqref{eq:sch-local-envelope}, gives
\[
 \|T\|_p^p
 \leq\frac2p\sum_{i=1}^{M}\sum_{j=1}^{K}
       [8+\Csharp+G_*(w_{ij}p)].
\]
For any compact operator \(X\), strict singular-value Chebyshev counting says
\[
 n(t;X)t^p
 \leq\sum_{s_k(X)>t}s_k(X)^p
 \leq\|X\|_p^p,
\]
and hence \(n(t;X)\leq t^{-p}\|X\|_p^p\).  Apply this with
\(X=T\) and \(t=t_\eps\), then use \eqref{eq:sch-defect-count}.  This proves
the asserted right-hand side for each fixed \(p\in(0,1]\).  Taking the
infimum only after those fixed-\(p\) inequalities are established gives
\eqref{eq:sch-componentwise}.
\end{proof}

\subsection{What the variational scale improves}

It is useful to compare the optimized block function with a concrete fixed
scale, rather than making an unsupported global dominance claim.  Define
\[
 U_1(u)=\pi e+\Csharp\logplus(u).
\]
The choice \(\beta=1\) is admissible in the definition of \(G_*\), so
\begin{equation}
 G_*(u)\leq U_1(u)\qquad(u>0).
 \label{eq:sch-fixed-scale-dominance}
\end{equation}
This comparison is strict for every \(u>0\).  Indeed,
\(\beta_0\approx0.750545<1\).  If \(0<u\leq\beta_0\), then
\[
 U_1(u)-G_*(u)=\pi e(1-u)>0.
\]
If \(\beta_0<u<1\), strict convexity in the proof of
Lemma~\ref{lem:sch-Gstar} shows that the minimizer \(\beta_0\) is different
from the admissible but nonoptimal value \(1\).  Finally, if \(u\geq1\), the
gain is the positive constant
\begin{equation}
 U_1(u)-G_*(u)
 =\Delta_\sharp
 :=\pi e-\frac{\Csharp}{\log2}
   +\Csharp\log_2\beta_0
 \approx0.291050.
 \label{eq:sch-constant-gain}
\end{equation}
Thus the improvement is functional in the small-product regime: the variable
part changes from the constant \(\pi e\) to \(\pi eu\).

There is also a direct comparison with the component scale used in
\cite{AzimifardIndependent}.  In the block
\(\Gamma_{|I_i|,c|J_j|}\), select the admissible value
\(\beta_{ij}=|J_j|\).  Lemma~\ref{lem:sch-free-dilation} then gives
\(\Gamma_{c|I_i|,|J_j|}\), whose variable contribution in the cited
estimate is
\[
 \pi e|J_j|+
 \Csharp\logplus(c|I_i|p).
\]
This is the normalization $A=cA_0$, $B=B_0$ used in the cited assembly.  At
the global operator level the same scaling follows from
\[
 D_c(P_AQ_{cB}P_A)D_c^{-1}=P_{cA}Q_BP_{cA},
\]
with the dilation $D_c$ defined in
Lemma~\ref{lem:sch-one-interval}.
The independently optimized contribution is no larger, because it takes the
infimum over all \(\beta_{ij}>0\), including this value.  It also retains every
\(w_{ij}=c|I_i||J_j|\), instead of replacing component lengths by their
maxima.  Finally, the outer infimum in \(p\) includes every fixed admissible
exponent used in that construction.

These are precise comparisons of the two upper-bound formulas.  They do not
claim that \eqref{eq:sch-componentwise} has a smaller asymptotic order than, or
is numerically below, every other known theorem in every regime.  The
separated-interaction finite-rank estimates developed later in the paper cover complementary deep-tail
and rough-window regimes.

\begin{remark}[A safe explicit exponent]
\label{rem:sch-explicit-p}
For a fully explicit, nonoptimized bound one may take
\[
 p_0=\min\left\{1,\frac2L\right\},
 \qquad
 L=\log\frac1{\eps(1-\eps)}.
\]
This is always admissible.  No assertion that it globally minimizes the full
right-hand side of \eqref{eq:sch-componentwise} is needed.
\end{remark}

\section{Measurable Fourier windows, soft masks, and moments}
\label{sec:fourier-extensions}

This section develops the direct singular-value method in a form that does not
require the windows to be intervals.  The first result applies to arbitrary
bounded measurable sets.  The second replaces characteristic functions by
soft masks.  The third replaces boundedness by a moment condition and therefore
also treats unbounded windows.  Throughout,
\[
  (\mathcal Ff)(\xi)=\int_{\mathbb R}e^{-2\pi i x\xi}f(x)\,dx
\]
denotes the unitary Fourier transform on $L^2(\mathbb R)$.  Singular values are
listed as $s_1(K)\ge s_2(K)\ge\cdots$, with multiplicity.

\subsection{The exact Taylor remainder on the imaginary axis}

The absence of an exponential loss in the next lemma is important.  A generic
complex Taylor estimate would introduce an unnecessary factor such as
$e^{|t|}$.

\begin{lemma}[Imaginary-axis Taylor remainder]
\label{lem:imaginary-taylor}
For every $t\in\mathbb R$ and every integer $N\ge1$,
\begin{equation}
 \left|e^{it}-\sum_{k=0}^{N-1}\frac{(it)^k}{k!}\right|
 \le \frac{|t|^N}{N!}.
 \label{eq:imaginary-taylor}
\end{equation}
\end{lemma}

\begin{proof}
Fix $N\ge1$.  Taylor's formula with integral remainder, applied along the real
line segment from $0$ to $t$, gives
\[
 e^{it}-\sum_{k=0}^{N-1}\frac{(it)^k}{k!}
 =\frac{i^N}{(N-1)!}\int_0^t e^{is}(t-s)^{N-1}\,ds.
\]
If $t\ge0$, taking absolute values and using $|e^{is}|=1$ yields
\[
 \left|e^{it}-\sum_{k=0}^{N-1}\frac{(it)^k}{k!}\right|
 \le \frac1{(N-1)!}\int_0^t(t-s)^{N-1}\,ds
 =\frac{t^N}{N!}.
\]
For $t<0$, reverse the direction of integration, or apply the already proved
case to $-t$ and take complex conjugates.  This gives exactly
\eqref{eq:imaginary-taylor}.
\end{proof}

For $N=0$ we use the empty Taylor approximant, namely the zero operator.  Thus
the estimates below remain valid for $N=0$ without interpreting
Lemma~\ref{lem:imaginary-taylor} at that index.

\subsection{Arbitrary bounded measurable Fourier windows}

For a bounded measurable set $E\subset\mathbb R$, write
\[
 \operatorname{diam}_{\mathrm{ess}}(E)
 :=\operatorname*{ess\,sup}_{x,y\in E}|x-y|.
\]
For a null set we use the convention
$\operatorname{diam}_{\mathrm{ess}}(E)=0$.
Changing $E$ on a null set does not change this quantity or any operator below.
When $|E|>0$, the midpoint of
$[\operatorname*{ess\,inf}E,\operatorname*{ess\,sup}E]$ will be called an
essential centre of $E$; for a null set, a centre may be chosen arbitrarily.

\begin{theorem}[Factorial singular-value envelope for measurable windows]
\label{thm:bounded-measurable-fourier}
Let $A,B\subset\mathbb R$ be bounded measurable sets of finite measure and let
$c>0$.  Put
\[
 K_{A,B,c}:=P_{cB}\mathcal F P_A,
 \qquad S_{A,B,c}:=K_{A,B,c}^*K_{A,B,c},
\]
and define
\begin{equation}
 m:=c|A||B|,
 \qquad
 \zeta:=\frac{\pi c}{2}
 \operatorname{diam}_{\mathrm{ess}}(A)
 \operatorname{diam}_{\mathrm{ess}}(B).
 \label{eq:m-zeta-bounded}
\end{equation}
Then, for every $N\in\mathbb N_0$,
\begin{equation}
 s_{N+1}(K_{A,B,c})
 \le \sqrt m\,\frac{\zeta^N}{N!}.
 \label{eq:bounded-fourier-envelope}
\end{equation}
Consequently, for $0<\epsilon<1/2$,
\begin{equation}
 \Lambda_\epsilon(S_{A,B,c})
 \le \mathcal N_{\mathrm{fac}}(m,\zeta,\epsilon)
 :=\min\left\{N\in\mathbb N_0:
       \sqrt m\,\frac{\zeta^N}{N!}\le\sqrt\epsilon\right\}.
 \label{eq:exact-factorial-count}
\end{equation}
In particular, $m\le\epsilon$ implies
$\Lambda_\epsilon(S_{A,B,c})=0$.  If $m>\epsilon$, set
\[
 \rho:=\frac12\log\frac{m}{\epsilon}>0.
\]
Then the explicit Lambert--$W$ consequence is
\begin{equation}
 \Lambda_\epsilon(S_{A,B,c})
 \le
 \left\lceil
 \frac{\rho}{W_0\!\left(\rho/(e\zeta)\right)}
 \right\rceil
 =
 \left\lceil
 \frac{\rho}{W_0\!\left(2\rho/(e\pi h)\right)}
 \right\rceil,
 \qquad
 h:=c\operatorname{diam}_{\mathrm{ess}}(A)
       \operatorname{diam}_{\mathrm{ess}}(B),
 \label{eq:bounded-fourier-lambert}
\end{equation}
where $W_0$ is the principal real branch of the Lambert function.
The exact factorial minimum in \eqref{eq:exact-factorial-count} is at least as
strong as its explicit surrogate \eqref{eq:bounded-fourier-lambert}.
At order zero the convention is $\zeta^0=1$, including when $\zeta=0$.
\end{theorem}

\begin{proof}
If either set has measure zero, then $K_{A,B,c}=0$ and every assertion is
immediate.  We therefore suppose $|A||B|>0$.

\emph{Step 1: centre the phase.}
Let $\alpha$ and $\beta$ be essential centres of $A$ and $B$, respectively.
For $x\in A$ and $\xi\in cB$, put
\[
 u=x-\alpha,
 \qquad v=\xi-c\beta.
\]
The elementary identity
\[
 x\xi=uv+\alpha\xi+c\beta x-\alpha c\beta
\]
shows that the Fourier kernel factors as
\begin{equation}
 e^{-2\pi i x\xi}
 =e^{-2\pi i\alpha\xi}
  e^{-2\pi i c\beta x}
  e^{2\pi i\alpha c\beta}
  e^{-2\pi iuv}.
 \label{eq:centered-phase-factorization}
\end{equation}
The first three factors have modulus one and are one-variable multipliers, up
to the harmless constant phase.  They therefore do not affect ranks, operator
norms, or singular values.

\emph{Step 2: construct a rank-$N$ kernel.}
For $N\ge1$, replace the last factor in
\eqref{eq:centered-phase-factorization} by
\[
 q_{N-1}(u,v):=\sum_{k=0}^{N-1}\frac{(-2\pi iuv)^k}{k!}.
\]
After restoring the one-variable phase factors, the resulting integral kernel
is a sum of $N$ products of a function of $x$ and a function of $\xi$.
Consequently, the associated operator $K_N$ has rank at most $N$.
For $N=0$, set $K_0=0$, which has rank zero.

\emph{Step 3: estimate the remainder in Hilbert--Schmidt norm.}
By Lemma~\ref{lem:imaginary-taylor}, for $N\ge1$,
\[
 |e^{-2\pi iuv}-q_{N-1}(u,v)|
 \le \frac{(2\pi|u||v|)^N}{N!}.
\]
The choice of essential centres gives
\[
 |u|\le\frac12\operatorname{diam}_{\mathrm{ess}}(A),
 \qquad
 |v|\le\frac c2\operatorname{diam}_{\mathrm{ess}}(B)
\]
almost everywhere on $A\times cB$.  Hence $2\pi|u||v|\le\zeta$, and
\[
 \|K_{A,B,c}-K_N\|_{\mathrm{HS}}^2
 \le |A|\,|cB|\left(\frac{\zeta^N}{N!}\right)^2
 =m\left(\frac{\zeta^N}{N!}\right)^2.
\]
For $N=0$, the same conclusion is simply
$\|K_{A,B,c}\|\le\|K_{A,B,c}\|_{\mathrm{HS}}=\sqrt m$.
Since operator norm is bounded by Hilbert--Schmidt norm,
\[
 \|K_{A,B,c}-K_N\|
 \le\sqrt m\,\frac{\zeta^N}{N!}.
\]
The approximation-number characterization
\[
 s_{N+1}(K)=\inf_{\operatorname{rank}R\le N}\|K-R\|
\]
now proves \eqref{eq:bounded-fourier-envelope}.

\emph{Step 4: transfer the singular-value estimate to the plunge count.}
The eigenvalues of $S_{A,B,c}=K_{A,B,c}^*K_{A,B,c}$ are
$s_j(K_{A,B,c})^2$.  Therefore
\[
 \Lambda_\epsilon(S_{A,B,c})
 \le \#\{j:s_j(K_{A,B,c})>\sqrt\epsilon\}.
\]
If the inequality in the defining set of
\eqref{eq:exact-factorial-count} holds, then
$s_{N+1}(K_{A,B,c})\le\sqrt\epsilon$; at most $N$ singular values are
strictly larger than $\sqrt\epsilon$.  This proves
\eqref{eq:exact-factorial-count}.  Taking $N=0$ proves the stated zero-count
criterion.

\emph{Step 5: solve a sufficient inequality with Lambert $W$.}
Assume $m>\epsilon$.  Then both sets have positive measure; a
positive-measure subset of $\mathbb R$ has positive essential diameter, so
$\zeta>0$ and the displayed Lambert expression is well defined.  The
elementary factorial inequality
$N!\ge(N/e)^N$ gives, for $N\ge1$,
\[
 \sqrt m\,\frac{\zeta^N}{N!}
 \le \sqrt m\left(\frac{e\zeta}{N}\right)^N.
\]
The right-hand side is at most $\sqrt\epsilon$ whenever
\begin{equation}
 N\log\frac{N}{e\zeta}\ge\rho.
 \label{eq:lambert-sufficient}
\end{equation}
Let
\[
 N_*:=\frac{\rho}{W_0(\rho/(e\zeta))}.
\]
Writing $w=W_0(\rho/(e\zeta))$, so that $we^w=\rho/(e\zeta)$,
one obtains
\[
 N_*=e\zeta e^w>e\zeta,
 \qquad
 \log\frac{N_*}{e\zeta}=w,
 \qquad
 N_*\log\frac{N_*}{e\zeta}=\rho.
\]
The function $x\mapsto x\log(x/(e\zeta))$ is strictly increasing for
$x>e\zeta$ because its derivative is
$1+\log(x/(e\zeta))>1$.  Thus $N=\lceil N_*\rceil$ satisfies
\eqref{eq:lambert-sufficient}, proving the first expression in
\eqref{eq:bounded-fourier-lambert}.  The second follows from
$\zeta=\pi h/2$.
\end{proof}

The next corollary records precisely what the Lambert expression does in a
deep tail.  It is an asymptotic statement about the explicit certificate, not
an assertion that the actual plunge count has the same asymptotic size.

\begin{corollary}[Ultra-deep direct certificate]
\label{cor:ultradeep-fourier}
Let $L=\log(1/\epsilon)$.  Suppose, as $\epsilon\downarrow0$, that
\[
 \frac{L}{h}\longrightarrow\infty,
 \qquad |\log m|=o(L).
\]
Then the unrounded expression on the right of
\eqref{eq:bounded-fourier-lambert} satisfies
\begin{equation}
 \frac{\rho}{W_0(2\rho/(e\pi h))}
 \sim
 \frac{L}{2W_0(L/(e\pi h))}
 =o(L).
 \label{eq:ultradeep-asymptotic}
\end{equation}
In particular, the right side of \eqref{eq:bounded-fourier-lambert} is
sublinear in $\log(1/\epsilon)$.
\end{corollary}

\begin{proof}
The hypothesis on $m$ gives
\[
 L+\log m=L(1+o(1))>0
\]
for all sufficiently small $\epsilon$.  Thus $m>\epsilon$ eventually, so the
quantity $\rho$ from Theorem~\ref{thm:bounded-measurable-fourier} is defined,
and
\[
 \rho=\frac12(L+\log m)=\frac L2(1+o(1)).
\]
Consequently,
\[
 \frac{2\rho}{e\pi h}
 =\frac{L}{e\pi h}(1+o(1))\longrightarrow\infty.
\]
For positive $x\to\infty$ and $a_x=1+o(1)$, monotonicity of $W_0$ together
with $W_0(x)=\log x-\log\log x+o(1)$ gives
$W_0(a_x\,x)/W_0(x)\to1$.  Applying this observation to the preceding display
and substituting $\rho\sim L/2$ proves the asymptotic equivalence in
\eqref{eq:ultradeep-asymptotic}.  Finally,
$W_0(L/(e\pi h))\to\infty$, so division by this factor makes the expression
$o(L)$.  The ceiling changes it by at most one and therefore does not affect
the conclusion.
\end{proof}

\subsection{Sharper polynomial approximation and residual-energy counting}
\label{subsec:sharper-fourier-polynomial}

The factorial estimate in Theorem~\ref{thm:bounded-measurable-fourier}
uses one particular polynomial, estimates its residual by a pointwise
supremum, and then asks the entire residual operator to lie below one
singular-value threshold.  All three choices can be sharpened without adding
regularity assumptions on the windows.  The residual-energy mechanism used
below is Lemma~\ref{lem:sharp-residual-energy}, with the strict-integer
function $\kappa$ defined in Section~\ref{sec:sharp-residual-energy}.

We now specialize this principle to bounded measurable Fourier windows while
retaining the normalization
\[
 K_{A,B,c}=P_{cB}\mathcal F P_A,
 \qquad
 (\mathcal Ff)(\xi)=\int_{\mathbb R}e^{-2\pi i x\xi}f(x)\,dx.
\]
Let $A,B\subset\mathbb R$ be bounded measurable sets of finite measure and
let $c>0$.  For centres $\alpha,\beta\in\mathbb R$, set
\begin{equation}
 t_{\alpha,\beta}(x,\xi)
 :=2\pi(x-\alpha)(\xi-c\beta),
 \qquad (x,\xi)\in A\times cB.
 \label{eq:centered-product-variable}
\end{equation}
The push-forward of Lebesgue measure on $A\times cB$ under
$t_{\alpha,\beta}$ is a finite positive Borel measure
$\nu_{\alpha,\beta}$.  Its total mass is
\begin{equation}
 \nu_{\alpha,\beta}(\mathbb R)=|A|\,|cB|=c|A||B|=m.
 \label{eq:pushforward-total-mass}
\end{equation}
No density or regularity of $\nu_{\alpha,\beta}$ is assumed.

For $N\geq1$, let $\mathcal P_{N-1}$ be the complex polynomials of degree at
most $N-1$, and put $\mathcal P_{-1}:=\{0\}$.  Define the weighted best
polynomial residual
\begin{align}
 \mathcal R_N(\alpha,\beta)^2
 &: =\min_{p\in\mathcal P_{N-1}}
 \int_{\mathbb R}|e^{-it}-p(t)|^2\,d\nu_{\alpha,\beta}(t)
 \label{eq:best-weighted-residual}\\
 &=\min_{p\in\mathcal P_{N-1}}
 \int_A\int_{cB}
 \left|e^{-it_{\alpha,\beta}(x,\xi)}
       -p\bigl(t_{\alpha,\beta}(x,\xi)\bigr)\right|^2
 \,d\xi\,dx.
 \nonumber
\end{align}
At $N=0$ this definition gives
$\mathcal R_0(\alpha,\beta)=\sqrt m$.

\begin{theorem}[Weighted best-polynomial and moment envelopes]
\label{thm:best-polynomial-moment-envelope}
For every $N\in\mathbb N_0$ and every $\alpha,\beta\in\mathbb R$,
\begin{equation}
 s_{N+1}(K_{A,B,c})\leq\mathcal R_N(\alpha,\beta).
 \label{eq:best-polynomial-sv-envelope}
\end{equation}
For $N\geq1$, the minimum in
\eqref{eq:best-weighted-residual} is an explicit finite-dimensional
least-squares problem.  More precisely, with indices $0\leq j,k<N$, define
\begin{equation}
 G_{jk}:=\int_{\mathbb R}t^{j+k}\,d\nu_{\alpha,\beta}(t),
 \qquad
 b_j:=\int_{\mathbb R}t^j e^{-it}\,d\nu_{\alpha,\beta}(t).
 \label{eq:polynomial-Gram-data}
\end{equation}
If $G^\dagger$ denotes the Moore--Penrose inverse, then
\begin{equation}
 \mathcal R_N(\alpha,\beta)^2
 =m-b^*G^\dagger b.
 \label{eq:best-residual-Gram-formula}
\end{equation}
In the original $A\times B$ variables the Gram entries are products of
one-dimensional moments:
\begin{equation}
 G_{jk}
 =c(2\pi c)^{j+k}
 \left(\int_A(x-\alpha)^{j+k}\,dx\right)
 \left(\int_B(\eta-\beta)^{j+k}\,d\eta\right).
 \label{eq:Gram-product-moments}
\end{equation}

Define, with the convention $|x-a|^0=1$,
\begin{equation}
 \mathcal M_N(A):=\inf_{a\in\mathbb R}
       \int_A|x-a|^{2N}\,dx,
 \qquad
 \mathcal M_N(B):=\inf_{b\in\mathbb R}
       \int_B|\eta-b|^{2N}\,d\eta,
 \label{eq:optimized-even-moments}
\end{equation}
and
\begin{equation}
 E_N^{\mathrm{mom}}
 :=\frac{(2\pi c)^N\sqrt c}{N!}
 \bigl(\mathcal M_N(A)\mathcal M_N(B)\bigr)^{1/2}.
 \label{eq:optimized-moment-error}
\end{equation}
Then there is a rank-at-most-$N$ approximation to $K_{A,B,c}$ whose
Hilbert--Schmidt residual is at most $E_N^{\mathrm{mom}}$, and hence
\begin{equation}
 s_{N+1}(K_{A,B,c})\leq E_N^{\mathrm{mom}}.
 \label{eq:optimized-moment-sv}
\end{equation}
If $D_A=\operatorname{diam}_{\mathrm{ess}}(A)$,
$D_B=\operatorname{diam}_{\mathrm{ess}}(B)$, and
\begin{equation}
 \zeta=\frac{\pi c}{2}D_AD_B,
 \label{eq:zeta-recalled-sharper}
\end{equation}
then
\begin{equation}
 E_N^{\mathrm{mom}}
 \leq\sqrt m\,\frac{\zeta^N}{N!}.
 \label{eq:moment-dominates-diameter-Taylor}
\end{equation}
If $A$ and $B$ are intervals of positive length, equality in the optimized
moments gives the explicit envelope
\begin{equation}
 E_N^{\mathrm{mom}}
 =\sqrt m\,\frac{\zeta^N}{N!(2N+1)}.
 \label{eq:interval-moment-improvement}
\end{equation}
The formula includes $N=0$, when the additional factor is one; for every
$N\geq1$ it is strictly smaller than the diameter-based Taylor envelope.
\end{theorem}

\begin{proof}
If $m=0$, then at least one window is null, $K_{A,B,c}=0$, and all the
assertions follow immediately.  We may therefore assume $m>0$ whenever a
minimizing centre is needed.

\emph{Step 1: convert a product polynomial into a finite-rank kernel.}
The centred phase identity
\begin{equation}
 e^{-2\pi ix\xi}
 =e^{-2\pi i\alpha\xi}
  e^{-2\pi i c\beta x}
  e^{2\pi i\alpha c\beta}
  e^{-it_{\alpha,\beta}(x,\xi)}
 \label{eq:best-poly-phase-factorization}
\end{equation}
separates three unimodular one-variable factors from the interaction.
If
$p(t)=\sum_{k=0}^{N-1}a_kt^k$, replace the last factor in
\eqref{eq:best-poly-phase-factorization} by
$p(t_{\alpha,\beta}(x,\xi))$.  Since
\[
 t_{\alpha,\beta}(x,\xi)^k
 =(2\pi)^k(x-\alpha)^k(\xi-c\beta)^k,
\]
the resulting integral kernel is a sum of at most $N$ products of one
function of $x$ and one function of $\xi$.  The corresponding operator,
denoted by $Q_p$, therefore has rank at most $N$.  All the separated
functions belong to the relevant $L^2$ spaces because both windows are
bounded and have finite measure.  At $N=0$, take $p=0$ and $Q_p=0$.

\emph{Step 2: identify the residual exactly.}
All prefactors in \eqref{eq:best-poly-phase-factorization} have modulus one.
The Hilbert--Schmidt norm of the residual is consequently
\begin{align}
 \|K_{A,B,c}-Q_p\|_{\mathrm{HS}}^2
 &=\int_A\int_{cB}
 \left|e^{-it_{\alpha,\beta}(x,\xi)}
 -p\bigl(t_{\alpha,\beta}(x,\xi)\bigr)\right|^2
 \,d\xi\,dx.
 \label{eq:poly-residual-is-HS}
\end{align}
The finite-dimensional space $\mathcal P_{N-1}$ is closed in
$L^2(\nu_{\alpha,\beta})$ after quotienting by its possible null
polynomials, so the minimum is attained.  Choosing a minimizer in
\eqref{eq:poly-residual-is-HS} and using the approximation-number
characterization of $s_{N+1}$ proves
\eqref{eq:best-polynomial-sv-envelope}.

\emph{Step 3: solve the weighted least-squares problem, including a singular
Gram matrix.}
Write $p(t)=\sum_{k=0}^{N-1}a_kt^k$ and $a=(a_0,\ldots,a_{N-1})^T$.
Expanding the square gives
\begin{equation}
 \int|e^{-it}-p(t)|^2\,d\nu
 =m-2\operatorname{Re}(a^*b)+a^*Ga,
 \label{eq:Gram-objective-expanded}
\end{equation}
where $G$ and $b$ are defined in
\eqref{eq:polynomial-Gram-data}.  The normal equations are
\begin{equation}
 Ga=b.
 \label{eq:Gram-normal-equations}
\end{equation}
They are consistent even when $G$ is singular.  Indeed, if $z\in\ker G$,
then
\[
 0=z^*Gz=\int\left|\sum_{k=0}^{N-1}z_kt^k\right|^2d\nu(t),
\]
so the polynomial $\sum z_kt^k$ vanishes $\nu$-almost everywhere.  It
follows that
\[
 z^*b
 =\int\overline{\sum_{k=0}^{N-1}z_kt^k}\,e^{-it}\,d\nu(t)=0.
\]
Thus $b\perp\ker G$, equivalently $b\in\operatorname{Ran}G$.  The vector
$a=G^\dagger b$ solves \eqref{eq:Gram-normal-equations}.  Substitution into
\eqref{eq:Gram-objective-expanded} yields
\[
 m-2b^*G^\dagger b+b^*G^\dagger GG^\dagger b
 =m-b^*G^\dagger b,
\]
which proves \eqref{eq:best-residual-Gram-formula}.  Finally, substitute
$\xi=c\eta$ into the first formula in
\eqref{eq:polynomial-Gram-data}.  Since $d\xi=c\,d\eta$ and
\[
 t_{\alpha,\beta}(x,c\eta)
 =2\pi c(x-\alpha)(\eta-\beta),
\]
The integrand is absolutely integrable because both windows are bounded and
have finite measure.  Fubini's theorem therefore separates the two factors and
gives \eqref{eq:Gram-product-moments}.

\emph{Step 4: show that the optimized moments are attained.}
For $N=0$, both quantities in \eqref{eq:optimized-even-moments} are simply
the measures of the corresponding sets, independently of the centre.
Suppose $N\geq1$.  Because $A$ is bounded, the function
\[
 a\longmapsto\int_A|x-a|^{2N}\,dx
\]
is continuous.  If $A\subset[-R,R]$ up to a null set and $|a|>2R$, then
$|x-a|\geq|a|/2$ almost everywhere on $A$, so
\[
 \int_A|x-a|^{2N}\,dx
 \geq |A|\left(\frac{|a|}{2}\right)^{2N}\longrightarrow\infty
 \quad\text{as }|a|\to\infty.
\]
Thus the function is coercive and has a minimizer $\alpha_N$.  The same
argument gives a minimizer $\beta_N$ for $B$.

\emph{Step 5: use Taylor's remainder at the moment-minimizing centres.}
For $N\geq1$, take the degree-$(N-1)$ Taylor polynomial of $e^{-it}$ at
$t=0$.  Lemma~\ref{lem:imaginary-taylor} and
\eqref{eq:centered-product-variable} give
\begin{align}
 \left|e^{-it_{\alpha_N,\beta_N}(x,\xi)}
 -\sum_{k=0}^{N-1}\frac{(-it_{\alpha_N,\beta_N}(x,\xi))^k}{k!}\right|
 &\leq\frac{(2\pi)^N}{N!}
 |x-\alpha_N|^N|\xi-c\beta_N|^N.
 \label{eq:moment-Taylor-pointwise}
\end{align}
Square, integrate on $A\times cB$, and use $\xi=c\eta$.  The result is
\begin{align}
 \|K_{A,B,c}-Q_N^{\mathrm{mom}}\|_{\mathrm{HS}}^2
 &\leq
 \frac{c(2\pi c)^{2N}}{(N!)^2}
 \mathcal M_N(A)\mathcal M_N(B)
 =\bigl(E_N^{\mathrm{mom}}\bigr)^2.
 \label{eq:moment-HS-residual}
\end{align}
The approximant has rank at most $N$ by Step~1.  For $N=0$, the zero
approximant has squared Hilbert--Schmidt error
$\|K_{A,B,c}\|_{\mathrm{HS}}^2=m$, which is exactly
$(E_0^{\mathrm{mom}})^2$.  This proves
\eqref{eq:optimized-moment-sv} at every order.

\emph{Step 6: compare with the diameter envelope.}
Let $a_0$ and $b_0$ be the midpoints of the smallest essential closed
intervals containing $A$ and $B$.  Then
\[
 \mathcal M_N(A)
 \leq\int_A|x-a_0|^{2N}\,dx
 \leq |A|\left(\frac{D_A}{2}\right)^{2N},
\]
and similarly
\[
 \mathcal M_N(B)
 \leq |B|\left(\frac{D_B}{2}\right)^{2N}.
\]
Substitution into \eqref{eq:optimized-moment-error} gives
\begin{align*}
 E_N^{\mathrm{mom}}
 &\leq
 \frac{(2\pi c)^N\sqrt c}{N!}
 \sqrt{|A||B|}
 \left(\frac{D_AD_B}{4}\right)^N
 =\sqrt m\,\frac{\zeta^N}{N!},
\end{align*}
which is \eqref{eq:moment-dominates-diameter-Taylor}.

\emph{Step 7: evaluate the interval case.}
If $A$ is an interval of length $a$, symmetry and convexity show that the
minimizing centre is its midpoint.  After translation,
\begin{equation}
 \mathcal M_N(A)
 =\int_{-a/2}^{a/2}|x|^{2N}\,dx
 =\frac{a^{2N+1}}{2^{2N}(2N+1)}.
 \label{eq:interval-exact-even-moment}
\end{equation}
For $N\geq1$, convexity also proves minimality directly; for $N=0$ every
centre gives $\mathcal M_0(A)=a$.  Applying the same computation to an
interval $B$ of length $b$ and substituting in
\eqref{eq:optimized-moment-error} yields
\[
 E_N^{\mathrm{mom}}
 =\sqrt{cab}\,
 \frac{(\pi cab/2)^N}{N!(2N+1)}
 =\sqrt m\,\frac{\zeta^N}{N!(2N+1)},
\]
as asserted.
\end{proof}

The next construction improves the pointwise polynomial itself.  Let
$J_k$ denote the Bessel function of the first kind and $T_k$ the Chebyshev
polynomial determined by $T_k(\cos\theta)=\cos(k\theta)$.  Define
\begin{equation}
 \chi_0(\zeta):=1,
 \qquad
 \chi_N(\zeta):=\min\left\{
  1,\ \frac{\zeta^N}{N!},
  2\sum_{k=N}^{\infty}|J_k(\zeta)|
 \right\}
 \quad(N\geq1).
 \label{eq:Chebyshev-Bessel-chi}
\end{equation}
The first entry corresponds to the zero polynomial, the second to Taylor's
polynomial, and the third to a truncated Chebyshev--Bessel expansion.  Thus
the minimum never weakens the Taylor estimate.

\begin{lemma}[Chebyshev--Bessel polynomial with a geometric tail]
\label{lem:Chebyshev-Bessel-polynomial}
Let $\zeta\geq0$ and $N\in\mathbb N_0$.  There is a complex polynomial
$p_{N-1}$ of degree at most $N-1$, with $p_{-1}=0$, such that
\begin{equation}
 \sup_{|t|\leq\zeta}|e^{-it}-p_{N-1}(t)|
 \leq\chi_N(\zeta).
 \label{eq:Chebyshev-Bessel-uniform-error}
\end{equation}
For every $k\in\mathbb N_0$,
\begin{equation}
 |J_k(\zeta)|\leq\frac{(\zeta/2)^k}{k!}.
 \label{eq:Bessel-Poisson-bound}
\end{equation}
Consequently, if $N\geq1$ and $N+1>\zeta/2$, then
\begin{equation}
 2\sum_{k=N}^{\infty}|J_k(\zeta)|
 \leq
 \frac{2(\zeta/2)^N}
 {N!\left(1-\dfrac{\zeta}{2(N+1)}\right)}.
 \label{eq:Bessel-geometric-tail}
\end{equation}
In particular, in this range the infinite sum in
\eqref{eq:Chebyshev-Bessel-chi} may be replaced by the right-hand side of
\eqref{eq:Bessel-geometric-tail} to obtain a completely elementary
certificate.
\end{lemma}

\begin{proof}
\emph{Step 1: derive the Chebyshev--Bessel expansion.}
For $\zeta\geq0$ and $z\neq0$, multiplication of the two exponential power
series gives the Bessel generating identity
\begin{equation}
 e^{(\zeta/2)(z-z^{-1})}
 =\sum_{k=-\infty}^{\infty}J_k(\zeta)z^k.
 \label{eq:Bessel-generating-identity}
\end{equation}
Indeed, the coefficient of $z^k$ for $k\geq0$ is
\[
 \sum_{\ell=0}^{\infty}
 \frac{(-1)^\ell(\zeta/2)^{2\ell+k}}
      {\ell!(\ell+k)!}
 =J_k(\zeta),
\]
and the negative coefficients obey
$J_{-k}(\zeta)=(-1)^kJ_k(\zeta)$.  On $|z|=1$ the product series is
absolutely convergent, so the rearrangement is legitimate.

Set $z=-ie^{i\theta}$.  Then
$z-z^{-1}=-2i\cos\theta$, and pairing the $k$ and $-k$ terms in
\eqref{eq:Bessel-generating-identity} yields
\begin{equation}
 e^{-i\zeta\cos\theta}
 =J_0(\zeta)
  +2\sum_{k=1}^{\infty}(-i)^kJ_k(\zeta)\cos(k\theta).
 \label{eq:Jacobi-Anger-cosine}
\end{equation}
With $s=\cos\theta$ and $T_k(s)=\cos(k\theta)$, this becomes
\begin{equation}
 e^{-i\zeta s}
 =J_0(\zeta)
  +2\sum_{k=1}^{\infty}(-i)^kJ_k(\zeta)T_k(s),
 \qquad -1\leq s\leq1.
 \label{eq:Chebyshev-Bessel-expansion}
\end{equation}

\emph{Step 2: bound the Bessel coefficients from Poisson's integral.}
For an integer $k\geq0$, Poisson's integral representation is
\begin{equation}
 J_k(\zeta)
 =\frac{(\zeta/2)^k}
 {\sqrt\pi\,\Gamma(k+\tfrac12)}
 \int_{-1}^{1}e^{i\zeta u}(1-u^2)^{k-1/2}\,du.
 \label{eq:Bessel-Poisson-integral}
\end{equation}
For completeness, this representation follows by expanding
$e^{i\zeta u}$, observing that its odd powers integrate to zero, and using
\[
 \int_{-1}^{1}u^{2\ell}(1-u^2)^{k-1/2}\,du
 =B\!\left(\ell+\tfrac12,k+\tfrac12\right).
\]
The identity
$\Gamma(\ell+\tfrac12)=(2\ell)!\sqrt\pi/(4^\ell\ell!)$
then recovers exactly the Bessel series displayed in Step~1.  Termwise
integration is justified by uniform convergence of the exponential series
on $[-1,1]$.

Taking absolute values in \eqref{eq:Bessel-Poisson-integral} and using
$|e^{i\zeta u}|=1$ gives
\begin{align*}
 |J_k(\zeta)|
 &\leq\frac{(\zeta/2)^k}
 {\sqrt\pi\,\Gamma(k+\tfrac12)}
 \int_{-1}^{1}(1-u^2)^{k-1/2}\,du\\
 &=\frac{(\zeta/2)^k}
 {\sqrt\pi\,\Gamma(k+\tfrac12)}
 B\!\left(\tfrac12,k+\tfrac12\right)
 =\frac{(\zeta/2)^k}{\Gamma(k+1)},
\end{align*}
which is \eqref{eq:Bessel-Poisson-bound}.  This also proves absolute and
uniform convergence of \eqref{eq:Chebyshev-Bessel-expansion}, because
$|T_k(s)|\leq1$ on $[-1,1]$ and
$\sum_{k\geq0}(\zeta/2)^k/k!<\infty$.

\emph{Step 3: truncate and compare the three admissible polynomials.}
Assume first that $N\geq1$ and $\zeta>0$.  Truncating
\eqref{eq:Chebyshev-Bessel-expansion} gives
\begin{equation}
 p^{\mathrm{CB}}_{N-1}(t)
 :=J_0(\zeta)
 +2\sum_{k=1}^{N-1}(-i)^kJ_k(\zeta)
 T_k\!\left(\frac{t}{\zeta}\right).
 \label{eq:truncated-Chebyshev-Bessel-polynomial}
\end{equation}
This is a polynomial in $t$ of degree at most $N-1$.  For $|t|\leq\zeta$,
the omitted tail and $|T_k(t/\zeta)|\leq1$ imply
\begin{equation}
 |e^{-it}-p^{\mathrm{CB}}_{N-1}(t)|
 \leq2\sum_{k=N}^{\infty}|J_k(\zeta)|.
 \label{eq:CB-exact-tail}
\end{equation}
Two other degree-at-most-$(N-1)$ choices are available: the zero polynomial,
whose error is exactly one, and the Taylor polynomial, whose error is at most
$\zeta^N/N!$ by Lemma~\ref{lem:imaginary-taylor}.  Select the polynomial
with the smallest of these three certified errors.  This proves
\eqref{eq:Chebyshev-Bessel-uniform-error}.

If $\zeta=0$ and $N\geq1$, the interval $[-\zeta,\zeta]$ consists only of
$0$; the constant polynomial $p(t)=1$ has zero error, which agrees with the
Taylor entry $\zeta^N/N!=0$.  If $N=0$, the prescribed zero polynomial has
uniform error one for every $\zeta\geq0$.  Thus all endpoint cases are
covered without dividing by $\zeta$.

\emph{Step 4: sum the Poisson bounds geometrically.}
Put $a:=\zeta/2$.  For $k\geq N$, consecutive terms in the exponential tail
satisfy
\[
 \frac{a^{k+1}/(k+1)!}{a^k/k!}
 =\frac{a}{k+1}
 \leq\frac{a}{N+1}=:q.
\]
If $N+1>a$, then $0\leq q<1$, and therefore
\begin{align*}
 2\sum_{k=N}^{\infty}|J_k(\zeta)|
 &\leq2\sum_{k=N}^{\infty}\frac{a^k}{k!}
 \leq2\frac{a^N}{N!}\sum_{\ell=0}^{\infty}q^\ell
 =\frac{2a^N}{N!(1-q)}.
\end{align*}
This is exactly \eqref{eq:Bessel-geometric-tail}.
\end{proof}

\begin{theorem}[Combined best-polynomial, moment, and Bessel certificate]
\label{thm:combined-sharper-Fourier-certificate}
Retain the assumptions and notation of
Theorem~\ref{thm:best-polynomial-moment-envelope}.  Choose essential centres
$\alpha_0,\beta_0$ of $A,B$.  For every $N\geq1$, choose moment-minimizing
centres $\alpha_N,\beta_N$ as in
\eqref{eq:optimized-even-moments}; at $N=0$ choose arbitrary centres.  Put
\begin{equation}
 \begin{aligned}
 \mathfrak R_N
 &:=\min\bigl\{
 \mathcal R_N(\alpha_0,\beta_0),
 \mathcal R_N(\alpha_N,\beta_N)
 \bigr\},\\
 \widehat E_N
 &:=\min\left\{
 E_N^{\mathrm{mom}},\sqrt m\,\chi_N(\zeta)
 \right\}.
 \end{aligned}
 \label{eq:exact-and-explicit-hybrid-residuals}
\end{equation}
Then
\begin{equation}
 \mathfrak R_N\leq\widehat E_N
 \leq\sqrt m\,\frac{\zeta^N}{N!}
 \qquad(N\in\mathbb N_0).
 \label{eq:hybrid-residual-hierarchy}
\end{equation}
For every $0<\epsilon<1/2$,
\begin{align}
 \Lambda_\epsilon(S_{A,B,c})
 &\leq
 \min_{N\in\mathbb N_0}
 \left[
 N+\kappa\!\left(\frac{\mathfrak R_N^2}{\epsilon}\right)
 \right]
 \label{eq:exact-best-residual-energy-count}\\
 &\leq
 \min_{N\in\mathbb N_0}
 \left[
 N+\kappa\!\left(\frac{\widehat E_N^2}{\epsilon}\right)
 \right]
 \label{eq:explicit-hybrid-residual-energy-count}\\
 &\leq
 \min_{N\in\mathbb N_0}
 \left[
 N+\kappa\!\left(
 \frac{m\chi_N(\zeta)^2}{\epsilon}
 \right)
 \right].
 \label{eq:Bessel-residual-energy-count}
\end{align}
Each minimum is finite.  Moreover, every bound above is no larger than the
factorial certificate
$\mathcal N_{\mathrm{fac}}(m,\zeta,\epsilon)$ in
\eqref{eq:exact-factorial-count}.  Thus the refinement cannot worsen that
certificate; it may be strict because it uses either a better polynomial, a
smaller weighted residual, or the Hilbert--Schmidt energy across the entire
residual.
\end{theorem}

\begin{proof}
If $m=0$, then $K_{A,B,c}=0$ and the choice $N=0$ makes every right-hand
side zero.  Assume henceforth that $m>0$.  Then both windows have positive
measure, hence positive essential diameter, and therefore $\zeta>0$.

\emph{Step 1: compare the exact and explicit residuals.}
At the centres chosen to minimize the moments, the Taylor construction in
\eqref{eq:moment-HS-residual} gives
\[
 \mathcal R_N(\alpha_N,\beta_N)\leq E_N^{\mathrm{mom}}.
\]
At the essential centres, one has almost everywhere
\begin{align*}
 |t_{\alpha_0,\beta_0}(x,\xi)|
 &\leq2\pi\frac{D_A}{2}\frac{cD_B}{2}
 =\zeta.
\end{align*}
Apply the polynomial from
Lemma~\ref{lem:Chebyshev-Bessel-polynomial}, square its uniform error, and
integrate over $A\times cB$.  Since that product set has measure $m$,
\[
 \mathcal R_N(\alpha_0,\beta_0)
 \leq\sqrt m\,\chi_N(\zeta).
\]
Taking the smaller of the two exact residuals proves the first inequality in
\eqref{eq:hybrid-residual-hierarchy}.  The second follows from
\eqref{eq:moment-dominates-diameter-Taylor} and from the Taylor entry in the
definition of $\chi_N$.  At $N=0$, all quantities equal $\sqrt m$, so the
same hierarchy remains valid.

\emph{Step 2: count with the exact residual energy.}
For each of the two centre choices in
\eqref{eq:exact-and-explicit-hybrid-residuals}, the minimum defining
$\mathcal R_N$ is attained and gives a rank-at-most-$N$ approximant by
Theorem~\ref{thm:best-polynomial-moment-envelope}.  Choose the centre pair
that realizes the smaller residual and denote the corresponding approximant
by $Q_N$.  Since $K_{A,B,c}$ is a contraction,
\[
 \Lambda_\epsilon(S_{A,B,c})
 \leq n(\sqrt\epsilon;K_{A,B,c}).
\]
Apply Lemma~\ref{lem:sharp-residual-energy} with
$T=K_{A,B,c}$, $T_N=Q_N$, and $\tau=\sqrt\epsilon$.  Because
$\|K_{A,B,c}-Q_N\|_{\mathcal S_2}=\mathfrak R_N$, it gives
\[
 \Lambda_\epsilon(S_{A,B,c})
 \leq N+\kappa(\mathfrak R_N^2/\epsilon)
\]
for every $N$.  Minimizing proves
\eqref{eq:exact-best-residual-energy-count}.

\emph{Step 3: pass to explicit upper envelopes.}
The function $\kappa$ is nondecreasing on $[0,\infty)$.  Therefore the
first inequality in \eqref{eq:hybrid-residual-hierarchy} gives
\eqref{eq:explicit-hybrid-residual-energy-count}.  Since
$\widehat E_N\leq\sqrt m\chi_N(\zeta)$, the same argument gives
\eqref{eq:Bessel-residual-energy-count}.

\emph{Step 4: verify finiteness and comparison with the old certificate.}
The factorial sequence
\[
 a_N:=\sqrt m\,\frac{\zeta^N}{N!}
\]
tends to zero.  Hence for all sufficiently large $N$, its square divided by
$\epsilon$ is at most one, so the corresponding $\kappa$ term vanishes.
Thus every displayed minimum has a finite candidate.

Let
\[
 N_{\mathrm{fac}}
 :=\mathcal N_{\mathrm{fac}}(m,\zeta,\epsilon).
\]
By its definition,
\[
 \sqrt m\,\frac{\zeta^{N_{\mathrm{fac}}}}
 {N_{\mathrm{fac}}!}\leq\sqrt\epsilon.
\]
The hierarchy \eqref{eq:hybrid-residual-hierarchy} therefore implies
\begin{equation}
 \frac{\mathfrak R_{N_{\mathrm{fac}}}^2}{\epsilon}\leq1,
 \qquad
 \frac{\widehat E_{N_{\mathrm{fac}}}^2}{\epsilon}\leq1,
 \qquad
 \frac{m\chi_{N_{\mathrm{fac}}}(\zeta)^2}{\epsilon}\leq1.
 \label{eq:factorial-order-new-residuals-small}
\end{equation}
Because $\kappa(r)=0$ on $[0,1]$, evaluating every new minimum at
$N=N_{\mathrm{fac}}$ shows that it is at most $N_{\mathrm{fac}}$.
This proves the final assertion.
\end{proof}

\begin{remark}[What the sharpening does and does not claim]
\label{rem:sharper-Fourier-scope}
The quantity $\mathcal R_N(\alpha,\beta)$ is the exact least-squares error
within the degree-$(N-1)$ polynomial-in-the-centred-product class.  It is not
claimed to be the unrestricted best rank-$N$ approximation error.  Likewise,
the factor $1/(2N+1)$ in \eqref{eq:interval-moment-improvement} sharpens the
diameter-based \emph{Taylor envelope}; it is not an assertion that the true
$(N+1)$st singular value has that exact size.  The Chebyshev--Bessel bound is
most effective after the truncation order passes the oscillatory scale.  The
minimum in \eqref{eq:Chebyshev-Bessel-chi} is retained because the Bessel tail
need not beat Taylor at every finite parameter value.  Finally, the strict
rounding in $\kappa$ comes from the open singular-value threshold; replacing
$\kappa(r)$ by $\lfloor r\rfloor$ would lose a genuine endpoint improvement
when $r$ is an integer.
\end{remark}

\subsection{Soft masks and an exact two-channel defect factorization}

Let $0\le u,v\le1$ be measurable.  Set
\[
 U:=M_u,
 \qquad V:=\mathcal F^{-1}M_v\mathcal F,
 \qquad K_{u,v}:=M_{\sqrt v}\mathcal F M_{\sqrt u},
 \qquad S_{u,v}:=K_{u,v}^*K_{u,v}=U^{1/2}VU^{1/2}.
\]
Thus hard cutoffs are recovered by $u=\mathbf1_A$ and
$v=\mathbf1_{cB}$.

\begin{theorem}[Soft bounded masks]
\label{thm:soft-bounded-masks}
Suppose $u,v\in L^1(\mathbb R)$ and their essential supports are bounded.
Choose centres $\alpha,\beta\in\mathbb R$ and radii $r_u,r_v<\infty$ such
that
\[
 \operatorname{ess\,supp}u\subset[\alpha-r_u,\alpha+r_u],
 \qquad
 \operatorname{ess\,supp}v\subset[\beta-r_v,\beta+r_v].
\]
Define
\[
 m_{u,v}:=\left(\int_{\mathbb R}u\right)
            \left(\int_{\mathbb R}v\right),
 \qquad
 \zeta_{u,v}:=2\pi r_ur_v.
\]
Then, for every $N\in\mathbb N_0$,
\begin{equation}
 s_{N+1}(K_{u,v})
 \le\sqrt{m_{u,v}}\,\frac{\zeta_{u,v}^N}{N!}.
 \label{eq:soft-bounded-envelope}
\end{equation}
Moreover, the column operator
\begin{equation}
 \mathcal T_{u,v}:=
 \begin{pmatrix}
 (I-U)^{1/2}VU^{1/2}\\[2mm]
 (V-V^2)^{1/2}U^{1/2}
 \end{pmatrix}:L^2(\mathbb R)\longrightarrow
 L^2(\mathbb R)\oplus L^2(\mathbb R)
 \label{eq:soft-defect-column}
\end{equation}
satisfies the exact identity
\begin{equation}
 \mathcal T_{u,v}^*\mathcal T_{u,v}=S_{u,v}-S_{u,v}^2.
 \label{eq:soft-defect-identity}
\end{equation}
It also obeys the same factorial envelope,
\begin{equation}
 s_{N+1}(\mathcal T_{u,v})
 \le\sqrt{m_{u,v}}\,\frac{\zeta_{u,v}^N}{N!}.
 \label{eq:soft-defect-envelope}
\end{equation}
Consequently, for $0<\epsilon<1/2$, the direct count uses the threshold
$\sqrt\epsilon$, whereas the exact defect count may use
$\sqrt{\epsilon(1-\epsilon)}$.  For this common envelope the direct threshold
is the stronger of the two.
\end{theorem}

\begin{proof}
\emph{Step 1: approximate the weighted Fourier kernel.}
The integral kernel of $K_{u,v}$ is
\[
 k(\xi,x)=\sqrt{v(\xi)}e^{-2\pi ix\xi}\sqrt{u(x)}.
\]
Factor the phase around $(\alpha,\beta)$ exactly as in
\eqref{eq:centered-phase-factorization}, with $c\beta$ there replaced by
$\beta$ here.  Taylor expansion of the interaction
$e^{-2\pi i(x-\alpha)(\xi-\beta)}$ through degree $N-1$ produces an
operator $K_N$ of rank at most $N$.  Lemma~\ref{lem:imaginary-taylor} and the
support conditions give the pointwise bound
\[
 |k(\xi,x)-k_N(\xi,x)|
 \le \sqrt{v(\xi)u(x)}\frac{(2\pi r_ur_v)^N}{N!}.
\]
After squaring and integrating first in $x$ and then in $\xi$, we obtain
\[
 \|K_{u,v}-K_N\|_{\mathrm{HS}}^2
 \le m_{u,v}\left(\frac{\zeta_{u,v}^N}{N!}\right)^2.
\]
The approximation-number characterization proves
\eqref{eq:soft-bounded-envelope}, including $N=0$ by taking $K_0=0$.

\emph{Step 2: verify the defect identity.}
Both $U$ and $V$ are positive contractions.  Direct multiplication of the two
rows in \eqref{eq:soft-defect-column} gives
\begin{align*}
 \mathcal T_{u,v}^*\mathcal T_{u,v}
 &=U^{1/2}V(I-U)VU^{1/2}
   +U^{1/2}(V-V^2)U^{1/2}\\
 &=U^{1/2}(V^2-VUV+V-V^2)U^{1/2}\\
 &=U^{1/2}(V-VUV)U^{1/2}\\
 &=U^{1/2}VU^{1/2}
   -U^{1/2}VU VU^{1/2}\\
 &=S_{u,v}-S_{u,v}^2,
\end{align*}
which proves \eqref{eq:soft-defect-identity}.

\emph{Step 3: factor the defect column through $K_{u,v}$.}
Define $W:L^2(\mathbb R_\xi)\to L^2(\mathbb R_x)\oplus
L^2(\mathbb R_x)$ by
\[
 Wg=
 \begin{pmatrix}
 M_{\sqrt{1-u}}\mathcal F^{-1}M_{\sqrt v}g\\
 \mathcal F^{-1}M_{\sqrt{1-v}}g
 \end{pmatrix}.
\]
Using functional calculus for
$V=\mathcal F^{-1}M_v\mathcal F$, one checks row by row that
\[
 WK_{u,v}=
 \begin{pmatrix}
 (I-U)^{1/2}VU^{1/2}\\
 (V-V^2)^{1/2}U^{1/2}
 \end{pmatrix}
 =\mathcal T_{u,v}.
\]
Furthermore,
\begin{align*}
 W^*W
 &=M_{\sqrt v}\mathcal F M_{1-u}\mathcal F^{-1}M_{\sqrt v}
   +M_{1-v}\\
 &\le M_v+M_{1-v}=I.
\end{align*}
Thus $W$ is a contraction.  If $K_N$ is the rank-$N$ approximant constructed
in Step~1, then $WK_N$ has rank at most $N$ and
\[
 \|\mathcal T_{u,v}-WK_N\|
 =\|W(K_{u,v}-K_N)\|
 \le\|K_{u,v}-K_N\|.
\]
This proves \eqref{eq:soft-defect-envelope}.  Finally, the eigenvalues of
$S_{u,v}$ in $(\epsilon,1-\epsilon)$ correspond exactly to singular values of
$\mathcal T_{u,v}$ exceeding $\sqrt{\epsilon(1-\epsilon)}$, while they form a
subset of the eigenvalues of $S_{u,v}$ exceeding $\epsilon$, equivalently the
singular values of $K_{u,v}$ exceeding $\sqrt\epsilon$.  Since
$\sqrt{\epsilon(1-\epsilon)}<\sqrt\epsilon$, the last assertion follows.
\end{proof}

\subsection{Unbounded windows controlled by moments}

It is convenient to scale the frequency variable without changing the
Hilbert-space norm.  For $c>0$, define
\begin{equation}
 (\mathcal F_cf)(\eta):=\sqrt c
 \int_{\mathbb R}e^{-2\pi i c x\eta}f(x)\,dx.
 \label{eq:scaled-fourier}
\end{equation}
The change of variables $\xi=c\eta$ and Plancherel's theorem show that
$\mathcal F_c$ is unitary.  Moreover, $P_B\mathcal F_cP_A$ is unitarily
equivalent to $P_{cB}\mathcal FP_A$.

For a nonnegative weight $w\in L^1(\mathbb R)$ and a centre $a\in\mathbb R$,
write
\[
 M_N(w;a):=\int_{\mathbb R}|x-a|^{2N}w(x)\,dx.
\]

\begin{theorem}[Moment envelope for hard and soft unbounded windows]
\label{thm:moment-hard}
Let $N\in\mathbb N_0$, $c>0$, and let $0\le u,v\le1$ belong to
$L^1(\mathbb R)$.  Suppose that, for some $\alpha,\beta\in\mathbb R$,
\[
 M_N(u;\alpha)<\infty,
 \qquad M_N(v;\beta)<\infty.
\]
For
\[
 K_{u,v,c}:=M_{\sqrt v}\mathcal F_cM_{\sqrt u},
 \qquad S_{u,v,c}:=K_{u,v,c}^*K_{u,v,c},
\]
one has
\begin{equation}
 s_{N+1}(K_{u,v,c})
 \le
 \frac{(2\pi c)^N\sqrt c}{N!}
 \bigl(M_N(u;\alpha)M_N(v;\beta)\bigr)^{1/2}.
 \label{eq:soft-moment-envelope}
\end{equation}
In particular, if $A,B$ are arbitrary measurable sets of finite measure and
their displayed moments are finite, then
\begin{equation}
 s_{N+1}(P_{cB}\mathcal FP_A)
 \le
 \frac{(2\pi c)^N\sqrt c}{N!}
 \left(
   \int_A|x-\alpha|^{2N}\,dx
   \int_B|\eta-\beta|^{2N}\,d\eta
 \right)^{1/2}.
 \label{eq:hard-moment-envelope}
\end{equation}
For $0<\epsilon<1/2$, whenever the right side of
\eqref{eq:soft-moment-envelope} is at most $\sqrt\epsilon$, it follows that
$\Lambda_\epsilon(S_{u,v,c})\le N$.
\end{theorem}

\begin{proof}
The hard statement follows from the soft statement by setting
$u=\mathbf1_A$ and $v=\mathbf1_B$ and using the unitary equivalence following
\eqref{eq:scaled-fourier}.  We therefore prove the soft statement.

\emph{Step 1: lower moments are automatically finite.}
If $N\ge1$, the case $k=0$ is the assumed integrability of $u$.  For
$1\le k<N$, H\"older's inequality gives
\[
 \int |x-\alpha|^{2k}u(x)\,dx
 \le \left(\int u\right)^{1-k/N}
      M_N(u;\alpha)^{k/N}<\infty,
\]
and similarly for $v$.  Hence all separated functions appearing in the
degree-$(N-1)$ Taylor polynomial below lie in the required weighted
$L^2$ spaces.  For $N=0$ no such terms occur.

\emph{Step 2: remove one-variable phases and truncate the interaction.}
The integral kernel of $K_{u,v,c}$ is
\[
 \sqrt c\,\sqrt{v(\eta)}e^{-2\pi icx\eta}\sqrt{u(x)}.
\]
The identity
\[
 x\eta=(x-\alpha)(\eta-\beta)+\alpha\eta+\beta x-\alpha\beta
\]
separates all but the interaction term.  Replace
$e^{-2\pi ic(x-\alpha)(\eta-\beta)}$ by its Taylor polynomial of degree
$N-1$.  The resulting operator $K_N$ has rank at most $N$.
Lemma~\ref{lem:imaginary-taylor} gives, almost everywhere,
\[
 |k(\eta,x)-k_N(\eta,x)|
 \le
 \sqrt c\,\sqrt{v(\eta)u(x)}
 \frac{(2\pi c)^N
       |x-\alpha|^N|\eta-\beta|^N}{N!}.
\]

\emph{Step 3: integrate the product remainder.}
Squaring the last inequality and applying Tonelli's theorem yields
\begin{align*}
 \|K_{u,v,c}-K_N\|_{\mathrm{HS}}^2
 &\le \frac{c(2\pi c)^{2N}}{(N!)^2}
 \left(\int |x-\alpha|^{2N}u(x)\,dx\right)
 \left(\int |\eta-\beta|^{2N}v(\eta)\,d\eta\right)\\
 &=\frac{c(2\pi c)^{2N}}{(N!)^2}
   M_N(u;\alpha)M_N(v;\beta).
\end{align*}
Operator norm is no larger than Hilbert--Schmidt norm, and $K_N$ has rank at
most $N$.  The approximation-number formula proves
\eqref{eq:soft-moment-envelope}.  At $N=0$, the same argument means simply
that the zero approximant gives
$s_1(K_{u,v,c})\le\sqrt{c(\int u)(\int v)}$.

\emph{Step 4: count.}
Since $K_{u,v,c}$ is a contraction and
$S_{u,v,c}=K_{u,v,c}^*K_{u,v,c}$, every eigenvalue of $S_{u,v,c}$ larger
than $\epsilon$ is the square of a singular value of $K_{u,v,c}$ larger than
$\sqrt\epsilon$.  If the right side of
\eqref{eq:soft-moment-envelope} is at most $\sqrt\epsilon$, at most $N$
such singular values exist.  The plunge interval is a subset of
$(\epsilon,1]$, so $\Lambda_\epsilon(S_{u,v,c})\le N$.
\end{proof}

The next lemma combines bounded-core approximation with an explicit soft
tail.  Its main point is that the tail and the core approximation residual
are orthogonal in Hilbert--Schmidt space.  Pythagoras therefore replaces the
former triangle-inequality sum by a root-sum-square bound.  It is useful when
global high moments are inconvenient.

\begin{lemma}[Bounded core with an orthogonal tail and best polynomial residual]
\label{lem:soft-truncation-tail}
Let $u,v\in L^1(\mathbb R)$ take values in $[0,1]$, let $E,F$ be measurable,
and put
\[
 u_E=u\mathbf1_E,\qquad v_F=v\mathbf1_F,
 \qquad
 U=\int_{\mathbb R}u,\quad U_E=\int_Eu,
 \quad V=\int_{\mathbb R}v,\quad V_F=\int_Fv.
\]
Then the truncation error has the exact Hilbert--Schmidt norm
\begin{align}
 \left\|K_{u,v,c}-K_{u_E,v_F,c}\right\|_{\mathrm{HS}}^2
 &=c\{V(U-U_E)+(V-V_F)U_E\}\notag\\
 &=c(UV-U_EV_F).
 \label{eq:exact-soft-tail}
\end{align}
Thus the previous tail notation simplifies exactly to
\begin{equation}
 \delta_{E,F}^2=c(UV-U_EV_F).
 \label{eq:core-tail-delta-simplified}
\end{equation}

Suppose in addition that $E\cap\operatorname{ess\,supp}u$ and
$F\cap\operatorname{ess\,supp}v$ are essentially bounded.  If $U_EV_F>0$,
let $\alpha,\beta$ be essential centres of these two sets and put
\[
 \zeta_{E,F}:=\frac{\pi c}{2}
 \operatorname{diam}_{\mathrm{ess}}
   (E\cap\operatorname{ess\,supp}u)
 \operatorname{diam}_{\mathrm{ess}}
   (F\cap\operatorname{ess\,supp}v).
\]
Then, for every $N\in\mathbb N_0$,
\begin{equation}
 s_{N+1}(K_{u,v,c})
 \le
 \left\{
 c(UV-U_EV_F)
 +cU_EV_F\left(\frac{\zeta_{E,F}^N}{N!}\right)^2
 \right\}^{1/2}.
 \label{eq:core-tail-envelope}
\end{equation}
At $N=0$ the convention $\zeta_{E,F}^0=1$ is used.  If $U_EV_F=0$, the
core operator vanishes, the second term in
\eqref{eq:core-tail-envelope} is set equal to zero, and the resulting bound
is $s_1(K_{u,v,c})\leq\sqrt{cUV}$.

There is a systematically sharper best-polynomial version.  For $N\geq1$
and arbitrary centres $\alpha,\beta\in\mathbb R$, define
\begin{equation}
 \begin{aligned}
 \mathfrak e_N(\alpha,\beta)^2
 &:=\frac1{U_EV_F}\inf_{p\in\mathcal P_{N-1}}
 \int_E\!\int_F
 \left|e^{-2\pi ic(x-\alpha)(\eta-\beta)}\right.\\[-1mm]
 &\hspace{42mm}\left.
 -p\bigl((x-\alpha)(\eta-\beta)\bigr)\right|^2
 u(x)v(\eta)\,d\eta\,dx,
 \end{aligned}
 \label{eq:best-core-polynomial-error}
\end{equation}
where $\mathcal P_{N-1}$ is the space of complex polynomials of degree at
most $N-1$, and put $\mathfrak e_0=1$.  Then
\begin{equation}
 s_{N+1}(K_{u,v,c})
 \le
 \left\{
 c(UV-U_EV_F)+cU_EV_F\mathfrak e_N(\alpha,\beta)^2
 \right\}^{1/2}.
 \label{eq:best-polynomial-core-tail}
\end{equation}
For the essential centres used above,
\begin{equation}
 \mathfrak e_N(\alpha,\beta)
 \leq \frac{\zeta_{E,F}^N}{N!}.
 \label{eq:best-polynomial-beats-taylor}
\end{equation}
If $U_EV_F=0$, all terms containing $U_EV_F\mathfrak e_N^2$ are understood
to be zero.  Thus \eqref{eq:best-polynomial-core-tail} contains
\eqref{eq:core-tail-envelope} as a computable fallback.
\end{lemma}

\begin{proof}
We give the support orthogonality, rank, and approximation steps explicitly.

\emph{Step 1: split the full kernel into its core and tail.}
The kernel of $K_{u,v,c}$, in the scaled frequency variable $\eta$, is
\[
 k(\eta,x)=\sqrt c\,\sqrt{v(\eta)}
 e^{-2\pi icx\eta}\sqrt{u(x)}.
\]
Since $\sqrt{u_E}=\sqrt u\,\mathbf1_E$ and similarly for $v_F$, the core
kernel is
\[
 k_{E,F}(\eta,x)=k(\eta,x)\mathbf1_E(x)\mathbf1_F(\eta).
\]
Equivalently, disjointness of $E$ and $E^c$ gives the operator decomposition
\begin{align}
 K_{u,v,c}-K_{u_E,v_F,c}
 &=M_{\sqrt v}\mathcal F_cM_{\sqrt{u-u_E}}
   +M_{\sqrt{v-v_F}}\mathcal F_cM_{\sqrt{u_E}}.
 \label{eq:tail-decomposition}
\end{align}
The two kernels on the right have disjoint $x$-supports.  More directly,
$k-k_{E,F}$ is supported, up to a null set, in $(E\times F)^c$.  Hence
\begin{align*}
 &\|K_{u,v,c}-K_{u_E,v_F,c}\|_{\mathrm{HS}}^2\\[-1mm]
 &\qquad=c\int_{(E\times F)^c}u(x)v(\eta)\,d\eta\,dx\\
 &\qquad=c\left(
       \int_{E^c}u(x)\,dx\int_{\mathbb R}v(\eta)\,d\eta
      +\int_Eu(x)\,dx\int_{F^c}v(\eta)\,d\eta
    \right)\\
 &\qquad=c\{V(U-U_E)+(V-V_F)U_E\}.
\end{align*}
Expanding the braces cancels the two terms $VU_E$:
\[
 VU-VU_E+VU_E-V_FU_E=UV-U_EV_F.
\]
This proves \eqref{eq:exact-soft-tail} and
\eqref{eq:core-tail-delta-simplified}.

\emph{Step 2: construct the rank-$N$ core Taylor approximant.}
Assume first that $U_EV_F>0$.  The identity
\[
 x\eta=(x-\alpha)(\eta-\beta)+\alpha\eta+\beta x-\alpha\beta
\]
separates the phase into three one-variable or constant unimodular factors
and the interaction $e^{-2\pi ic(x-\alpha)(\eta-\beta)}$.  Replace this
interaction by
\[
 q_{N-1}(x,\eta)
 =\sum_{j=0}^{N-1}
   \frac{(-2\pi ic)^j}{j!}(x-\alpha)^j(\eta-\beta)^j
\]
when $N\geq1$, and use the zero kernel when $N=0$.  Every summand is a
product of a function of $x$ and a function of $\eta$; after the separated
phase factors and weights are restored, the resulting operator $R_N$ has
rank at most $N$.

For the essential centres specified in the statement, almost every
$(x,\eta)\in E\times F$ with $u(x)v(\eta)>0$ satisfies
\[
 2\pi c|x-\alpha||\eta-\beta|\leq\zeta_{E,F}.
\]
Lemma~\ref{lem:imaginary-taylor} therefore gives
\begin{equation}
 \|K_{u_E,v_F,c}-R_N\|_{\mathrm{HS}}^2
 \leq cU_EV_F
       \left(\frac{\zeta_{E,F}^N}{N!}\right)^2.
 \label{eq:core-taylor-HS-square}
\end{equation}
For $N=0$, this is the exact identity
$\|K_{u_E,v_F,c}\|_{\mathrm{HS}}^2=cU_EV_F$.

\emph{Step 3: use Hilbert--Schmidt orthogonality rather than the triangle
inequality.}
Write
\[
 K_{u,v,c}-R_N
 =\bigl(K_{u,v,c}-K_{u_E,v_F,c}\bigr)
  +\bigl(K_{u_E,v_F,c}-R_N\bigr).
\]
The first kernel is supported in $(E\times F)^c$, whereas the second is
supported in $E\times F$.  Their Hilbert--Schmidt inner product is therefore
zero.  Pythagoras, \eqref{eq:exact-soft-tail}, and
\eqref{eq:core-taylor-HS-square} imply
\[
 \|K_{u,v,c}-R_N\|_{\mathrm{HS}}^2
 \leq c(UV-U_EV_F)
    +cU_EV_F\left(\frac{\zeta_{E,F}^N}{N!}\right)^2.
\]
Since $\|T\|\leq\|T\|_{\mathrm{HS}}$ and
\[
 s_{N+1}(K_{u,v,c})
 =\inf_{\operatorname{rank}R\leq N}\|K_{u,v,c}-R\|,
\]
the last display proves \eqref{eq:core-tail-envelope}.  If $U_EV_F=0$,
the core kernel is zero almost everywhere and the same argument reduces to
$s_1(K_{u,v,c})\leq\|K_{u,v,c}\|_{\mathrm{HS}}=\sqrt{cUV}$.

\emph{Step 4: replace Taylor's polynomial by the weighted best polynomial.}
Fix $N\geq1$, and let $p(z)=\sum_{j=0}^{N-1}a_jz^j$.  Restoring the separated
phase factors turns $p((x-\alpha)(\eta-\beta))$ into a kernel that is a sum
of at most $N$ separated products.  The associated operator $R_p$ therefore
has rank at most $N$.  By \eqref{eq:best-core-polynomial-error},
\[
 \inf_{p\in\mathcal P_{N-1}}
 \|K_{u_E,v_F,c}-R_p\|_{\mathrm{HS}}^2
 =cU_EV_F\mathfrak e_N(\alpha,\beta)^2.
\]
The infimum is attained because the polynomials of degree at most $N-1$
form a finite-dimensional, hence closed, subspace of the relevant weighted
$L^2$ space after null polynomials are factored out.  For a minimizing
polynomial, the same disjoint-support Pythagorean identity from Step~3 gives
\eqref{eq:best-polynomial-core-tail}.

Finally, the Taylor polynomial $q_{N-1}$ is one of the competitors in
\eqref{eq:best-core-polynomial-error}.  Its pointwise error on the centred
core is at most $\zeta_{E,F}^N/N!$.  Dividing its squared weighted error by
$U_EV_F$ proves \eqref{eq:best-polynomial-beats-taylor}.
\end{proof}

\begin{remark}[Finite moment-matrix evaluation of the best core residual]
\label{rem:core-moment-matrix}
The quantity in \eqref{eq:best-core-polynomial-error} is obtained from an
$N\times N$ positive semidefinite Gram matrix.  Put
$z=(x-\alpha)(\eta-\beta)$ and, for $0\leq j,k<N$, define
\begin{align*}
 G_{jk}
 &=\int_E\!\int_F z^{j+k}u(x)v(\eta)\,d\eta\,dx\\
 &=\left(\int_E(x-\alpha)^{j+k}u(x)\,dx\right)
   \left(\int_F(\eta-\beta)^{j+k}v(\eta)\,d\eta\right),\\
 b_j
 &=\int_E\!\int_F
 z^j e^{-2\pi icz}u(x)v(\eta)\,d\eta\,dx.
\end{align*}
Let $G^\dagger$ denote the Moore--Penrose pseudoinverse.  Then
\begin{equation}
 \mathfrak e_N(\alpha,\beta)^2
 =1-\frac{b^*G^\dagger b}{U_EV_F}.
 \label{eq:best-core-moment-matrix}
\end{equation}
Indeed, in
$L^2(E\times F,u(x)v(\eta)\,dx\,d\eta)$, $G$ is the Gram matrix of
$1,z,\ldots,z^{N-1}$ and $b$ is the vector of their inner products with
$e^{-2\pi icz}$.  The normal equations are $Ga=b$.  If $q\in\ker G$, then
$\sum_jq_jz^j=0$ almost everywhere and hence $q^*b=0$; consequently
$b\in\operatorname{ran}G$ even when $G$ is singular.  The orthogonal
projection coefficients may be taken as $a=G^\dagger b$, and the projection
identity gives
\[
 \inf_a\left\|e^{-2\pi icz}
              -\sum_{j=0}^{N-1}a_jz^j\right\|_{L^2(uv)}^2
 =U_EV_F-b^*G^\dagger b.
\]
This proves \eqref{eq:best-core-moment-matrix}.  Thus the centres and
truncation sets may be optimized numerically while retaining a rigorous
rank-$N$ certificate.
\end{remark}

\begin{remark}[Comparison with the triangle-inequality core--tail estimate]
Let
\[
 a=c(UV-U_EV_F),\qquad
 b=cU_EV_F\left(\frac{\zeta_{E,F}^N}{N!}\right)^2.
\]
The corresponding triangle-inequality estimate is $\sqrt a+\sqrt b$, whereas
\eqref{eq:core-tail-envelope} is $\sqrt{a+b}$.  Since
$\sqrt{a+b}\leq\sqrt a+\sqrt b$, the new estimate is never worse and is
strictly better whenever $a,b>0$.  At $N=0$ it recovers exactly the full
Hilbert--Schmidt bound $\sqrt{cUV}$.
\end{remark}

\begin{corollary}[Residual-energy core--tail count]
\label{cor:core-tail-residual-energy}
With the notation of Lemma~\ref{lem:soft-truncation-tail}, fix
$0<\epsilon<1/2$ and define
\[
 \mathcal E_{N,E,F}^2
 :=c(UV-U_EV_F)+cU_EV_F\mathfrak e_N(\alpha,\beta)^2.
\]
Then
\begin{equation}
 \Lambda_\epsilon(S_{u,v,c})
 \leq
 \min_{N\in\mathbb N_0}
 \left\{N+\kappa\!\left(
              \frac{\mathcal E_{N,E,F}^2}{\epsilon}
                         \right)\right\}.
 \label{eq:core-tail-residual-energy-count}
\end{equation}
In particular, when $\alpha,\beta$ are the essential centres specified in
Lemma~\ref{lem:soft-truncation-tail}, replacing $\mathfrak e_N$ by
$\zeta_{E,F}^N/N!$ gives a fully explicit bound.
\end{corollary}

\begin{proof}
The proof of Lemma~\ref{lem:soft-truncation-tail} constructs a rank-$N$
operator $R_N$ satisfying
\[
 \|K_{u,v,c}-R_N\|_{\mathrm{HS}}^2
 \leq\mathcal E_{N,E,F}^2.
\]
The operator $K_{u,v,c}$ is a contraction because the two multiplication
operators are contractions and $\mathcal F_c$ is unitary.  Apply the direct
branch of Theorem~\ref{thm:sharp-residual-master} and minimize over $N$.
If $\mathcal E_{N,E,F}^2\leq\epsilon$, then
$\kappa(\mathcal E_{N,E,F}^2/\epsilon)=0$, so the corollary recovers the
rank-$N$ threshold rule.  It can improve that rule by permitting residual
energy above $\epsilon$ and charging only the exact strict-count correction
$\kappa$.
\end{proof}

\subsection{An unbounded soft example with a sublogarithmic certificate}

The following example lies genuinely outside every bounded-support theorem.

\begin{example}[Super-Gaussian soft masks]
\label{ex:supergaussian}
Let
\[
 u(x)=v(x)=e^{-x^4},
 \qquad K_c=M_{e^{-x^4/2}}\mathcal F_cM_{e^{-x^4/2}},
 \qquad S_c=K_c^*K_c.
\]
For every $N\in\mathbb N_0$,
\begin{equation}
 s_{N+1}(K_c)
 \le \sqrt c\,\frac{(2\pi c)^N}{N!}
       \frac12\Gamma\!\left(\frac N2+\frac14\right).
 \label{eq:supergaussian-singular}
\end{equation}
For each fixed $c>0$ this implies
\begin{equation}
 \log s_{N+1}(K_c)
 \le-\frac12N\log N+O_c(N),
 \label{eq:supergaussian-log}
\end{equation}
and, as $\epsilon\downarrow0$,
\begin{equation}
 \Lambda_\epsilon(S_c)
 \le (1+o(1))
 \frac{\log(1/\epsilon)}{\log\log(1/\epsilon)}.
 \label{eq:supergaussian-count}
\end{equation}
Here the last display is an upper certificate; it does not claim a matching
lower bound.
\end{example}

\begin{proof}
By symmetry and the substitution $t=x^4$,
\begin{align*}
 M_N(u;0)
 &=2\int_0^\infty x^{2N}e^{-x^4}\,dx\\
 &=\frac12\int_0^\infty t^{N/2-3/4}e^{-t}\,dt
 =\frac12\Gamma\!\left(\frac N2+\frac14\right).
\end{align*}
The two moments in \eqref{eq:soft-moment-envelope} are equal, so their
geometric mean is this same number.  This proves
\eqref{eq:supergaussian-singular}.

Stirling's formula gives
\[
 \log\Gamma\!\left(\frac N2+\frac14\right)
 =\frac N2\log N+O(N),
 \qquad
 \log(N!)=N\log N-N+O(\log N).
\]
Taking logarithms in \eqref{eq:supergaussian-singular}, and absorbing
$N\log(2\pi c)$ and all lower-order terms into $O_c(N)$, proves
\eqref{eq:supergaussian-log}.

Finally set $L=\log(1/\epsilon)$.  The meaning of
\eqref{eq:supergaussian-log} is that there are constants $C_c$ and $N_c$ such
that
\[
 \log s_{N+1}(K_c)\le-\frac12N\log N+C_cN
 \qquad(N\ge N_c).
\]
Choose
\[
 \delta_L:=\frac{3(\log\log L+2C_c+1)}{\log L},
 \qquad
 N=\left\lceil(1+\delta_L)\frac{L}{\log L}\right\rceil.
\]
Here $\delta_L\to0$ and
$\log N=\log L-\log\log L+o(\log L)$.  Directly dividing the next left-hand
side by $L$ shows that
\[
 \frac{N\log N-2C_cN}{L}
 \ge (1+\delta_L)
 \left(1-\frac{\log\log L+2C_c+o(1)}{\log L}\right)>1
\]
for all sufficiently large $L$.  Thus
$s_{N+1}(K_c)\le e^{-L/2}=\sqrt\epsilon$.  The direct singular-value count
gives $\Lambda_\epsilon(S_c)\le N$.  Because $\delta_L\to0$, this is
\eqref{eq:supergaussian-count}.
\end{proof}

\begin{remark}[What changes, and what does not]
The measurable-window result removes interval regularity but still needs a
bounded essential hull.  The moment result removes boundedness but asks for a
finite moment at the approximation order being used.  Neither step changes the
spectral thresholds: the direct route still uses $\sqrt\epsilon$, and the
exact defect route still uses $\sqrt{\epsilon(1-\epsilon)}$.  What improves is
the available singular-value envelope, and therefore the boundary function
obtained after solving for $N$.
\end{remark}

\section{Finite partitions, local geometry, and threshold allocation}
\label{sec:partitions}

The one-piece Taylor estimate can be assembled without first replacing a
measurable set by its enclosing interval.  The correct assembly norm is the
operator norm of a finite scalar error matrix.  This gives a direct way to
retain local diameters, local measures, and empty gaps.  A second, independent
assembly uses the defect operator and a weighted Ky--Fan inequality.  We give
both arguments in detail because they encode different losses, and we retain
both certificates.

Throughout this section, singular values are numbered in nonincreasing order,
and
\[
 n(\tau;L):=\#\{k:s_k(L)>\tau\},\qquad \tau>0,
\]
uses a strict inequality.  In every Fourier-localization statement below,
$c>0$.  If either global window has measure zero, then
$K_{A,B,c}=0$ and every associated plunge count and partition envelope is
understood to be zero.  Otherwise, null pieces in measurable partitions are
discarded.

\subsection{A scalar error matrix for finite operator partitions}

We begin with a Hilbert-space statement that is not specific to the Fourier
kernel.

\begin{theorem}[Finite block approximation]
\label{thm:finite-block-matrix}
Let
\[
 \mathcal H=\bigoplus_{i=1}^{M}\mathcal H_i,
 \qquad
 \mathcal G=\bigoplus_{j=1}^{K}\mathcal G_j,
\]
and let $L:\mathcal H\to\mathcal G$ be a compact operator with blocks
$L_{ji}:\mathcal H_i\to\mathcal G_j$.  Suppose that, for every $(j,i)$,
there is an operator $R_{ji}$ such that
\[
 \operatorname{rank}R_{ji}\leq r_{ji},
 \qquad
 \|L_{ji}-R_{ji}\|\leq E_{ji},
\]
where $r_{ji}\in\mathbb N_0$ and $E_{ji}\geq0$.  Let
\(
 \mathbf E=(E_{ji})_{1\leq j\leq K,\,1\leq i\leq M}
\)
act from $\ell^2_M$ to $\ell^2_K$.  Then
\begin{equation}
 \label{eq:block-matrix-sv}
 s_{1+\sum_{j,i}r_{ji}}(L)
 \leq \|\mathbf E\|_{\ell^2_M\to\ell^2_K}.
\end{equation}
Equivalently, if $\|\mathbf E\|\leq\tau$, then
\begin{equation}
 \label{eq:block-matrix-count}
 n(\tau;L)\leq\sum_{j,i}r_{ji}.
\end{equation}
\end{theorem}

\begin{proof}
Form the block operator $R:\mathcal H\to\mathcal G$ whose $(j,i)$ block is
$R_{ji}$.  Each embedded block
\(
 \iota_jR_{ji}\pi_i
\)
has rank at most $r_{ji}$, where $\pi_i$ is the $i$th coordinate projection
and $\iota_j$ is the $j$th coordinate inclusion.  Since
\[
 R=\sum_{j=1}^{K}\sum_{i=1}^{M}\iota_jR_{ji}\pi_i,
\]
subadditivity of rank gives
\begin{equation}
 \label{eq:block-rank-sum}
 \operatorname{rank}R\leq\sum_{j,i}r_{ji}=:r.
\end{equation}

It remains to estimate $L-R$.  Take
$f=(f_1,\ldots,f_M)\in\mathcal H$ and set
\(
 x_i=\|f_i\|_{\mathcal H_i}.
\)
The $j$th component of $(L-R)f$ is
\[
 ((L-R)f)_j=\sum_{i=1}^{M}(L_{ji}-R_{ji})f_i.
\]
The triangle inequality and the assumed block estimates imply
\begin{equation}
 \label{eq:block-coordinate-bound}
 \|((L-R)f)_j\|_{\mathcal G_j}
 \leq\sum_{i=1}^{M}E_{ji}x_i=(\mathbf E x)_j.
\end{equation}
After squaring, summing over $j$, and using
$\|x\|_{\ell^2_M}=\|f\|_{\mathcal H}$, we obtain
\[
 \|(L-R)f\|_{\mathcal G}^2
 \leq\|\mathbf E x\|_{\ell^2_K}^2
 \leq\|\mathbf E\|^2\|f\|_{\mathcal H}^2.
\]
Thus $\|L-R\|\leq\|\mathbf E\|$.  The approximation-number
characterization of singular values now yields
\[
 s_{r+1}(L)
 =\inf_{\operatorname{rank}G\leq r}\|L-G\|
 \leq\|L-R\|
 \leq\|\mathbf E\|,
\]
which is \eqref{eq:block-matrix-sv}.  If $\|\mathbf E\|\leq\tau$, then
$s_{r+1}(L)\leq\tau$; because $n(\tau;L)$ counts only singular values
strictly larger than $\tau$, at most the first $r$ singular values can be
counted.  This proves \eqref{eq:block-matrix-count}.
\end{proof}

We apply the theorem to finite measurable partitions
\[
 A=\mathop{\dot\bigcup}_{i=1}^{M}A_i,
 \qquad
 B=\mathop{\dot\bigcup}_{j=1}^{K}B_j
\]
of bounded sets of finite measure.  Write
\[
 K_{A,B,c}=P_{cB}\mathcal F P_A:L^2(A)\longrightarrow L^2(cB).
\]
The orthogonal decompositions of its domain and range identify its
$(j,i)$ block with
\[
 K_{ji}=P_{cB_j}\mathcal F P_{A_i}.
\]
Let
\[
 d_i=\operatorname{diam}_{\mathrm{ess}}(A_i),
 \qquad
 e_j=\operatorname{diam}_{\mathrm{ess}}(B_j),
 \qquad
 \zeta_{ji}=\frac{\pi c}{2}d_i e_j.
\]

\begin{corollary}[Locally centered Taylor assembly]
\label{cor:local-taylor-matrix}
For arbitrary integers $N_{ji}\geq0$, define the scalar matrix
\begin{equation}
 \label{eq:local-error-matrix}
 \mathbf E(\mathbf N)_{ji}
 :=\sqrt{c\,|A_i|\,|B_j|}\,
   \frac{\zeta_{ji}^{N_{ji}}}{N_{ji}!}.
\end{equation}
At order zero we use the convention $\zeta_{ji}^0=1$.
Then
\begin{equation}
 \label{eq:local-taylor-global-sv}
 s_{1+\sum_{j,i}N_{ji}}(K_{A,B,c})
 \leq \|\mathbf E(\mathbf N)\|_{\ell^2_M\to\ell^2_K}.
\end{equation}
Consequently, for $0<\epsilon<1/2$,
\begin{equation}
 \label{eq:matrix-plunge-bound}
 \|\mathbf E(\mathbf N)\|\leq\sqrt\epsilon
 \quad\Longrightarrow\quad
 \Lambda_\epsilon(A,cB)\leq\sum_{j,i}N_{ji}.
\end{equation}
\end{corollary}

\begin{proof}
Fix $(j,i)$.  Let $\alpha_i$ and $\beta_j$ be the midpoints of the essential
convex hulls of $A_i$ and $B_j$, respectively.  For
$x\in A_i$ and $\xi\in cB_j$, modulo null sets,
\[
 |x-\alpha_i|\leq\frac{d_i}{2},
 \qquad
 |\xi-c\beta_j|\leq\frac{ce_j}{2}.
\]
The identity
\begin{equation}
 \label{eq:block-phase-splitting}
 x\xi
 =\alpha_i\xi+c\beta_jx-\alpha_i c\beta_j
  +(x-\alpha_i)(\xi-c\beta_j)
\end{equation}
separates all dependence on the absolute locations of the two pieces into
one-variable unimodular factors.  Expand only the last interaction:
\begin{equation}
 \label{eq:block-taylor-kernel}
 e^{-2\pi i(x-\alpha_i)(\xi-c\beta_j)}
 =\sum_{k=0}^{N_{ji}-1}
   \frac{(-2\pi i)^k}{k!}(x-\alpha_i)^k(\xi-c\beta_j)^k
  +\mathcal R_{ji}(x,\xi).
\end{equation}
The sum is empty when $N_{ji}=0$.  Each displayed summand is a product of a
function of $x$ and a function of $\xi$; multiplication by the separated
phases from \eqref{eq:block-phase-splitting} does not change that fact.
Therefore the corresponding integral operator $R_{ji}$ has rank at most
$N_{ji}$.

The integral form of the exponential Taylor remainder gives, for real $t$,
\(
 |e^{it}-\sum_{k=0}^{N-1}(it)^k/k!|\leq |t|^N/N!
\)
when $N\geq1$; for $N=0$ the same assertion reads $|e^{it}|\leq1$.
Here
\[
 |2\pi(x-\alpha_i)(\xi-c\beta_j)|
 \leq 2\pi\frac{d_i}{2}\frac{ce_j}{2}
 =\zeta_{ji}.
\]
Hence the remainder kernel is bounded pointwise by
$\zeta_{ji}^{N_{ji}}/N_{ji}!$.  Its Hilbert--Schmidt norm is consequently at
most
\[
 \bigl(|A_i|\,|cB_j|\bigr)^{1/2}
 \frac{\zeta_{ji}^{N_{ji}}}{N_{ji}!}
 =\sqrt{c|A_i||B_j|}\,
 \frac{\zeta_{ji}^{N_{ji}}}{N_{ji}!}
 =\mathbf E(\mathbf N)_{ji}.
\]
The operator norm is bounded by the Hilbert--Schmidt norm, so
Theorem~\ref{thm:finite-block-matrix} proves
\eqref{eq:local-taylor-global-sv}.

Finally let $S=K_{A,B,c}^*K_{A,B,c}$.  If an eigenvalue $\lambda$ of $S$
belongs to $(\epsilon,1-\epsilon)$, then the corresponding singular value of
$K_{A,B,c}$ is $\sqrt\lambda>\sqrt\epsilon$.  Thus
\[
 \Lambda_\epsilon(A,cB)
 \leq n(\sqrt\epsilon;K_{A,B,c}).
\]
The last assertion now follows from
\eqref{eq:local-taylor-global-sv} and the strict definition of $n$.
\end{proof}

Corollary~\ref{cor:local-taylor-matrix} naturally produces the integer-valued
partition envelope
\begin{equation}
 \label{eq:partition-matrix-envelope}
 \mathfrak B_{\mathrm{mat}}(\epsilon;A,B,c)
 :=\inf_{\substack{
   A=\dot\bigcup_i A_i,\ B=\dot\bigcup_j B_j\\
   \text{finite measurable partitions}\\
   N_{ji}\in\mathbb N_0,\
   \|\mathbf E(\mathbf N)\|\leq\sqrt\epsilon}}
   \sum_{j,i}N_{ji}.
\end{equation}
The trivial one-piece partitions are admissible, so this optimization can
never be worse than the corresponding one-piece Taylor certificate.
Refinement is optional rather than automatically advantageous: smaller local
diameters reduce the entries of $\mathbf E$, while the sum of the local ranks
can increase.  The infimum in \eqref{eq:partition-matrix-envelope} makes this
tradeoff explicit.

\subsection{A gap-stable consequence}

The phase splitting in \eqref{eq:block-phase-splitting} makes the next useful
feature immediate.

\begin{corollary}[Separated copies]
\label{cor:gap-stable-copies}
Let $E,F\subset\mathbb R$ be bounded measurable sets with positive measures
$a=|E|$ and $b=|F|$, and essential diameters $d_E$ and $d_F$.  If
$R>d_E$, put, modulo null sets,
\[
 A_R=E\,\dot\cup\,(R+E),
 \qquad B=F,
 \qquad
 \zeta_0=\frac{\pi c}{2}d_Ed_F.
\]
Then, for every $N\in\mathbb N_0$,
\begin{equation}
 \label{eq:gap-stable-sv}
 s_{2N+1}(P_{cF}\mathcal F P_{A_R})
 \leq \sqrt{2cab}\,\frac{\zeta_0^N}{N!}.
\end{equation}
In particular, for $0<\epsilon<1/2$, the condition
\begin{equation}
 \label{eq:gap-stable-condition}
 \sqrt{2cab}\,\frac{\zeta_0^N}{N!}\leq\sqrt\epsilon
\end{equation}
implies
\(
 \Lambda_\epsilon(A_R,cF)\leq2N,
\)
uniformly in the separation $R$.
\end{corollary}

\begin{proof}
Use the two spatial pieces $E$ and $R+E$ and the single frequency piece $F$.
After changing $E$ on a null set, it lies in an interval of length $d_E$;
therefore $R>d_E$ makes these two spatial pieces disjoint.  Translation does
not change measure or essential diameter.  Applying the
local Taylor construction at order $N$ to either block therefore gives the
same error
\[
 e_N=\sqrt{cab}\,\frac{\zeta_0^N}{N!}
\]
and rank at most $N$.  The scalar error matrix is the one-row matrix
$\mathbf E=(e_N\ e_N)$, whose $\ell^2_2\to\ell^2_1$ norm is
$\sqrt2e_N$.  Formula \eqref{eq:gap-stable-sv} follows from
Theorem~\ref{thm:finite-block-matrix}; the plunge-count statement follows as
in the last step of Corollary~\ref{cor:local-taylor-matrix}.
\end{proof}

If one instead uses the convex hull of $A_R$ as a single piece, its diameter
grows like $R$ and the resulting Taylor parameter grows with the empty gap.
Corollary~\ref{cor:gap-stable-copies} avoids that loss.  Componentwise
interval-union Schatten estimates are already gap-stable when the pieces are
intervals.  The point here is different: the same stability now holds for
arbitrary bounded measurable components and simultaneously retains their
local filling measures $a$ and $b$.

\subsection{Sets with the same hull need not have the same plunge count}

The following elementary comparison isolates why measure sensitivity is not
cosmetic.  It is an exact finite-parameter distinction, not merely a change in
an asymptotic constant.

\begin{proposition}[A same-hull separation]
\label{prop:same-hull-separation}
Let $I=J=[-1/2,1/2]$, $c=1/10$, and $\epsilon=1/100$.  There exists a compact
fat Cantor set $C\subset I$ such that
\[
 |C|=\frac1{10},
 \qquad
 \operatorname*{ess\,inf}C=-\frac12,
 \qquad
 \operatorname*{ess\,sup}C=\frac12,
\]
and
\begin{equation}
 \label{eq:same-hull-counts}
 \Lambda_\epsilon(C,cC)=0,
 \qquad
 \Lambda_\epsilon(I,cJ)\geq1.
\end{equation}
Thus the two problems have the same spatial and frequency hulls, but their
plunge counts are already different at the displayed parameters.
\end{proposition}

\begin{proof}
We first construct the set on $[0,1]$.  Start with the interval $C_0=[0,1]$.
At stage $n\geq1$, remove from the centre of each of the $2^{n-1}$ intervals
surviving stage $n-1$ an open interval of length
\[
 a_n:=\frac{0.9}{2^{2n-1}}.
\]
The total length removed at stage $n$ is
$2^{n-1}a_n=0.9/2^n$; hence the total removed length is
$\sum_{n\geq1}0.9/2^n=0.9$.  If $L_n$ denotes the common length of a
stage-$n$ surviving interval, then
\[
 2^nL_n=1-0.9\sum_{k=1}^n2^{-k}
 =0.1+0.9\,2^{-n},
 \qquad
 L_n=\frac{0.1}{2^n}+\frac{0.9}{4^n}\longrightarrow0.
\]
The nested intersection $C_0^*$ of the surviving compact sets is compact,
has measure $0.1$, has no isolated points, and is nowhere dense.  Symmetry of
the construction shows that each stage-$n$ interval contains exactly
$0.1/2^n$ measure of $C_0^*$.  Thus every relative neighborhood of $0$ and
of $1$ meets $C_0^*$ in positive measure.  For $C=C_0^*-1/2$, it follows that
$\operatorname*{ess\,inf}C=-1/2$,
$\operatorname*{ess\,sup}C=1/2$, and $|C|=1/10$.

For arbitrary finite-measure sets $A,B$, the restricted Fourier operator
$K_{A,B,c}=P_{cB}\mathcal F P_A$ has kernel of modulus one on
$cB\times A$.  Hence
\begin{equation}
 \label{eq:trace-mass-identity}
 \operatorname{tr}(K_{A,B,c}^*K_{A,B,c})
 =\|K_{A,B,c}\|_{\mathcal S_2}^2
 =c|A||B|.
\end{equation}
For $A=B=C$, the right-hand side is
\(
 (1/10)(1/10)^2=1/1000<\epsilon.
\)
All eigenvalues are nonnegative, so no individual eigenvalue can exceed their
sum.  In particular no eigenvalue lies in $(\epsilon,1-\epsilon)$, proving
$\Lambda_\epsilon(C,cC)=0$.

For $A=I$ and $B=J$, use the unit vector $f=\mathbf1_I$ in $L^2(I)$.  With
the Fourier normalization used in this paper,
\[
 \widehat{\mathbf1_I}(\xi)=\frac{\sin(\pi\xi)}{\pi\xi},
\]
with the value at zero understood by continuity.  The Rayleigh principle gives
\begin{equation}
 \label{eq:interval-rayleigh}
 \lambda_1(P_IQ_{cJ}P_I)
 \geq\int_{-1/20}^{1/20}
   \left(\frac{\sin(\pi\xi)}{\pi\xi}\right)^2d\xi.
\end{equation}
For $|\xi|\leq1/20$, the alternating Taylor bound
$\sin u/u\geq1-u^2/6$ gives
\[
 \left(\frac{\sin(\pi\xi)}{\pi\xi}\right)^2
 \geq \left(1-\frac{\pi^2\xi^2}{6}\right)^2
 \geq 1-\frac{\pi^2\xi^2}{3}.
\]
Consequently,
\begin{align*}
 \lambda_1(P_IQ_{cJ}P_I)
 &\geq \frac1{10}-\frac{\pi^2}{3}
       \int_{-1/20}^{1/20}\xi^2\,d\xi \\
 &=\frac1{10}-\frac{\pi^2}{36000}
 >0.099>\epsilon.
\end{align*}
On the other hand, \eqref{eq:trace-mass-identity} gives
\(
 \lambda_1\leq\operatorname{tr}(P_IQ_{cJ}P_I)=1/10<1-\epsilon.
\)
Thus $\lambda_1\in(\epsilon,1-\epsilon)$, and the interval problem has at
least one plunge eigenvalue.
\end{proof}

Proposition~\ref{prop:same-hull-separation} does not assert optimality of the
new upper bounds.  It proves the narrower and necessary point that any bound
depending only on the two enclosing intervals discards data that can decide
even whether the plunge count vanishes.

\subsection{Weighted Ky--Fan allocation for measurable partitions}

The matrix assembly above approximates the restricted Fourier operator
directly.  We now instead decompose its defect operator.  The following simple
strict-count form of the Ky--Fan inequality will be used.

\begin{lemma}[Finite strict-count Ky--Fan inequality]
\label{lem:strict-kyfan}
If $X_1,\ldots,X_q$ are compact operators and $t_1,\ldots,t_q>0$, then
\begin{equation}
 \label{eq:strict-kyfan}
 n\!\left(\sum_{\nu=1}^{q}t_\nu;\sum_{\nu=1}^{q}X_\nu\right)
 \leq\sum_{\nu=1}^{q}n(t_\nu;X_\nu).
\end{equation}
\end{lemma}

\begin{proof}
It is enough to prove the two-operator case and iterate it.  Put
$p=n(a;X)$ and $q=n(b;Y)$.  By the strict definition of $n$,
\(
 s_{p+1}(X)\leq a
\)
and
\(
 s_{q+1}(Y)\leq b.
\)
For every $\delta>0$, the approximation-number characterization gives
operators $X_p,Y_q$ of ranks at most $p,q$ such that
\[
 \|X-X_p\|\leq a+\delta,
 \qquad
 \|Y-Y_q\|\leq b+\delta.
\]
The rank of $X_p+Y_q$ is at most $p+q$, while
\[
 \|(X+Y)-(X_p+Y_q)\|\leq a+b+2\delta.
\]
Letting $\delta\downarrow0$ yields
$s_{p+q+1}(X+Y)\leq a+b$.  Therefore at most $p+q$ singular values of
$X+Y$ are strictly larger than $a+b$.  This is the desired two-operator
inequality, and induction proves \eqref{eq:strict-kyfan}.
\end{proof}

\begin{theorem}[Weighted measurable-partition reduction]
\label{thm:weighted-measurable-partition}
Let $A,B\subset\mathbb R$ have finite measure and let
\[
 A=\mathop{\dot\bigcup}_{i=1}^{M}A_i,
 \qquad
 B=\mathop{\dot\bigcup}_{j=1}^{K}B_j
\]
be arbitrary finite measurable partitions, with null pieces discarded.  Fix
$0<\epsilon<1/2$ and set
\(
 t=\sqrt{\epsilon(1-\epsilon)}.
\)
For any local thresholds $0<\epsilon_{ij}<1/2$ satisfying
\begin{equation}
 \label{eq:weighted-threshold-budget}
 \sum_{i=1}^{M}\sum_{j=1}^{K}
 \sqrt{\epsilon_{ij}(1-\epsilon_{ij})}
 \leq t,
\end{equation}
one has
\begin{equation}
 \label{eq:weighted-measurable-bound}
 \Lambda_\epsilon(A,cB)
 \leq\sum_{i=1}^{M}\sum_{j=1}^{K}
 \Lambda_{\epsilon_{ij}}(A_i,cB_j).
\end{equation}
\end{theorem}

\begin{proof}
Let $P=P_A$, $Q=Q_{cB}$, and
\[
 S=PQP,
 \qquad
 T=P_{A^c}QP_A=(I-P)QP.
\]
Since $P$ and $Q$ are orthogonal projections,
\begin{align}
 T^*T
 &=PQ(I-P)QP \notag\\
 &=PQP-PQPQP
 =S-S^2.                                     \label{eq:global-defect-factor}
\end{align}
If $\lambda\in[0,1]$, then
\[
 \lambda(1-\lambda)>\epsilon(1-\epsilon)
 \quad\Longleftrightarrow\quad
 \epsilon<\lambda<1-\epsilon.
\]
Functional calculus and \eqref{eq:global-defect-factor} therefore give the
exact identity
\begin{equation}
 \label{eq:global-defect-count-partition}
 \Lambda_\epsilon(A,cB)=n(t;T).
\end{equation}

Because the frequency pieces are disjoint,
$Q_{cB}=\sum_jQ_{cB_j}$, and because $P_A=\sum_iP_{A_i}$, we have the finite
decomposition
\begin{equation}
 \label{eq:global-T-block-sum}
 T=\sum_{i=1}^{M}\sum_{j=1}^{K}
    P_{A^c}Q_{cB_j}P_{A_i}.
\end{equation}
Every block in this sum is compact.  Indeed,
$P_{cB_j}\mathcal FP_{A_i}$ is Hilbert--Schmidt with squared norm
$c|A_i||B_j|$, and applying the unitary $\mathcal F^{-1}$ on the output
identifies it with $Q_{cB_j}P_{A_i}$.  Left compression by $P_{A^c}$
preserves compactness.
For each pair define
\[
 \widetilde T_{ij}:=P_{A_i^c}Q_{cB_j}P_{A_i},
 \qquad
 S_{ij}:=P_{A_i}Q_{cB_j}P_{A_i}.
\]
Since $A^c\subset A_i^c$ modulo null sets,
\begin{equation}
 \label{eq:block-compression}
 P_{A^c}Q_{cB_j}P_{A_i}
 =P_{A^c}\widetilde T_{ij}.
\end{equation}
Left multiplication by the projection $P_{A^c}$ cannot increase singular
values.  Moreover, the same projection computation as above gives
\[
 \widetilde T_{ij}^*\widetilde T_{ij}
 =S_{ij}-S_{ij}^2.
\]
Writing
$t_{ij}=\sqrt{\epsilon_{ij}(1-\epsilon_{ij})}$, we consequently have
\begin{equation}
 \label{eq:local-defect-count-partition}
 n(t_{ij};P_{A^c}Q_{cB_j}P_{A_i})
 \leq n(t_{ij};\widetilde T_{ij})
 =\Lambda_{\epsilon_{ij}}(A_i,cB_j).
\end{equation}

Let $t_\Sigma=\sum_{i,j}t_{ij}$.  The budget
\eqref{eq:weighted-threshold-budget} says $t_\Sigma\leq t$, and strict
singular-value counts decrease as the threshold increases.  Combining
\eqref{eq:global-defect-count-partition},
\eqref{eq:global-T-block-sum}, Lemma~\ref{lem:strict-kyfan}, and
\eqref{eq:local-defect-count-partition} gives
\begin{align*}
 \Lambda_\epsilon(A,cB)
 &=n(t;T)
 \leq n(t_\Sigma;T)\\
 &\leq\sum_{i,j}
 n(t_{ij};P_{A^c}Q_{cB_j}P_{A_i})\\
 &\leq\sum_{i,j}\Lambda_{\epsilon_{ij}}(A_i,cB_j),
\end{align*}
which proves the theorem.
\end{proof}

The theorem can be written as an explicit weighted allocation.  Define
\[
 \eta(u)=\frac{1-\sqrt{1-4u^2}}{2},
 \qquad 0<u<\frac12;
\]
then $\sqrt{\eta(u)(1-\eta(u))}=u$.  If weights
$\theta_{ij}>0$ satisfy $\sum_{i,j}\theta_{ij}\leq1$, then
\begin{equation}
 \label{eq:weighted-theta-form}
 \Lambda_\epsilon(A,cB)
 \leq\sum_{i,j}
 \Lambda_{\eta(\theta_{ij}t)}(A_i,cB_j).
\end{equation}
Equal weights $\theta_{ij}=1/(MK)$ are one admissible choice, but optimizing
the weights allows difficult blocks to receive a larger share of the defect
threshold.  This Ky--Fan envelope and the matrix envelope
\eqref{eq:partition-matrix-envelope} should be retained in parallel.  The
former combines exact local defect counts but divides a global threshold
budget; the latter combines direct finite-rank errors through a scalar spectral
norm.  Either can be smaller on a particular configuration.

\subsection{Optimized moments and best-polynomial block matrices}
\label{subsec:optimized-local-polynomials}

The midpoint--diameter entry in \eqref{eq:local-error-matrix} is convenient,
but it makes two avoidable choices: it centers every piece at the midpoint of
its essential hull, and it fixes the coefficients of the approximating
polynomial to be the Taylor coefficients.  Both choices can be optimized
without changing the rank.  We record the resulting hierarchy in a form that
can be inserted directly into the scalar block-matrix argument.

For a measurable set $E\subset\mathbb R$ of finite measure and
$N\in\mathbb N_0$, put
\begin{equation}
 \label{eq:optimized-central-moment}
 \mu_N(E):=\inf_{a\in\mathbb R}\int_E |x-a|^{2N}\,dx.
\end{equation}
Thus $\mu_0(E)=|E|$.  We use the convention
$\mathcal P_{-1}=\{0\}$, while $\mathcal P_{N-1}$ denotes the complex
polynomials of degree at most $N-1$ when $N\geq1$.

Fix finite measurable partitions
\[
 A=\mathop{\dot\bigcup}_{i=1}^{M}A_i,
 \qquad
 B=\mathop{\dot\bigcup}_{j=1}^{K}B_j,
\]
discarding null pieces, and assume in this subsection that all pieces are
bounded.  Given an integer array $\mathbf N=(N_{ji})$ with
$N_{ji}\in\mathbb N_0$, define three nonnegative $K\times M$ matrices.  The
optimized-moment matrix is
\begin{equation}
 \label{eq:optimized-moment-matrix}
 \mathbf E^{\rm mom}(\mathbf N)_{ji}
 :=\frac{(2\pi c)^{N_{ji}}\sqrt c}{N_{ji}!}
 \bigl(\mu_{N_{ji}}(A_i)\mu_{N_{ji}}(B_j)\bigr)^{1/2}.
\end{equation}
The best-polynomial matrix is
\begin{align}
 \mathbf E^{\rm best}(\mathbf N)_{ji}
 :=\sqrt c\inf_{\substack{\alpha,\beta\in\mathbb R\\
                           p\in\mathcal P_{N_{ji}-1}}}
 \biggl(&\int_{A_i}\int_{B_j}
 \left|e^{-2\pi i c(x-\alpha)(\eta-\beta)}
       -p\bigl((x-\alpha)(\eta-\beta)\bigr)\right|^2
 \,d\eta\,dx\biggr)^{1/2}.                 \label{eq:best-poly-matrix}
\end{align}
Finally, $\mathbf E^{\rm diam}(\mathbf N)$ denotes the old midpoint matrix,
\begin{equation}
 \label{eq:diameter-matrix-recalled}
 \mathbf E^{\rm diam}(\mathbf N)_{ji}
 :=\sqrt{c|A_i||B_j|}\,
   \frac{\bigl(\frac{\pi c}{2}d_i e_j\bigr)^{N_{ji}}}{N_{ji}!},
 \qquad
 d_i=\operatorname{diam}_{\rm ess}(A_i),\quad
 e_j=\operatorname{diam}_{\rm ess}(B_j).
\end{equation}
All three definitions give $\sqrt{c|A_i||B_j|}$ at order zero.

\begin{theorem}[Optimized local block-error hierarchy]
\label{thm:optimized-local-block-hierarchy}
With the preceding notation,
\begin{equation}
 \label{eq:optimized-block-sv}
 s_{1+\sum_{j,i}N_{ji}}(K_{A,B,c})
 \leq
 \|\mathbf E^{\rm best}(\mathbf N)\|_{\ell^2_M\to\ell^2_K}.
\end{equation}
Moreover, entrywise,
\begin{equation}
 \label{eq:entrywise-error-hierarchy}
 0\leq \mathbf E^{\rm best}(\mathbf N)_{ji}
 \leq \mathbf E^{\rm mom}(\mathbf N)_{ji}
 \leq \mathbf E^{\rm diam}(\mathbf N)_{ji},
\end{equation}
and hence
\begin{equation}
 \label{eq:norm-error-hierarchy}
 \|\mathbf E^{\rm best}(\mathbf N)\|
 \leq\|\mathbf E^{\rm mom}(\mathbf N)\|
 \leq\|\mathbf E^{\rm diam}(\mathbf N)\|.
\end{equation}
Consequently, for $0<\epsilon<1/2$, any one of the conditions
\[
 \|\mathbf E^{\rm best}(\mathbf N)\|\leq\sqrt\epsilon,
 \qquad
 \|\mathbf E^{\rm mom}(\mathbf N)\|\leq\sqrt\epsilon,
 \qquad
 \|\mathbf E^{\rm diam}(\mathbf N)\|\leq\sqrt\epsilon
\]
implies, respectively and with no change in the rank cost,
\begin{equation}
 \label{eq:optimized-block-plunge}
 \Lambda_\epsilon(A,cB)\leq\sum_{j,i}N_{ji}.
\end{equation}
In particular, replacing the diameter matrix by either optimized matrix can
never worsen the certificate.
\end{theorem}

\begin{proof}
\emph{Step 1: put a local block on the unscaled frequency piece.}
The unitary dilation
\[
 V_c:L^2(cB_j)\longrightarrow L^2(B_j),
 \qquad (V_cg)(\eta)=\sqrt c\,g(c\eta),
\]
turns the block $P_{cB_j}\mathcal F P_{A_i}$ into the operator from
$L^2(A_i)$ to $L^2(B_j)$ with kernel
\[
 \sqrt c\,e^{-2\pi i c x\eta}.
\]
For arbitrary centers $\alpha,\beta\in\mathbb R$, the identity
\[
 x\eta=(x-\alpha)(\eta-\beta)+\alpha\eta+\beta x-\alpha\beta
\]
shows that the factors not containing the interaction
$(x-\alpha)(\eta-\beta)$ are one-variable unimodular factors.  Multiplying
the input and output by them is unitary and does not affect either rank or
singular values.

\emph{Step 2: a polynomial in the interaction has the advertised rank.}
Let $N=N_{ji}$ and write
\[
 p(z)=\sum_{k=0}^{N-1}a_kz^k
\]
when $N\geq1$; at $N=0$ the sum is empty.  After restoring the separated
unimodular factors, the kernel associated with $p((x-\alpha)(\eta-\beta))$
is a sum of at most $N$ products of a function of $x$ and a function of
$\eta$.  The corresponding operator therefore has rank at most $N$.
The Hilbert--Schmidt norm of the residual is exactly
\[
 \sqrt c\left(\int_{A_i}\int_{B_j}
 \left|e^{-2\pi i c(x-\alpha)(\eta-\beta)}
       -p((x-\alpha)(\eta-\beta))\right|^2
 \,d\eta\,dx\right)^{1/2}.
\]
Thus the infimum in \eqref{eq:best-poly-matrix} is the best
Hilbert--Schmidt residual among this locally centered, rank-at-most-$N$
polynomial class.

The infimum over the centers and coefficients need not be assumed to be
attained.  For every $\delta>0$ and every $(j,i)$, choose an approximant of
rank at most $N_{ji}$ whose residual norm is at most
$\mathbf E^{\rm best}_{ji}+\delta$.  There are only finitely many blocks, so
Theorem~\ref{thm:finite-block-matrix}, applied to the scalar matrix
$\mathbf E^{\rm best}+\delta\mathbf J$, where $\mathbf J$ is the
$K\times M$ all-ones matrix, gives
\[
 s_{1+\sum N_{ji}}(K_{A,B,c})
 \leq\|\mathbf E^{\rm best}+\delta\mathbf J\|.
\]
Letting $\delta\downarrow0$ proves \eqref{eq:optimized-block-sv}.

\emph{Step 3: compare the best polynomial with optimized Taylor moments.}
For fixed $\alpha$ and $\beta$ and $N\geq1$, choose the Taylor polynomial
\[
 p_{N-1}(z)=\sum_{k=0}^{N-1}\frac{(-2\pi i c z)^k}{k!}.
\]
The sharp imaginary-axis remainder estimate gives
\[
 |e^{-2\pi i c z}-p_{N-1}(z)|
 \leq\frac{(2\pi c)^N|z|^N}{N!}.
\]
After squaring and integrating with
$z=(x-\alpha)(\eta-\beta)$, the product structure yields
\begin{align*}
 \mathbf E^{\rm best}_{ji}
 &\leq \frac{(2\pi c)^N\sqrt c}{N!}
 \left(
   \int_{A_i}|x-\alpha|^{2N}\,dx
   \int_{B_j}|\eta-\beta|^{2N}\,d\eta
 \right)^{1/2}.
\end{align*}
Taking the two independent infima proves
$\mathbf E^{\rm best}_{ji}\leq\mathbf E^{\rm mom}_{ji}$.
At $N=0$, the approximating polynomial is zero, the exponential has modulus
one, and both entries equal $\sqrt{c|A_i||B_j|}$, so the same comparison holds
without invoking the remainder lemma at order zero.

\emph{Step 4: compare optimized moments with diameters.}
Choose $\alpha$ and $\beta$ to be the midpoints of the essential convex
hulls of $A_i$ and $B_j$.  Up to null sets,
\[
 |x-\alpha|\leq d_i/2,
 \qquad |\eta-\beta|\leq e_j/2.
\]
Therefore
\[
 \mu_N(A_i)\leq |A_i|(d_i/2)^{2N},
 \qquad
 \mu_N(B_j)\leq |B_j|(e_j/2)^{2N}.
\]
Substitution into \eqref{eq:optimized-moment-matrix} gives precisely
$\mathbf E^{\rm mom}_{ji}\leq\mathbf E^{\rm diam}_{ji}$.

\emph{Step 5: pass from entries to the scalar spectral norm and then count.}
If $0\leq C_{ji}\leq D_{ji}$ entrywise, then, for every vector $x$,
\[
 |Cx|\leq C|x|\leq D|x|
\]
entrywise.  It follows that
$\|C\|_{\ell^2\to\ell^2}\leq\|D\|_{\ell^2\to\ell^2}$.
This proves \eqref{eq:norm-error-hierarchy}.  Finally, an eigenvalue of
$K_{A,B,c}^*K_{A,B,c}$ in $(\epsilon,1-\epsilon)$ produces a singular value
of $K_{A,B,c}$ strictly larger than $\sqrt\epsilon$.  The strict-count
conclusion of Theorem~\ref{thm:finite-block-matrix} now proves
\eqref{eq:optimized-block-plunge}, including the case of equality in any of
the displayed norm constraints.
\end{proof}

The best-polynomial entry is computationally finite-dimensional.  To make
this explicit, fix $N\geq1$, a block $(A_i,B_j)$, and centers
$\alpha,\beta$.  Put $z=(x-\alpha)(\eta-\beta)$ and define, for
$0\leq k,\ell<N$,
\begin{align}
 G_{k\ell}
 &:=\int_{A_i}\int_{B_j}z^{k+\ell}\,d\eta\,dx
   =\left(\int_{A_i}(x-\alpha)^{k+\ell}\,dx\right)
    \left(\int_{B_j}(\eta-\beta)^{k+\ell}\,d\eta\right),
                                                        \label{eq:poly-gram}\\
 d_k&:=\int_{A_i}\int_{B_j}
       z^k e^{-2\pi i c z}\,d\eta\,dx.                 \label{eq:poly-data}
\end{align}
The Gram matrix $G$ is positive definite.  Indeed, if a polynomial $p$ of
degree less than $N$ satisfied
$\int\!\int|p((x-\alpha)(\eta-\beta))|^2=0$, then Fubini's theorem gives,
for almost every $\eta\in B_j\setminus\{\beta\}$, that the polynomial
$x\mapsto p((x-\alpha)(\eta-\beta))$ would vanish on the positive-measure
set $A_i$.  It would therefore be the zero polynomial, and hence $p=0$.
For a coefficient vector $a=(a_0,\ldots,a_{N-1})^T$, the squared residual is
\begin{equation}
 |A_i||B_j|-2\operatorname{Re}(a^*d)+a^*Ga.
 \label{eq:poly-quadratic}
\end{equation}
The unique minimizing coefficients solve $Ga=d$, and the fixed-center best
squared residual is
\begin{equation}
 |A_i||B_j|-d^*G^{-1}d.                    \label{eq:poly-best-residual}
\end{equation}
Restoring the dilation factor in \eqref{eq:best-poly-matrix} gives the
explicit block identity
\begin{equation}
 \mathbf E^{\rm best}(\mathbf N)_{ji}^{2}
 =c\inf_{\alpha,\beta\in\mathbb R}
 \left[|A_i||B_j|-d(\alpha,\beta)^*
 G(\alpha,\beta)^{-1}d(\alpha,\beta)\right].
 \label{eq:poly-best-residual-scaled}
\end{equation}
Thus \eqref{eq:best-poly-matrix} requires only a finite Gram system at each
chosen pair of centers, followed, if desired, by a two-variable optimization
in $(\alpha,\beta)$.  Formula \eqref{eq:poly-best-residual} is nonnegative
because it is the squared distance to a finite-dimensional subspace.

\begin{corollary}[Exact moment gain for interval blocks]
\label{cor:interval-moment-factor}
Suppose that $A_i$ and $B_j$ are intervals, modulo null sets, of lengths
$d_i$ and $e_j$.  Then for every $N=N_{ji}\in\mathbb N_0$,
\begin{equation}
 \label{eq:interval-exact-moment}
 \mu_N(A_i)=\frac{|A_i|(d_i/2)^{2N}}{2N+1},
 \qquad
 \mu_N(B_j)=\frac{|B_j|(e_j/2)^{2N}}{2N+1},
\end{equation}
and consequently
\begin{equation}
 \label{eq:interval-block-factor}
 \boxed{\quad
 \mathbf E^{\rm mom}(\mathbf N)_{ji}
 =\frac{1}{2N_{ji}+1}\,
  \mathbf E^{\rm diam}(\mathbf N)_{ji}.
 \quad}
\end{equation}
Thus $1/(2N+1)$ is the exact factor between the moment-based and
diameter-based certified Hilbert--Schmidt remainder bounds for a full
interval--interval block, not merely an asymptotic comparison.  The actual
best-polynomial residual may be smaller still.
\end{corollary}

\begin{proof}
Let $I=[m-d/2,m+d/2]$.  For $N\geq1$, the function
$a\mapsto\int_I|x-a|^{2N}\,dx$ is strictly convex and is symmetric about
$m$, so its unique minimizer is $a=m$.  Direct integration gives
\[
 \int_I|x-m|^{2N}\,dx
 =2\int_0^{d/2}u^{2N}\,du
 =\frac{d^{2N+1}}{2^{2N}(2N+1)}
 =\frac{|I|(d/2)^{2N}}{2N+1}.
\]
For $N=0$, the integral equals $|I|$ for every center, and the same formula
holds.  Applying this calculation to both intervals in
\eqref{eq:optimized-moment-matrix} gives one factor
$(2N+1)^{-1/2}$ from each moment; their product is $(2N+1)^{-1}$, which proves
\eqref{eq:interval-block-factor}.
\end{proof}

Define $\mathfrak B_{\rm best}$, $\mathfrak B_{\rm mom}$, and
$\mathfrak B_{\rm diam}$ by the same infimum as
in \eqref{eq:partition-matrix-envelope}, replacing the constraint matrix by
$\mathbf E^{\rm best}$, $\mathbf E^{\rm mom}$, and
$\mathbf E^{\rm diam}$, respectively.  Theorem
\ref{thm:optimized-local-block-hierarchy} proves the formal dominance chain
\begin{equation}
 \label{eq:optimized-envelope-chain}
 \Lambda_\epsilon(A,cB)
 \leq\mathfrak B_{\rm best}
 \leq\mathfrak B_{\rm mom}
 \leq\mathfrak B_{\rm diam}.
\end{equation}
The last quantity is the original matrix envelope
$\mathfrak B_{\rm mat}$.  No assertion of strictness is needed for this
chain; Corollary~\ref{cor:interval-moment-factor} identifies a concrete
blockwise strict gain whenever $N_{ji}\geq1$ on an interval--interval block.

\subsection{Two-level orthogonal-input threshold allocation}
\label{subsec:two-level-allocation}

The flat budget in \eqref{eq:weighted-threshold-budget} uses the triangle
inequality across every pair $(i,j)$.  The pieces with different spatial
input indices $i$, however, have orthogonal initial spaces.  Their residual
errors therefore admit an outer $\ell^2$ budget.  Only the frequency sum
inside a fixed input piece requires an $\ell^1$ budget.

\begin{lemma}[Strict count with orthogonal initial spaces]
\label{lem:orthogonal-input-strict-count}
Let $P_1,\ldots,P_M$ be mutually orthogonal projections, and let
$X_i$ be compact operators satisfying $X_i=X_iP_i$.  For arbitrary
$\tau_i>0$,
\begin{equation}
 \label{eq:orthogonal-input-count}
 n\!\left(\Bigl(\sum_{i=1}^{M}\tau_i^2\Bigr)^{1/2};
            \sum_{i=1}^{M}X_i\right)
 \leq\sum_{i=1}^{M}n(\tau_i;X_i).
\end{equation}
\end{lemma}

\begin{proof}
Put $r_i=n(\tau_i;X_i)$.  For every $\delta>0$, the
approximation-number characterization supplies an operator $R_i'$ of rank at
most $r_i$ such that
\[
 \|X_i-R_i'\|\leq\tau_i+\delta.
\]
Replace $R_i'$ by $R_i=R_i'P_i$.  Its rank does not increase and, because
$X_i=X_iP_i$,
\[
 \|X_i-R_i\|
 =\|(X_i-R_i')P_i\|\leq\tau_i+\delta.
\]
Set $E_i=X_i-R_i$ and $R=\sum_iR_i$.  Then
$\operatorname{rank}R\leq\sum_i r_i$.  For any vector $f$,
orthogonality of the $P_i$ and Cauchy--Schwarz give
\begin{align*}
 \left\|\sum_iE_if\right\|
 &=\left\|\sum_iE_iP_if\right\|
 \leq\sum_i(\tau_i+\delta)\|P_if\|\\
 &\leq
 \left(\sum_i(\tau_i+\delta)^2\right)^{1/2}
 \left(\sum_i\|P_if\|^2\right)^{1/2}\\
 &\leq
 \left(\sum_i(\tau_i+\delta)^2\right)^{1/2}\|f\|.
\end{align*}
Hence
\[
 s_{1+\sum_i r_i}\!\left(\sum_iX_i\right)
 \leq\left(\sum_i(\tau_i+\delta)^2\right)^{1/2}.
\]
Letting $\delta\downarrow0$ yields a non-strict upper bound by
$(\sum_i\tau_i^2)^{1/2}$.  Since $n$ counts singular values strictly above
its threshold, \eqref{eq:orthogonal-input-count} follows, including equality
at the threshold.
\end{proof}

Recall
\begin{equation}
 \label{eq:eta-recalled-two-level}
 \eta(u):=\frac{1-\sqrt{1-4u^2}}{2},
 \qquad 0<u<\frac12,
\end{equation}
so that $\eta(u)\in(0,1/2)$ and
$\sqrt{\eta(u)(1-\eta(u))}=u$.

\begin{theorem}[Two-level row and Fourier-dual column allocations]
\label{thm:two-level-partition-allocation}
Let $A,B\subset\mathbb R$ have finite measure, let
\[
 A=\mathop{\dot\bigcup}_{i=1}^{M}A_i,
 \qquad
 B=\mathop{\dot\bigcup}_{j=1}^{K}B_j
\]
be finite measurable partitions with null pieces discarded, and fix
$0<\epsilon<1/2$.  Set
\[
 t=\sqrt{\epsilon(1-\epsilon)}.
\]

\emph{Spatial-input (row) allocation.}  Suppose that positive numbers
$\tau_i$ and $\tau_{ij}$ satisfy
\begin{equation}
 \label{eq:two-level-row-budget}
 \sum_{i=1}^{M}\tau_i^2\leq t^2,
 \qquad
 \sum_{j=1}^{K}\tau_{ij}\leq\tau_i
 \quad(1\leq i\leq M).
\end{equation}
Then
\begin{equation}
 \label{eq:two-level-row-bound}
 \Lambda_\epsilon(A,cB)
 \leq\sum_{i=1}^{M}\sum_{j=1}^{K}
 \Lambda_{\eta(\tau_{ij})}(A_i,cB_j).
\end{equation}

\emph{Fourier-dual (column) allocation.}  Suppose instead that positive
numbers $\sigma_j$ and $\sigma_{ij}$ satisfy
\begin{equation}
 \label{eq:two-level-column-budget}
 \sum_{j=1}^{K}\sigma_j^2\leq t^2,
 \qquad
 \sum_{i=1}^{M}\sigma_{ij}\leq\sigma_j
 \quad(1\leq j\leq K).
\end{equation}
Then
\begin{equation}
 \label{eq:two-level-column-bound}
 \Lambda_\epsilon(A,cB)
 \leq\sum_{j=1}^{K}\sum_{i=1}^{M}
 \Lambda_{\eta(\sigma_{ij})}(A_i,cB_j).
\end{equation}
In both statements the budget inequalities are non-strict, whereas all
individual thresholds are positive.  The hypotheses automatically imply
$0<\tau_{ij},\sigma_{ij}<1/2$, so \eqref{eq:eta-recalled-two-level} is always
evaluated in its stated domain.
\end{theorem}

\begin{proof}
We prove the two allocations separately.

\emph{Step 1: the exact global spatial-side defect.}
Write
\[
 P=P_A,
 \qquad Q=Q_{cB}=\mathcal F^{-1}P_{cB}\mathcal F,
 \qquad S=PQP,
 \qquad T=(I-P)QP.
\]
Since $P$ and $Q$ are orthogonal projections,
\[
 T^*T=PQ(I-P)QP=PQP-PQPQP=S-S^2.
\]
For $\lambda\in[0,1]$,
\[
 \lambda(1-\lambda)>t^2
 \quad\Longleftrightarrow\quad
 \epsilon<\lambda<1-\epsilon.
\]
Functional calculus therefore gives the exact strict-count identity
\begin{equation}
 \label{eq:two-level-global-defect-count}
 \Lambda_\epsilon(A,cB)=n(t;T).
\end{equation}

\emph{Step 2: group the defect by its orthogonal spatial input pieces.}
Put $P_i=P_{A_i}$, $Q_j=Q_{cB_j}$, and
\[
 T_{ij}:=(I-P)Q_jP_i,
 \qquad
 T_i:=\sum_{j=1}^{K}T_{ij}=(I-P)QP_i.
\]
Then $T=\sum_iT_i$ and $T_i=T_iP_i$.  The projections $P_i$ are mutually
orthogonal.  Every $T_{ij}$ is compact: the operator
$P_{cB_j}\mathcal FP_{A_i}$ is Hilbert--Schmidt with squared norm
$c|A_i||B_j|$, and unitary conjugation and compression preserve compactness.

Let
$\rho=(\sum_i\tau_i^2)^{1/2}\leq t$.  Monotonicity in the strict threshold,
Lemma~\ref{lem:orthogonal-input-strict-count}, and then the ordinary strict
Ky--Fan inequality inside each fixed $i$ give
\begin{align}
 n(t;T)
 &\leq n(\rho;T)
 \leq\sum_i n(\tau_i;T_i)                                      \notag\\
 &\leq\sum_i n\!\left(\sum_j\tau_{ij};T_i\right)
 \leq\sum_{i,j}n(\tau_{ij};T_{ij}).             \label{eq:row-two-level-count}
\end{align}
Here the second line uses
$\sum_j\tau_{ij}\leq\tau_i$: increasing a threshold can only decrease a
strict singular-value count.  The final inequality is
Lemma~\ref{lem:strict-kyfan} applied to $T_i=\sum_jT_{ij}$.

\emph{Step 3: compare each compressed block with its exact local defect.}
Define
\[
 \widetilde T_{ij}:=(I-P_i)Q_jP_i,
 \qquad S_{ij}:=P_iQ_jP_i.
\]
Because $A^c\subset A_i^c$ modulo null sets,
$T_{ij}=(I-P)\widetilde T_{ij}$.  Hence left compression gives
\begin{equation}
 n(u;T_{ij})\leq n(u;\widetilde T_{ij}),
 \qquad u>0.                                      \label{eq:row-local-compression}
\end{equation}
The projection identity yields
\[
 \widetilde T_{ij}^*\widetilde T_{ij}
 =P_iQ_j(I-P_i)Q_jP_i=S_{ij}-S_{ij}^2.
\]
For $0<u<1/2$, the roots in $[0,1]$ of
$\lambda(1-\lambda)=u^2$ are $\eta(u)$ and $1-\eta(u)$.  Therefore
\begin{equation}
 n(u;\widetilde T_{ij})
 =\Lambda_{\eta(u)}(A_i,cB_j).                  \label{eq:row-local-exact}
\end{equation}
The budget implies
$\tau_{ij}\leq\tau_i\leq t<1/2$, so
\eqref{eq:row-local-exact} applies with $u=\tau_{ij}$.  Combining
\eqref{eq:two-level-global-defect-count}--\eqref{eq:row-local-exact} proves
\eqref{eq:two-level-row-bound}.

\emph{Step 4: move to the Fourier-side defect for the column allocation.}
Let
\[
 P'=P_{cB},
 \qquad R=\mathcal FP_A\mathcal F^{-1},
 \qquad S'=P'RP',
 \qquad U=(I-P')RP'.
\]
If $K=P_{cB}\mathcal FP_A$, then $S=K^*K$ and $S'=KK^*$.  These two compact
positive operators have the same nonzero eigenvalues, with multiplicity.
The plunge interval does not contain zero, and hence
\[
 \Lambda_\epsilon(A,cB)=\Lambda_\epsilon(S').
\]
Also
\[
 U^*U=P'R(I-P')RP'=S'-(S')^2,
\]
so the Fourier-side analogue of \eqref{eq:two-level-global-defect-count} is
\begin{equation}
 \label{eq:dual-global-defect-count}
 \Lambda_\epsilon(A,cB)=n(t;U).
\end{equation}

Put $P'_j=P_{cB_j}$ and
$R_i=\mathcal FP_{A_i}\mathcal F^{-1}$.  Define
\[
 U_{ij}:=(I-P')R_iP'_j,
 \qquad U_j:=\sum_{i=1}^{M}U_{ij}=(I-P')RP'_j.
\]
Then $U=\sum_jU_j$, $U_j=U_jP'_j$, and the $P'_j$ are mutually
orthogonal.  Applying Lemma~\ref{lem:orthogonal-input-strict-count} across
$j$ and Lemma~\ref{lem:strict-kyfan} across $i$ within each fixed $j$ gives,
with $\rho'=(\sum_j\sigma_j^2)^{1/2}\leq t$,
\begin{equation}
 n(t;U)
 \leq n(\rho';U)
 \leq\sum_jn(\sigma_j;U_j)
 \leq\sum_j n\!\left(\sum_i\sigma_{ij};U_j\right)
 \leq\sum_{j,i}n(\sigma_{ij};U_{ij}).             \label{eq:column-two-level-count}
\end{equation}

\emph{Step 5: identify the local Fourier-side defects.}
For each pair define
\[
 \widetilde U_{ij}:=(I-P'_j)R_iP'_j,
 \qquad S'_{ij}:=P'_jR_iP'_j.
\]
Since $(cB)^c\subset(cB_j)^c$ modulo null sets,
$U_{ij}=(I-P')\widetilde U_{ij}$, and therefore
\[
 n(u;U_{ij})\leq n(u;\widetilde U_{ij}).
\]
Furthermore,
\[
 \widetilde U_{ij}^*\widetilde U_{ij}
 =S'_{ij}-(S'_{ij})^2.
\]
The nonzero spectrum of $S'_{ij}$ is the same as that of
$P_{A_i}Q_{cB_j}P_{A_i}$, because these are respectively the two products
$K_{ij}K_{ij}^*$ and $K_{ij}^*K_{ij}$ for
$K_{ij}=P_{cB_j}\mathcal FP_{A_i}$.  Thus, for $0<u<1/2$,
\begin{equation}
 n(u;\widetilde U_{ij})
 =\Lambda_{\eta(u)}(A_i,cB_j).                   \label{eq:column-local-exact}
\end{equation}
The column budget gives
$0<\sigma_{ij}\leq\sigma_j\leq t<1/2$.  Substituting
\eqref{eq:column-local-exact} into \eqref{eq:column-two-level-count} and
using \eqref{eq:dual-global-defect-count} proves
\eqref{eq:two-level-column-bound}.

All count functions used strict inequalities.  Consequently, equality in
either global budget and equality at any local singular-value threshold cause
no endpoint loss: an eigenvalue satisfying
$\lambda(1-\lambda)=u^2$ lies at an endpoint and is excluded from both
strict counts.
\end{proof}

The row allocation has a useful normalized form.  Choose $a_i>0$ and
$b_{ij}>0$ such that
\[
 \sum_i a_i^2\leq1,
 \qquad
 \sum_jb_{ij}\leq1\quad\text{for each }i,
\]
and set $\tau_i=ta_i$, $\tau_{ij}=ta_ib_{ij}$.  Then
\begin{equation}
 \label{eq:normalized-row-allocation}
 \Lambda_\epsilon(A,cB)
 \leq\sum_{i,j}\Lambda_{\eta(ta_ib_{ij})}(A_i,cB_j).
\end{equation}
The dual form uses $a_j>0$ with $\sum_ja_j^2\leq1$ and
$b_{ij}>0$ with $\sum_ib_{ij}\leq1$ for each $j$.

Finally, the new allocation formally dominates the flat $\ell^1$ budget.
Indeed, if positive $u_{ij}$ satisfy $\sum_{i,j}u_{ij}\leq t$, set
$\tau_{ij}=u_{ij}$ and $\tau_i=\sum_ju_{ij}$.  Then
\[
 \left(\sum_i\tau_i^2\right)^{1/2}
 \leq\sum_i\tau_i
 =\sum_{i,j}u_{ij}
 \leq t,
\]
so every allocation admitted by \eqref{eq:weighted-threshold-budget} is
admitted by \eqref{eq:two-level-row-budget}, with exactly the same local
right-hand side.  The inclusion of feasible allocations is generally strict:
with two spatial pieces and one frequency piece, the choice
$\tau_{11}=\tau_{21}=t/\sqrt2$ exhausts the outer $\ell^2$ budget but has
flat sum $\sqrt2\,t>t$.  The Fourier-dual column envelope supplies a second,
independently optimized certificate, and the minimum of the row and column
envelopes retains whichever orthogonality is more favorable for the given
partition.

\section{A finite-measure window outside the cited regularity hypotheses}
\label{sec:rough-example}

The moment theorem is not merely a different proof for bounded intervals.  It
also applies to finite-measure windows that lie outside every positive-order
translation-perimeter class.  The following example makes that distinction
explicit.  For a measurable set $E\subset\R$ and $0<\gamma\leq 1$, write
\begin{equation}\label{eq:translation-perimeter}
 \operatorname{Per}_{\gamma}(E)
 :=\sup_{h\in\R\setminus\{0\}}
   \frac{|E\mathbin{\triangle}(E+h)|}{|h|^{\gamma}}.
\end{equation}
This is the translation-modulus version of fractional perimeter used in rough
window estimates.  Replacing the supremum by one over $0<|h|\leq 1$ would not
change any conclusion below, because all of the test translations tend to
zero.

\begin{theorem}[An unbounded window with infinite fractional perimeter]
\label{thm:rough-window}
For $n\geq 1$, set
\begin{equation}\label{eq:rough-data}
 m_n=e^{-n^4},\qquad
 K_n=\bigl\lceil e^{n^6}\bigr\rceil,\qquad
 \ell_n=\frac{m_n}{K_n},
\end{equation}
and form the alternating microinterval cluster
\begin{equation}\label{eq:rough-cluster}
 C_n
 =\bigcup_{k=0}^{K_n-1}
   [n+2k\ell_n,\,n+(2k+1)\ell_n).
\end{equation}
Let
\begin{equation}\label{eq:rough-window}
 A_*:=\bigcup_{n=1}^{\infty}C_n.
\end{equation}
Then the following assertions hold.
\begin{enumerate}
 \item $A_*$ is unbounded but has finite positive measure, and
       $\int_{A_*}|x|^q\,dx<\infty$ for every $q\geq0$.
 \item $\operatorname{Per}_{\gamma}(A_*)=\infty$ for every
       $0<\gamma\leq1$.
 \item If
       \(M_N=\int_{A_*}x^{2N}\,dx\), then
       \begin{equation}\label{eq:rough-moment-asymptotic}
        \log M_N=\frac{N}{2}\log N+O(N)
        \qquad(N\longrightarrow\infty).
       \end{equation}
 \item For fixed $c>0$, put
       \begin{equation}\label{eq:rough-KS}
        K_c=P_{cA_*}\F P_{A_*},
        \qquad S_c=K_c^*K_c=P_{A_*}Q_{cA_*}P_{A_*}.
       \end{equation}
       Then
       \begin{equation}\label{eq:rough-singular-decay}
        \log s_{N+1}(K_c)
        \leq-\frac{N}{2}\log N+O_c(N),
       \end{equation}
       and, as $\eps\downarrow0$,
       \begin{equation}\label{eq:rough-count}
        \operatorname{Tr}\one_{(\eps,1-\eps)}(S_c)
        =O_c\!\left(
          \frac{\log(1/\eps)}{\log\log(1/\eps)}
        \right).
       \end{equation}
\end{enumerate}
\end{theorem}

\begin{proof}
We verify the four claims separately.

\smallskip
\noindent\emph{Step 1: location and measure of the clusters.}
Since $K_n\ell_n=m_n$, every point of $C_n$ lies in
\[
 [n,n+2K_n\ell_n]=[n,n+2m_n].
\]
Moreover, $2m_n\leq2e^{-1}<1$.  Thus
\begin{equation}\label{eq:cluster-contained}
 C_n\subset[n,n+1),
\end{equation}
so the clusters are pairwise disjoint.  The $K_n$ half-open intervals in
\eqref{eq:rough-cluster} are also pairwise disjoint and each has length
$\ell_n$.  Consequently
\begin{equation}\label{eq:cluster-measure}
 |C_n|=K_n\ell_n=m_n=e^{-n^4}
\end{equation}
and hence
\begin{equation}\label{eq:rough-total-measure}
 0<|A_*|=\sum_{n=1}^{\infty}e^{-n^4}<\infty.
\end{equation}
Because $C_n$ has positive measure and lies to the right of $n$, the set
$A_*$ is not essentially bounded; in particular,
$\operatorname{diam}_{\mathrm{ess}}(A_*)=\infty$.

For $q\geq0$, inclusion \eqref{eq:cluster-contained} and
\eqref{eq:cluster-measure} give
\begin{equation}\label{eq:all-moments-first}
 \int_{A_*}|x|^q\,dx
 =\sum_{n=1}^{\infty}\int_{C_n}x^q\,dx
 \leq\sum_{n=1}^{\infty}(n+1)^q e^{-n^4}<\infty.
\end{equation}
The final series converges because the exponential $e^{-n^4}$ dominates
every fixed polynomial.  This proves the first assertion.

\smallskip
\noindent\emph{Step 2: every positive fractional perimeter is infinite.}
Fix $0<\gamma\leq1$ and translate by
\begin{equation}\label{eq:test-shift}
 h_n:=\ell_n.
\end{equation}
Translation by $h_n$ moves each occupied microinterval in $C_n$ into the
adjacent gap:
\[
 C_n+h_n
 =\bigcup_{k=0}^{K_n-1}
 [n+(2k+1)\ell_n,\,n+(2k+2)\ell_n).
\]
Up to endpoints, this translated cluster is disjoint from $C_n$.  It is
contained in $[n,n+2m_n]$, which by \eqref{eq:cluster-contained} meets no
other cluster.  Therefore
\begin{equation}\label{eq:symmetric-difference-lower}
 |A_*\mathbin{\triangle}(A_*+h_n)|
 \geq |(C_n+h_n)\setminus A_*|
 =m_n.
\end{equation}
Using $h_n=m_n/K_n$ in \eqref{eq:translation-perimeter} yields
\begin{align}
 \operatorname{Per}_{\gamma}(A_*)
 &\geq \frac{m_n}{\ell_n^{\gamma}}
   =m_n^{1-\gamma}K_n^{\gamma} \notag\\
 &\geq
   \exp\!\bigl(\gamma n^6-(1-\gamma)n^4\bigr).
 \label{eq:perimeter-blowup}
\end{align}
For every fixed $\gamma>0$, the exponent on the right tends to $+\infty$.
Thus the supremum in \eqref{eq:translation-perimeter} is infinite.  Notice
also that $h_n\to0$, so the same argument applies to the local version of
the translation modulus.

\smallskip
\noindent\emph{Step 3: sharp logarithmic order of the moments.}
For every integer $N\geq1$, the location of $C_n$ gives the two-sided
comparison
\begin{equation}\label{eq:moment-sandwich}
 \sum_{n=1}^{\infty}e^{-n^4}n^{2N}
 \leq M_N
 \leq\sum_{n=1}^{\infty}e^{-n^4}(n+1)^{2N}.
\end{equation}
For the lower bound, let
\[
 r_N=(N/2)^{1/4},\qquad j_N=\lfloor r_N\rfloor.
\]
When $N$ is large, $j_N\geq r_N/2$.  Keeping only the $j_N$th term in the
left side of \eqref{eq:moment-sandwich}, and using
$j_N^4\leq r_N^4=N/2$, gives
\begin{align}
 \log M_N
 &\geq-j_N^4+2N\log j_N \\
 &\geq-\frac N2+2N\log(r_N/2) \notag\\
 &=\frac N2\log N-N\!\left(
     \frac12+\frac12\log2+2\log2
   \right).
 \label{eq:moment-lower}
\end{align}
Thus $\log M_N\geq (N/2)\log N-O(N)$.

For the upper bound, $n+1\leq2n$ for $n\geq1$, and hence
\begin{align}
 M_N
 &\leq4^N\sum_{n=1}^{\infty}e^{-n^4}n^{2N} \notag\\
 &=4^N\sum_{n=1}^{\infty}e^{-n^4/2}
              \bigl(e^{-n^4/2}n^{2N}\bigr) \\
 &\leq4^N\left(\sum_{n=1}^{\infty}e^{-n^4/2}\right)
       \sup_{x>0}\bigl(e^{-x^4/2}x^{2N}\bigr).
 \label{eq:moment-upper-sup}
\end{align}
The logarithm of the expression in the supremum is
$2N\log x-x^4/2$.  Its derivative is $2N/x-2x^3$, so its unique maximum
occurs at $x=N^{1/4}$ and equals
\begin{equation}\label{eq:continuous-maximum}
 \frac N2\log N-\frac N2.
\end{equation}
The series in parentheses in \eqref{eq:moment-upper-sup} is an absolute
constant.  Taking logarithms proves
\[
 \log M_N\leq\frac N2\log N+O(N).
\]
Together with \eqref{eq:moment-lower}, this is
\eqref{eq:rough-moment-asymptotic}.

\smallskip
\noindent\emph{Step 4: singular values and the plunge count.}
Apply Theorem~\ref{thm:moment-hard} with $A=B=A_*$ and with both centering
parameters equal to zero.  The two moment factors are both $M_N$, so for
every $N\geq0$,
\begin{equation}\label{eq:rough-moment-sv}
 s_{N+1}(K_c)
 \leq\frac{\sqrt c\,(2\pi c)^N}{N!}\,M_N.
\end{equation}
For $N\geq1$, the elementary Stirling bound $N!\geq(N/e)^N$ and
\eqref{eq:rough-moment-asymptotic} imply
\begin{align}
 \log s_{N+1}(K_c)
 &\leq \frac12\log c+N\log(2\pi c)
       +\log M_N-\log N! \notag\\
 &\leq-\frac N2\log N+O_c(N),
 \label{eq:rough-sv-log-proof}
\end{align}
which is \eqref{eq:rough-singular-decay}.

To obtain the stated integer count with explicit threshold constants, put
\begin{equation}\label{eq:rough-L}
 L=\log(1/\eps).
\end{equation}
Equation \eqref{eq:rough-sv-log-proof} means that there are constants
$C_c>0$ and $N_c$ such that
\begin{equation}\label{eq:rough-sv-explicit}
 \log s_{N+1}(K_c)
 \leq-\frac N2\log N+C_cN
 \qquad(N\geq N_c).
\end{equation}
For all sufficiently large $L$, choose
\begin{equation}\label{eq:rough-choice-N}
 N_\eps=\left\lceil\frac{4L}{\log L}\right\rceil.
\end{equation}
Increasing the lower threshold for $L$ if necessary ensures
$N_\eps\geq N_c$, $\log N_\eps\geq\frac12\log L$, and
$C_c\leq\frac14\log N_\eps$.  Substitution in
\eqref{eq:rough-sv-explicit} then gives
\begin{align}
 \log s_{N_\eps+1}(K_c)
 &\leq-\frac14N_\eps\log N_\eps \\
 &\leq-\frac14\frac{4L}{\log L}\frac12\log L
 =-\frac L2.
 \label{eq:rough-threshold}
\end{align}
Hence $s_{N_\eps+1}(K_c)\leq e^{-L/2}=\sqrt\eps$.  The direct spectral
transfer inequality gives
\[
 \operatorname{Tr}\one_{(\eps,1-\eps)}(S_c)
 \leq n(\sqrt\eps;K_c)\leq N_\eps.
\]
Finally, $\log L=\log\log(1/\eps)$, so
\eqref{eq:rough-choice-N} proves \eqref{eq:rough-count}.  This completes all
four assertions.
\end{proof}

\begin{remark}[Exact noncoverage, and what the example does not claim]
\label{rem:rough-noncoverage}
The point of Theorem~\ref{thm:rough-window} is hypothesis-level noncoverage,
not a claim that every earlier conclusion is false for $A_*$.  The failures
can be checked one by one.
\begin{itemize}
 \item A theorem restricted to a finite union of bounded intervals cannot be
       invoked: $A_*$ is unbounded and contains infinitely many interval
       components.  In particular, the finite-union component theorem based
       on the one-sided block estimate in \cite{AzimifardIndependent} does not
       cover this window.
 \item Any bounded-window or finite-diameter theorem cannot be invoked because
       $\operatorname{diam}_{\mathrm{ess}}(A_*)=\infty$.  Thus a
diameter-based Taylor
       parameter is infinite even though every polynomial moment is finite.
 \item Any hypothesis requiring finitely many boundary points fails.  Indeed,
       $n$ is a boundary point of $A_*$ for every integer $n\geq1$, so
       $\partial A_*$ is unbounded and infinite.  More strongly, for each
       $0<\delta<1/4$, the $\delta$-neighborhood of $\partial A_*$ contains
       the pairwise disjoint intervals $(n-\delta,n+\delta)$ for all $n$.
       Its Lebesgue measure is therefore infinite.  Consequently no global
       finite boundary-tube or upper-Minkowski-content parameter is available
       in bounded-domain results of that kind, including their formulations
       in \cite{KulikovLarsen}.
 \item A rough-window theorem requiring
       $\operatorname{Per}_{\gamma}(A)<\infty$ for at least one
       $\gamma>0$, such as the corresponding quantitative regime in
       \cite{HughesIsraelMayeli}, cannot be invoked because
       \eqref{eq:perimeter-blowup} excludes every
       $\gamma\in(0,1]$.
\end{itemize}
On the other hand, $|A_*|<\infty$, so the localization operator is still
Hilbert--Schmidt before squaring and trace class after squaring.  In
particular, the generic estimate
\[
 n(\sqrt\eps;K_c)
 \leq\frac{\|K_c\|_{\mathcal S_2}^2}{\eps}
 =\frac{c|A_*|^2}{\eps}
\]
and qualitative theorems valid for arbitrary finite-measure windows remain
applicable.  The new information is that the moment method gives the much
smaller logarithmic-over-double-logarithmic certificate
\eqref{eq:rough-count} despite the simultaneous failure of boundedness,
finite-component geometry, finite boundary content, and every positive
fractional translation perimeter.
\end{remark}

\section{Beyond Fourier: kernels with finitely many bounded separated interactions}
\label{sec:analytic-kernels}

The Fourier kernel is special, but the finite-rank approximation argument does
not require a linear phase.  What it requires is that, after removing
one-variable phases, the remaining interaction be a finite sum of separated
products.  The following theorem makes that statement precise.  It is an
operator theorem, rather than a change of notation for the Fourier transform.
Related separated representations occur in numerical work on Fourier integral
operators \cite{CandesDemanetYing}; the estimates below are proved directly and
concern singular values and plunge counts rather than algorithmic complexity.
No analyticity assumption is used: the hypotheses are the displayed
measurability, bounded centered features, and finite weighted $L^2$ amplitude
norm.

Let $(X,\mu)$ and $(Y,\nu)$ be $\sigma$-finite measure spaces, and fix
integers $p,r\ge1$.  Let
$w_X,a_1,\ldots,a_p$ be measurable functions on $X$ and let
$w_Y,b_1,\ldots,b_p$ be measurable functions on $Y$, with
\[
  w_Xa_q\in L^2(X),\qquad w_Yb_q\in L^2(Y),
  \qquad 1\le q\le p.
\]
Set
\[
  \mathfrak a(x,y):=\sum_{q=1}^{p}a_q(x)b_q(y).
\]
Thus the amplitude has separated rank at most $p$.  Let
$\phi:X\to\mathbb R$, $\psi:Y\to\mathbb R$, and let
$X_\ell:X\to\mathbb R$, $Y_\ell:Y\to\mathbb R$ be measurable for
$1\le \ell\le r$.  Consider
\begin{equation}
  (Kf)(y)
  =w_Y(y)\int_X w_X(x)\mathfrak a(x,y)
       \exp\!\left(i\Phi(x,y)\right)f(x)\,d\mu(x),
  \label{eq:finite-interaction-operator}
\end{equation}
where
\begin{equation}
  \Phi(x,y)
  =\phi(x)+\psi(y)+\sum_{\ell=1}^{r}X_\ell(x)Y_\ell(y).
  \label{eq:finite-interaction-phase}
\end{equation}

Choose real centers $\xi_\ell,\eta_\ell$ and put
\[
  R_\ell:=\operatorname*{ess\,sup}_{x\in X}
      |X_\ell(x)-\xi_\ell|,
  \qquad
  S_\ell:=\operatorname*{ess\,sup}_{y\in Y}
      |Y_\ell(y)-\eta_\ell|.
\]
Throughout this section we assume that these radii are finite, and we set
\begin{equation}
  \zeta:=\sum_{\ell=1}^{r}R_\ell S_\ell,
  \qquad
  C:=\left\|w_X(x)w_Y(y)\mathfrak a(x,y)
       \right\|_{L^2(X\times Y)}.
  \label{eq:analytic-C-zeta}
\end{equation}
The hypotheses on the separated factors imply $C<\infty$ because
\[
 C\leq\sum_{q=1}^p
 \|w_Xa_q\|_{L^2(X)}\,\|w_Yb_q\|_{L^2(Y)}<\infty.
\]

For $N\ge1$ define
\[
  R_N:=p\binom{N+r-1}{r},
  \qquad R_0:=0.
\]

\begin{theorem}[Finite separated-interaction envelope]
\label{thm:finite-interaction-envelope}
The operator $K$ in \eqref{eq:finite-interaction-operator} is
Hilbert--Schmidt.  For every integer $N\ge0$ there is an operator $K_N$ of
rank at most $R_N$ such that
\begin{equation}
  \|K-K_N\|
  \le \|K-K_N\|_{\mathrm{HS}}
  \le C\frac{\zeta^N}{N!}.
  \label{eq:finite-interaction-error}
\end{equation}
For $N=0$ the right-hand side is understood as $C$, including when
$\zeta=0$.  Consequently,
\begin{equation}
  s_{R_N+1}(K)\le C\frac{\zeta^N}{N!}.
  \label{eq:finite-interaction-singular}
\end{equation}

If, in addition, $\|K\|\le1$, $S:=K^*K$, and
$0<\epsilon<1/2$, define
\begin{equation}
  N_{\mathrm{an}}(\epsilon)
  :=\min\left\{N\in\mathbb N_0:
       C\frac{\zeta^N}{N!}\le\sqrt\epsilon\right\}.
  \label{eq:analytic-exact-order}
\end{equation}
The admissible set is nonempty because $\zeta^N/N!\to0$.
Then
\begin{equation}
  \Lambda_\epsilon(S)
  \le p\binom{N_{\mathrm{an}}(\epsilon)+r-1}{r},
  \label{eq:finite-interaction-count}
\end{equation}
where the binomial coefficient on the right is interpreted as zero when
$N_{\mathrm{an}}(\epsilon)=0$.
\end{theorem}

\begin{proof}
We give the construction and the rank count explicitly.

\emph{Step 1: center every interaction.}
Write
\[
  \widetilde X_\ell(x):=X_\ell(x)-\xi_\ell,
  \qquad
  \widetilde Y_\ell(y):=Y_\ell(y)-\eta_\ell.
\]
For each $\ell$ the elementary identity
\[
  X_\ell Y_\ell
  =\widetilde X_\ell\widetilde Y_\ell
   +\eta_\ell X_\ell+\xi_\ell Y_\ell-\xi_\ell\eta_\ell
\]
separates all terms except
$\widetilde X_\ell\widetilde Y_\ell$.  Hence, with
\[
  \phi_0(x):=\phi(x)+\sum_{\ell=1}^{r}\eta_\ell X_\ell(x)
                 -\sum_{\ell=1}^{r}\xi_\ell\eta_\ell,
  \qquad
  \psi_0(y):=\psi(y)+\sum_{\ell=1}^{r}\xi_\ell Y_\ell(y),
\]
we have
\begin{equation}
  e^{i\Phi(x,y)}
  =e^{i\phi_0(x)}e^{i\psi_0(y)}e^{i\theta(x,y)},
  \qquad
  \theta(x,y):=
  \sum_{\ell=1}^{r}\widetilde X_\ell(x)\widetilde Y_\ell(y).
  \label{eq:centered-interaction}
\end{equation}
The first two factors in \eqref{eq:centered-interaction} have modulus one
and depend on one variable only.  Moreover,
\begin{equation}
  |\theta(x,y)|
  \le\sum_{\ell=1}^{r}
       |\widetilde X_\ell(x)|\,|\widetilde Y_\ell(y)|
  \le\sum_{\ell=1}^{r}R_\ell S_\ell=\zeta
  \label{eq:theta-zeta}
\end{equation}
for almost every $(x,y)$.

\emph{Step 2: expand the coupled exponential.}
For $N\ge1$, use the Taylor polynomial of degree $N-1$:
\[
  P_{N-1}(t):=\sum_{k=0}^{N-1}\frac{(it)^k}{k!}.
\]
The multinomial identity gives
\begin{align}
  P_{N-1}(\theta(x,y))
  &=\sum_{k=0}^{N-1}\frac{i^k}{k!}
      \left(\sum_{\ell=1}^{r}
       \widetilde X_\ell(x)\widetilde Y_\ell(y)\right)^k \notag\\
  &=\sum_{|\boldsymbol\nu|<N}
       \frac{i^{|\boldsymbol\nu|}}{\boldsymbol\nu!}
       \prod_{\ell=1}^{r}\widetilde X_\ell(x)^{\nu_\ell}
       \prod_{\ell=1}^{r}\widetilde Y_\ell(y)^{\nu_\ell},
  \label{eq:multindex-taylor}
\end{align}
where $\boldsymbol\nu=(\nu_1,\ldots,\nu_r)\in\mathbb N_0^r$,
$|\boldsymbol\nu|=\nu_1+\cdots+\nu_r$, and
$\boldsymbol\nu!=\nu_1!\cdots\nu_r!$.

Replace $e^{i\theta}$ in the kernel of $K$ by
$P_{N-1}(\theta)$ and call the resulting operator $K_N$.  After inserting
$\mathfrak a=\sum_{q=1}^{p}a_qb_q$ into
\eqref{eq:multindex-taylor}, every pair
$(q,\boldsymbol\nu)$ contributes a kernel of the form
\[
  \bigg[w_Y(y)b_q(y)e^{i\psi_0(y)}
    \prod_{\ell=1}^{r}\widetilde Y_\ell(y)^{\nu_\ell}\bigg]
  \bigg[w_X(x)a_q(x)e^{i\phi_0(x)}
    \prod_{\ell=1}^{r}\widetilde X_\ell(x)^{\nu_\ell}\bigg]
\]
times a scalar.  This is a rank-one kernel.  The factors are in $L^2$
because the centered features are essentially bounded.  The number of
multi-indices with $|\boldsymbol\nu|<N$ is
\[
  \sum_{k=0}^{N-1}\binom{k+r-1}{r-1}
  =\binom{N+r-1}{r}.
\]
It follows that
\[
  \operatorname{rank}K_N
  \le p\binom{N+r-1}{r}=R_N.
\]
For $N=0$ we set $K_0=0$, which has rank $R_0=0$.

\emph{Step 3: estimate the remainder without an extraneous exponential
factor.}
For real $t$ and $N\ge1$, the integral form of Taylor's remainder is
\[
  e^{it}-P_{N-1}(t)
  =\frac{(it)^N}{(N-1)!}
       \int_0^1(1-s)^{N-1}e^{ist}\,ds.
\]
Taking absolute values and using $|e^{ist}|=1$ yields
\begin{equation}
  |e^{it}-P_{N-1}(t)|
  \le\frac{|t|^N}{(N-1)!}
       \int_0^1(1-s)^{N-1}\,ds
  =\frac{|t|^N}{N!}.
  \label{eq:analytic-imaginary-remainder}
\end{equation}
Combining \eqref{eq:theta-zeta} and
\eqref{eq:analytic-imaginary-remainder}, the residual kernel is bounded
pointwise by
\[
  |w_X(x)w_Y(y)\mathfrak a(x,y)|\frac{\zeta^N}{N!}.
\]
Its $L^2(X\times Y)$ norm is therefore at most
$C\zeta^N/N!$.  The $L^2$ norm of a kernel is the
Hilbert--Schmidt norm of the associated operator, and operator norm is at
most Hilbert--Schmidt norm.  This proves
\eqref{eq:finite-interaction-error}.  It also proves that $K$ itself is
Hilbert--Schmidt.  The case $N=0$ follows directly from $K_0=0$ and
$\|K\|_{\mathrm{HS}}=C$.

\emph{Step 4: pass from approximation rank to singular values.}
The approximation-number characterization of singular values gives
\[
  s_{R_N+1}(K)
  =\inf_{\operatorname{rank}L\le R_N}\|K-L\|
  \le\|K-K_N\|.
\]
Together with Step 3 this is
\eqref{eq:finite-interaction-singular}.

\emph{Step 5: transfer the estimate to the plunge count.}
Suppose now that $K$ is a contraction.  If
$C\zeta^N/N!\le\sqrt\epsilon$, then Step 4 and the strict singular-value
count imply
\[
  n(\sqrt\epsilon;K)\le R_N.
\]
Every eigenvalue $\lambda$ of $S=K^*K$ in
$(\epsilon,1-\epsilon)$ produces a singular value
$\sqrt\lambda>\sqrt\epsilon$ of $K$.  Thus
\[
  \Lambda_\epsilon(S)\le n(\sqrt\epsilon;K)\le R_N.
\]
Choosing the least admissible integer $N$ gives
\eqref{eq:finite-interaction-count}.
\end{proof}

The integer factorial minimization in \eqref{eq:analytic-exact-order} is the
strongest certificate furnished by this approximation envelope.  A closed
Lambert--$W$ certificate follows as in the Fourier case.

\begin{corollary}[Lambert--$W$ order for separated interactions]
\label{cor:analytic-lambert}
Assume the hypotheses of Theorem~\ref{thm:finite-interaction-envelope},
$\|K\|\le1$, and $0<\epsilon<1/2$.  If $C\le\sqrt\epsilon$, then
$\Lambda_\epsilon(K^*K)=0$.  If $C>\sqrt\epsilon$ and $\zeta=0$, then
$\Lambda_\epsilon(K^*K)\le p$.  If $C>\sqrt\epsilon$ and $\zeta>0$, put
\[
  \rho:=\log\frac{C}{\sqrt\epsilon},
  \qquad
  N_W:=\left\lceil
     \frac{\rho}{W_0\!\left(\rho/(e\zeta)\right)}
  \right\rceil.
\]
Then
\[
  \Lambda_\epsilon(K^*K)
  \le p\binom{N_W+r-1}{r}.
\]
\end{corollary}

\begin{proof}
The first two cases follow from the order-$0$ and order-$1$ approximants in
Theorem~\ref{thm:finite-interaction-envelope}.  In the remaining case, let
\[
  N_*:=\frac{\rho}{W_0(\rho/(e\zeta))}.
\]
The identity $W_0(u)e^{W_0(u)}=u$ shows, after substitution, that
\[
  N_*\log\frac{N_*}{e\zeta}=\rho.
\]
In particular $N_*>e\zeta$.  The function
$x\mapsto x\log(x/(e\zeta))$ is increasing on $[e\zeta,\infty)$.  Moreover,
$N_W\geq N_*>0$ and $N_W$ is an integer, so $N_W\geq1$.  Thus
$N_W\ge N_*$ gives
\[
  N_W\log\frac{N_W}{e\zeta}\ge\rho.
\]
Using $N_W!\ge(N_W/e)^{N_W}$, we obtain
\[
  C\frac{\zeta^{N_W}}{N_W!}
  \le C\left(\frac{e\zeta}{N_W}\right)^{N_W}
  \le Ce^{-\rho}=\sqrt\epsilon.
\]
Theorem~\ref{thm:finite-interaction-envelope} now gives the asserted count.
\end{proof}

\subsection{A nonlinear separated-interaction example}

For real parameters $\alpha,\beta,\gamma$, define
\begin{equation}
  (K_{\alpha,\beta,\gamma}f)(y)
  :=\gamma\int_{-1}^{1}
       \exp\!\left(i\bigl[\alpha xy+\beta e^xe^y\bigr]\right)
       f(x)\,dx,
  \qquad -1\le y\le1.
  \label{eq:nonlinear-kernel-example}
\end{equation}

\begin{proposition}[A two-interaction nonlinear kernel]
\label{prop:nonlinear-kernel-example}
Set
\[
  \zeta_{\alpha,\beta}
  :=|\alpha|+|\beta|\sinh^2(1).
\]
For every $N\ge0$,
\begin{equation}
  s_{\binom{N+1}{2}+1}(K_{\alpha,\beta,\gamma})
  \le 2|\gamma|\frac{\zeta_{\alpha,\beta}^{N}}{N!},
  \label{eq:nonlinear-kernel-sv}
\end{equation}
with the same $N=0$ convention as above.  If $|\gamma|\le1/2$, then
$K_{\alpha,\beta,\gamma}$ is a contraction.  In that case, for
$0<\epsilon<1/2$ and
$S_{\alpha,\beta,\gamma}:=K_{\alpha,\beta,\gamma}^*
K_{\alpha,\beta,\gamma}$,
\begin{equation}
  \Lambda_\epsilon(S_{\alpha,\beta,\gamma})
  \le\binom{N_\epsilon+1}{2},
  \qquad
  N_\epsilon:=\min\left\{N\ge0:
     2|\gamma|\frac{\zeta_{\alpha,\beta}^{N}}{N!}
     \le\sqrt\epsilon\right\}.
  \label{eq:nonlinear-kernel-count}
\end{equation}
\end{proposition}

\begin{proof}
The amplitude in \eqref{eq:nonlinear-kernel-example} is the single
separated term $\gamma\cdot1$, so $p=1$.  There are two interactions.  For
the first choose
\[
  X_1(x)=x,\qquad Y_1(y)=\alpha y,
  \qquad \xi_1=\eta_1=0.
\]
The corresponding radii are $R_1=1$ and $S_1=|\alpha|$.  For the second
choose
\[
  X_2(x)=e^x,\qquad Y_2(y)=\beta e^y,
  \qquad \xi_2=\cosh(1),\quad \eta_2=\beta\cosh(1).
\]
Because the midpoint and half-width of $[e^{-1},e]$ are respectively
$\cosh(1)$ and $\sinh(1)$, we have
\[
  R_2=\sinh(1),
  \qquad S_2=|\beta|\sinh(1).
\]
Therefore
\[
  \zeta=R_1S_1+R_2S_2
  =|\alpha|+|\beta|\sinh^2(1).
\]
The $L^2([-1,1]^2)$ norm of the amplitude is
\[
  C=\left(\int_{-1}^{1}\int_{-1}^{1}|\gamma|^2\,dx\,dy\right)^{1/2}
   =2|\gamma|.
\]
With $r=2$, the rank in Theorem~\ref{thm:finite-interaction-envelope} is
$\binom{N+1}{2}$, which proves \eqref{eq:nonlinear-kernel-sv}.
Furthermore,
\[
  \|K_{\alpha,\beta,\gamma}\|
  \le\|K_{\alpha,\beta,\gamma}\|_{\mathrm{HS}}
  =2|\gamma|.
\]
Thus $|\gamma|\le1/2$ is a sufficient contraction condition, and
\eqref{eq:nonlinear-kernel-count} follows from the direct count in the
theorem.
\end{proof}

\subsection{Centered best-polynomial and anisotropic certificates}
\label{sec:analytic-sharp-certificates}

The factorial envelope in
Theorem~\ref{thm:finite-interaction-envelope} uses only the rectangular
estimate $|\theta|\leq\zeta$.  Three refinements are available without
changing the finite-interaction hypothesis.  First, a constant phase can be
removed after the interactions have been centered, leaving the exact
half-oscillation of the joint interaction.  Second, the scalar polynomial can
be chosen by weighted $L^2$ projection rather than by Taylor truncation.
Third, when the interactions have very different sizes, they can be assigned
different polynomial degrees.  We give all three variants, including the
actual monomial-span rank rather than only the cardinality of the displayed
separated expansion.

Retain the hypotheses and notation of
Theorem~\ref{thm:finite-interaction-envelope}.  In particular, for a fixed
choice of feature centers $\xi_\ell,\eta_\ell$, write
\begin{align}
 \widetilde X_\ell&:=X_\ell-\xi_\ell,
 &\widetilde Y_\ell&:=Y_\ell-\eta_\ell,\notag\\
 z_\ell(x,y)&:=\widetilde X_\ell(x)\widetilde Y_\ell(y),
 &\theta(x,y)&:=\sum_{\ell=1}^r z_\ell(x,y),
 \label{eq:sharp-analytic-z-theta}
\end{align}
and let $\phi_0,\psi_0$ be the one-variable phases in
\eqref{eq:centered-interaction}.  Put
\begin{equation}
 H(x,y):=w_X(x)w_Y(y)\mathfrak a(x,y),
 \qquad \|H\|_{L^2(X\times Y)}=C.
 \label{eq:sharp-analytic-H}
\end{equation}
If $(\mu\times\nu)(X\times Y)=0$, then $C=0$ and $K=0$, so all
remaining approximation and counting conclusions are immediate.  When joint
essential extrema are used below, we therefore assume
$(\mu\times\nu)(X\times Y)>0$.

\begin{lemma}[Exact scalar centering of the joint interaction]
\label{lem:sharp-analytic-half-oscillation}
Define
\begin{equation}
 \theta_-:=\operatorname*{ess\,inf}_{X\times Y}\theta,
 \qquad
 \theta_+:=\operatorname*{ess\,sup}_{X\times Y}\theta,
 \qquad
 \theta_0:=\frac{\theta_++\theta_-}{2},
 \qquad
 M:=\frac{\theta_+-\theta_-}{2}.
 \label{eq:sharp-analytic-theta0-M}
\end{equation}
Then $0\leq M\leq\zeta$, and
\begin{equation}
 M=\inf_{c\in\mathbb R}\|\theta-c\|_{L^\infty(X\times Y)}.
 \label{eq:sharp-analytic-best-scalar-center}
\end{equation}
Consequently, with $u:=\theta-\theta_0$,
\begin{equation}
 |u|\leq M\quad\hbox{a.e.},
 \qquad
 e^{i\Phi(x,y)}
 =e^{i\phi_0(x)}e^{i\psi_0(y)}e^{i\theta_0}e^{iu(x,y)}.
 \label{eq:sharp-analytic-centered-factorization}
\end{equation}
Thus $[-M,M]$ is a centered uniform scalar approximation interval, and it is
never larger than $[-\zeta,\zeta]$.  If the amplitude vanishes on an
extremal set, its weighted support may permit a smaller interval still.
\end{lemma}

\begin{proof}
The pointwise estimate $|\theta|\leq\zeta$ from
\eqref{eq:theta-zeta} implies
$-\zeta\leq\theta_-\leq\theta_+\leq\zeta$.  Hence $M$ is finite and
$M\leq\zeta$.

Let $c\in\mathbb R$.  By the definitions of essential infimum and
essential supremum,
\begin{equation}
 \|\theta-c\|_{L^\infty}
 =\max\{\theta_+-c,c-\theta_-\}.
 \label{eq:sharp-analytic-center-max}
\end{equation}
The maximum of two real numbers is at least their average, so
\[
 \max\{\theta_+-c,c-\theta_-\}
 \geq \frac{(\theta_+-c)+(c-\theta_-)}2=M.
\]
At $c=\theta_0$, the two entries of the maximum are both equal to $M$.
This proves \eqref{eq:sharp-analytic-best-scalar-center} and
$|\theta-\theta_0|\leq M$ almost everywhere.  Finally,
$e^{i\theta}=e^{i\theta_0}e^{i(\theta-\theta_0)}$ in
\eqref{eq:centered-interaction}; the factor $e^{i\theta_0}$ is a harmless scalar
phase.  This gives \eqref{eq:sharp-analytic-centered-factorization}.
\end{proof}

We next record the rank that is actually supplied by the monomial spans.
For $N\geq1$, set
\begin{align}
 \mathcal U_N
 &:={\rm span}\left\{
 w_Xa_qe^{i\phi_0}
 \prod_{\ell=1}^r\widetilde X_\ell^{\nu_\ell}:
 1\leq q\leq p,\ |\boldsymbol\nu|<N
 \right\}\subset L^2(X),\notag\\
 \mathcal V_N
 &:={\rm span}\left\{
 w_Yb_qe^{i\psi_0}
 \prod_{\ell=1}^r\widetilde Y_\ell^{\nu_\ell}:
 1\leq q\leq p,\ |\boldsymbol\nu|<N
 \right\}\subset L^2(Y),
 \label{eq:sharp-analytic-total-degree-spans}
\end{align}
and define
\begin{equation}
 d_N:=\min\{\dim\mathcal U_N,\dim\mathcal V_N\},
 \qquad d_0:=0.
 \label{eq:sharp-analytic-dN}
\end{equation}
Since there are $\binom{N+r-1}{r}$ multi-indices of total degree less
than $N$,
\begin{equation}
 d_N\leq p\binom{N+r-1}{r}=R_N.
 \label{eq:sharp-analytic-span-versus-cardinality}
\end{equation}
The inequality can be strict when feature monomials or amplitude factors are
linearly dependent in the relevant $L^2$ spaces.

\begin{theorem}[Weighted best scalar polynomial and span-compressed rank]
\label{thm:sharp-analytic-best-polynomial}
Let $u=\theta-\theta_0$ and $M$ be as in
Lemma~\ref{lem:sharp-analytic-half-oscillation}.  Define the finite positive
measure $\sigma$ on $[-M,M]$ by pushforward:
\begin{equation}
 \sigma(B)
 :=\int_{X\times Y}\mathbf1_B(u(x,y))|H(x,y)|^2
 \,d\mu(x)d\nu(y).
 \label{eq:sharp-analytic-pushforward}
\end{equation}
Thus $\sigma([-M,M])=C^2$.  For $N\geq1$, put
\begin{equation}
 \mathcal E_N^{\rm best}
 :=\min_{P\in\mathbb C[t],\ \deg P<N}
 \left(\int_{-M}^{M}|e^{it}-P(t)|^2\,d\sigma(t)\right)^{1/2},
 \qquad
 \mathcal E_0^{\rm best}:=C.
 \label{eq:sharp-analytic-best-error}
\end{equation}
Then the minimum is attained.  For every $N\geq0$ there is an operator
$K_N^{\rm best}$ such that
\begin{equation}
 \operatorname{rank}K_N^{\rm best}\leq d_N,
 \qquad
 \|K-K_N^{\rm best}\|
 \leq\|K-K_N^{\rm best}\|_{\rm HS}
 =\mathcal E_N^{\rm best}.
 \label{eq:sharp-analytic-best-operator}
\end{equation}
Consequently,
\begin{equation}
 s_{d_N+1}(K)\leq\mathcal E_N^{\rm best}
 \leq C\frac{M^N}{N!}
 \leq C\frac{\zeta^N}{N!},
 \qquad N\geq1.
 \label{eq:sharp-analytic-best-chain}
\end{equation}
For every $N\in\mathbb N_0$, if $K$ is a contraction and
$0<\epsilon<1/2$, then
\begin{equation}
 \mathcal E_N^{\rm best}\leq\sqrt\epsilon
 \quad\Longrightarrow\quad
 \Lambda_\epsilon(K^*K)\leq d_N.
 \label{eq:sharp-analytic-best-count}
\end{equation}

For $N\geq1$, more explicitly, let
\begin{equation}
 G^{(N)}_{jk}:=\int_{-M}^{M}t^{j+k}\,d\sigma(t),
 \qquad
 b^{(N)}_j:=\int_{-M}^{M}t^je^{it}\,d\sigma(t),
 \qquad 0\leq j,k<N.
 \label{eq:sharp-analytic-normal-data}
\end{equation}
Let $G^{(N)\dagger}$ denote the Moore--Penrose inverse.  One minimizing
coefficient vector is
\begin{equation}
 c=G^{(N)\dagger}b^{(N)},
 \label{eq:sharp-analytic-best-coefficients}
\end{equation}
and
\begin{equation}
 (\mathcal E_N^{\rm best})^2
 =C^2-(b^{(N)})^*G^{(N)\dagger}b^{(N)}.
 \label{eq:sharp-analytic-best-matrix-formula}
\end{equation}
This formula remains valid when $G^{(N)}$ is singular.
\end{theorem}

\begin{proof}
We separate the projection, rank, and spectral steps.

\emph{Step 1: existence of the weighted best polynomial.}
The restrictions to $[-M,M]$ of polynomials of degree less than $N$ form a
finite-dimensional subspace of $L^2(\sigma)$.  Every finite-dimensional
subspace of a Hilbert space is closed.  Therefore the orthogonal projection
of $e^{it}$ onto that subspace exists.  Any polynomial representing this
projection is a minimizer in \eqref{eq:sharp-analytic-best-error}.  If
$\sigma$ is supported on fewer than $N$ points, the representing coefficient
vector need not be unique, but the projected function and the minimum error
are unique.

\emph{Step 2: a polynomial in $u$ has the same total-degree separated
structure.}
Write a polynomial of degree less than $N$ as
$P(t)=\sum_{k=0}^{N-1}c_kt^k$.  Since
$u=\sum_{\ell=1}^rz_\ell-\theta_0$, the multinomial theorem gives
\begin{equation}
 P(u)=\sum_{|\boldsymbol\nu|<N}
 d_{\boldsymbol\nu}(P)
 \prod_{\ell=1}^rz_\ell^{\nu_\ell},
 \label{eq:sharp-analytic-general-polynomial-expansion}
\end{equation}
where, for $|\boldsymbol\nu|<N$,
\begin{equation}
 d_{\boldsymbol\nu}(P)
 =\frac{|\boldsymbol\nu|!}{\boldsymbol\nu!}
 \sum_{k=|\boldsymbol\nu|}^{N-1}
 c_k\binom{k}{|\boldsymbol\nu|}(-\theta_0)^{k-|\boldsymbol\nu|}.
 \label{eq:sharp-analytic-general-polynomial-coefficients}
\end{equation}
Indeed, one first chooses $|\boldsymbol\nu|$ copies of
$\sum_\ell z_\ell$ from $(\sum_\ell z_\ell-\theta_0)^k$ and then uses the
multinomial coefficient $|\boldsymbol\nu|!/\boldsymbol\nu!$.

Replace $e^{iu}$ in
\eqref{eq:sharp-analytic-centered-factorization} by $P(u)$.  After inserting
$\mathfrak a=\sum_{q=1}^pa_qb_q$ and
$z_\ell^{\nu_\ell}=\widetilde X_\ell^{\nu_\ell}
\widetilde Y_\ell^{\nu_\ell}$, the resulting kernel is a finite sum indexed
by $(q,\boldsymbol\nu)$ with $|\boldsymbol\nu|<N$.  Its range lies in
$\mathcal V_N$.  The range of its adjoint lies, up to complex conjugation,
in $\mathcal U_N$.  Hence its rank is at most both span dimensions and
therefore at most $d_N$.

This also gives a polynomial-specific refinement.  If
\begin{equation}
 I(P):=\{\boldsymbol\nu:d_{\boldsymbol\nu}(P)\neq0\},
 \label{eq:sharp-analytic-active-index-set}
\end{equation}
then one may replace $\mathcal U_N,\mathcal V_N$ by the spans in
\eqref{eq:sharp-analytic-total-degree-spans} restricted to
$\boldsymbol\nu\in I(P)$.  Thus zero coefficients and all $L^2$ linear
dependencies are legitimate rank savings; merely counting the written
rank-one summands is not necessary.

\emph{Step 3: the weighted residual is exactly the Hilbert--Schmidt
residual.}
All removed phase factors have modulus one.  Therefore the squared
Hilbert--Schmidt norm of the residual associated with $P$ is
\begin{align}
 \int_{X\times Y}|H(x,y)|^2
 |e^{iu(x,y)}-P(u(x,y))|^2\,d\mu(x)d\nu(y)
 &=\int_{-M}^{M}|e^{it}-P(t)|^2\,d\sigma(t),
 \label{eq:sharp-analytic-pushforward-error}
\end{align}
where the equality is precisely the pushforward identity defining
$\sigma$.  Choosing a minimizing polynomial proves
\eqref{eq:sharp-analytic-best-operator}.  The case $N=0$ is obtained from
the zero operator.

\emph{Step 4: compute the minimizer from a finite matrix.}
For $P(t)=\sum_{k=0}^{N-1}c_kt^k$, expansion of the squared norm gives
\begin{equation}
 \int|e^{it}-P(t)|^2\,d\sigma(t)
 =C^2-2\operatorname{Re}(c^*b^{(N)})+c^*G^{(N)}c.
 \label{eq:sharp-analytic-quadratic-form}
\end{equation}
The normal equations are $G^{(N)}c=b^{(N)}$.  They are consistent even if
$G^{(N)}$ is singular: if $a\in\ker G^{(N)}$, then
$\sum_ja_jt^j=0$ in $L^2(\sigma)$, and consequently
$a^*b^{(N)}=0$.  Hence $b^{(N)}$ lies in the range of $G^{(N)}$.  The
Moore--Penrose choice $c=G^{(N)\dagger}b^{(N)}$ solves the normal equations.
Substitution into \eqref{eq:sharp-analytic-quadratic-form} gives
\eqref{eq:sharp-analytic-best-matrix-formula}.

\emph{Step 5: compare with centered Taylor approximation.}
For $N\geq1$, the polynomial
$T_{N-1}(t)=\sum_{k=0}^{N-1}(it)^k/k!$ satisfies, by
\eqref{eq:analytic-imaginary-remainder},
\[
 |e^{it}-T_{N-1}(t)|\leq\frac{|t|^N}{N!}
 \leq\frac{M^N}{N!}
 \quad (|t|\leq M).
\]
It is an admissible competitor in
\eqref{eq:sharp-analytic-best-error}.  Since $\sigma([-M,M])=C^2$, this
proves $\mathcal E_N^{\rm best}\leq CM^N/N!$.  The inequality
$M\leq\zeta$ proves the last comparison in
\eqref{eq:sharp-analytic-best-chain}.

\emph{Step 6: singular values and the plunge count.}
The finite-rank approximation rule applied to
\eqref{eq:sharp-analytic-best-operator} gives
$s_{d_N+1}(K)\leq\mathcal E_N^{\rm best}$.  If $K$ is a contraction and
$\mathcal E_N^{\rm best}\leq\sqrt\epsilon$, then
$n(\sqrt\epsilon;K)\leq d_N$.  The direct transfer
\eqref{eq:abstract-direct-count} yields
$\Lambda_\epsilon(K^*K)\leq d_N$, proving
\eqref{eq:sharp-analytic-best-count}.
\end{proof}

The preceding theorem is data-sensitive but requires the weighted
pushforward measure.  A fully explicit uniform alternative is furnished by
Chebyshev polynomials and Bessel coefficients.

\begin{corollary}[Chebyshev--Bessel envelope with unchanged monomial rank]
\label{cor:sharp-analytic-chebyshev-bessel}
Let $J_k$ denote the Bessel function of the first kind and let $T_k$ denote
the $k$th Chebyshev polynomial.  For $N\geq0$ and $M\geq0$, define
\begin{equation}
 \chi_0(M):=1,
 \qquad
 \chi_N(0):=0\quad(N\geq1),
 \label{eq:sharp-analytic-chi-edge}
\end{equation}
and, when $M>0$ and $N\geq1$, define
\begin{equation}
 \chi_N(M):=\min\left\{
 1,\ \frac{M^N}{N!},\ 2\sum_{k=N}^{\infty}|J_k(M)|
 \right\}.
 \label{eq:sharp-analytic-chi}
\end{equation}
Then
\begin{equation}
 s_{d_N+1}(K)\leq C\chi_N(M),
 \label{eq:sharp-analytic-chebyshev-sv}
\end{equation}
and, if $K$ is a contraction and $0<\epsilon<1/2$,
\begin{equation}
 C\chi_N(M)\leq\sqrt\epsilon
 \quad\Longrightarrow\quad
 \Lambda_\epsilon(K^*K)\leq d_N
 \leq p\binom{N+r-1}{r}.
 \label{eq:sharp-analytic-chebyshev-count}
\end{equation}
For $M>0$ and $N\geq1$, the polynomial producing the Bessel branch is
\begin{equation}
 Q_{N-1,M}(t)
 :=J_0(M)+2\sum_{k=1}^{N-1}i^kJ_k(M)T_k(t/M),
 \label{eq:sharp-analytic-chebyshev-polynomial}
\end{equation}
which has degree at most $N-1$.  Thus the Chebyshev improvement does not
increase the combinatorial rank.

Moreover, whenever $N+1>M/2$,
\begin{equation}
 2\sum_{k=N}^{\infty}|J_k(M)|
 \leq
 \frac{2(M/2)^N}
 {N!\left(1-\dfrac{M}{2(N+1)}\right)}.
 \label{eq:sharp-analytic-explicit-bessel-tail}
\end{equation}
\end{corollary}

\begin{proof}
Assume first that $M>0$ and $N\geq1$, and put $s=t/M$.  To recall the special-function
identity used below, start from the Bessel generating series
\[
 \exp\!\left(\frac M2(z-z^{-1})\right)
 =\sum_{k\in\mathbb Z}J_k(M)z^k.
\]
Substitution of $z=ie^{i\alpha}$ and pairing the $k$ and $-k$ terms, using
$J_{-k}(M)=(-1)^kJ_k(M)$, gives
$e^{iM\cos\alpha}=J_0(M)+2\sum_{k\geq1}i^kJ_k(M)\cos(k\alpha)$.
Since $T_k(\cos\alpha)=\cos(k\alpha)$, the Jacobi--Anger expansion on
$[-1,1]$ is therefore
\begin{equation}
 e^{iMs}=J_0(M)+2\sum_{k=1}^{\infty}i^kJ_k(M)T_k(s).
 \label{eq:sharp-analytic-jacobi-anger}
\end{equation}
Since $|T_k(s)|\leq1$ for $|s|\leq1$, truncation after $k=N-1$ gives
\begin{equation}
 \sup_{|t|\leq M}|e^{it}-Q_{N-1,M}(t)|
 \leq2\sum_{k=N}^{\infty}|J_k(M)|.
 \label{eq:sharp-analytic-bessel-uniform-error}
\end{equation}
The series is absolutely and uniformly convergent.  Indeed, the Poisson
integral representation reads
\begin{equation}
 J_k(M)=\frac{(M/2)^k}{\sqrt\pi\,\Gamma(k+1/2)}
 \int_{-1}^{1}e^{iMs}(1-s^2)^{k-1/2}\,ds.
 \label{eq:sharp-analytic-bessel-poisson}
\end{equation}
Taking absolute values, we remove the unimodular factor from the
integrand.  The remaining beta integral is
\begin{align}
 \int_{-1}^{1}(1-s^2)^{k-1/2}\,ds
 &=\int_0^1v^{-1/2}(1-v)^{k-1/2}\,dv\notag\\
 &=B\!\left(\frac12,k+\frac12\right)
 =\frac{\sqrt\pi\,\Gamma(k+1/2)}{\Gamma(k+1)}.
 \label{eq:sharp-analytic-bessel-beta}
\end{align}
Here the first equality follows from evenness followed by $v=s^2$.
Substitution in \eqref{eq:sharp-analytic-bessel-poisson} and
$\Gamma(k+1)=k!$ gives, for every integer $k\geq0$,
\begin{equation}
 |J_k(M)|\leq\frac{(M/2)^k}{k!}.
 \label{eq:sharp-analytic-bessel-elementary-bound}
\end{equation}
The factorial majorant is summable, so it also justifies the pairing and
uniform truncation of the generating series above.

There are two other admissible polynomials of degree less than $N$: the
centered Taylor polynomial has uniform error at most $M^N/N!$, and the zero
polynomial has uniform error one.  Choosing the best of these three
competitors proves that the weighted best error is at most $C\chi_N(M)$.
Equation \eqref{eq:sharp-analytic-chebyshev-sv} and the count
\eqref{eq:sharp-analytic-chebyshev-count} now follow from
Theorem~\ref{thm:sharp-analytic-best-polynomial}.  If $M=0$, then $u=0$
almost everywhere and the degree-zero polynomial $1$ is exact for every
$N\geq1$.  The case $N=0$ uses the zero operator.

It remains to prove the closed tail estimate.  Set
$a_k=(M/2)^k/k!$.  For $k\geq N$,
\[
 \frac{a_{k+1}}{a_k}=\frac{M}{2(k+1)}
 \leq\frac{M}{2(N+1)}=:q<1.
\]
Hence \eqref{eq:sharp-analytic-bessel-elementary-bound} and the geometric
series give
\[
 2\sum_{k=N}^{\infty}|J_k(M)|
 \leq2\sum_{k=N}^{\infty}a_k
 \leq\frac{2a_N}{1-q},
\]
which is exactly
\eqref{eq:sharp-analytic-explicit-bessel-tail}.
\end{proof}

The total-degree approximation treats $u=\sum_\ell z_\ell-\theta_0$ as one scalar
variable.  The following alternative assigns a separate degree to every
interaction.  It is useful when a few interactions are strong and the
others are weak.

For each $\ell$, define its exact essential midpoint and half-oscillation by
\begin{align}
 z_{\ell,-}&:=\operatorname*{ess\,inf}_{X\times Y}z_\ell,
 &z_{\ell,+}&:=\operatorname*{ess\,sup}_{X\times Y}z_\ell,\notag\\
 m_\ell&:=\frac{z_{\ell,+}+z_{\ell,-}}2,
 &M_\ell&:=\frac{z_{\ell,+}-z_{\ell,-}}2,
 &t_\ell&:=z_\ell-m_\ell.
 \label{eq:sharp-analytic-individual-centers}
\end{align}
Thus $|t_\ell|\leq M_\ell$ almost everywhere and
\begin{equation}
 e^{i\theta}=e^{i\sum_{\ell=1}^rm_\ell}
 \prod_{\ell=1}^re^{it_\ell}.
 \label{eq:sharp-analytic-product-factorization}
\end{equation}

For a degree vector
$\boldsymbol N=(N_1,\ldots,N_r)\in\mathbb N^r$, set
\begin{align}
 \mathcal U_{\boldsymbol N}^{\Box}
 &:={\rm span}\left\{
 w_Xa_qe^{i\phi_0}
 \prod_{\ell=1}^r\widetilde X_\ell^{\nu_\ell}:
 1\leq q\leq p,\ 0\leq\nu_\ell<N_\ell
 \right\},\notag\\
 \mathcal V_{\boldsymbol N}^{\Box}
 &:={\rm span}\left\{
 w_Yb_qe^{i\psi_0}
 \prod_{\ell=1}^r\widetilde Y_\ell^{\nu_\ell}:
 1\leq q\leq p,\ 0\leq\nu_\ell<N_\ell
 \right\},\notag\\
 d_{\boldsymbol N}^{\Box}
 &:=\min\left\{\dim\mathcal U_{\boldsymbol N}^{\Box},
                  \dim\mathcal V_{\boldsymbol N}^{\Box}\right\}
 \leq p\prod_{\ell=1}^rN_\ell.
 \label{eq:sharp-analytic-box-spans}
\end{align}

\begin{theorem}[Anisotropic product approximation]
\label{thm:sharp-analytic-anisotropic-product}
For each $\ell$, let $P_\ell$ be a complex polynomial of degree less than
$N_\ell$ and define
\begin{equation}
 \delta_\ell:=\sup_{|t|\leq M_\ell}|e^{it}-P_\ell(t)|,
 \qquad
 q_\ell:=\sup_{|t|\leq M_\ell}|P_\ell(t)|.
 \label{eq:sharp-analytic-delta-q}
\end{equation}
For a permutation $\pi$ of $\{1,\ldots,r\}$, put
\begin{equation}
 \Delta_\pi
 :=\sum_{j=1}^r\delta_{\pi(j)}
       \prod_{h<j}q_{\pi(h)},
 \qquad
 \Delta:=\min_\pi\Delta_\pi.
 \label{eq:sharp-analytic-telescoping-delta}
\end{equation}
There is an operator $K_{\boldsymbol N,\boldsymbol P}^{\Box}$ of rank at
most $d_{\boldsymbol N}^{\Box}$ such that
\begin{equation}
 \|K-K_{\boldsymbol N,\boldsymbol P}^{\Box}\|
 \leq\|K-K_{\boldsymbol N,\boldsymbol P}^{\Box}\|_{\rm HS}
 \leq C\Delta.
 \label{eq:sharp-analytic-anisotropic-error}
\end{equation}
Since $q_\ell\leq1+\delta_\ell$, one always has the symmetric bound
\begin{equation}
 \Delta\leq\prod_{\ell=1}^r(1+\delta_\ell)-1.
 \label{eq:sharp-analytic-correct-error-product}
\end{equation}
One may also compare with the zero operator; hence there is an approximant
of rank at most $d_{\boldsymbol N}^{\Box}$ and error at most
\begin{equation}
 C\min\left\{1,\prod_{\ell=1}^r(1+\delta_\ell)-1\right\}.
 \label{eq:sharp-analytic-product-zero-minimum}
\end{equation}
If $K$ is a contraction, $0<\epsilon<1/2$, and the right-hand side of
\eqref{eq:sharp-analytic-product-zero-minimum} is at most
$\sqrt\epsilon$, then
\begin{equation}
 \Lambda_\epsilon(K^*K)\leq d_{\boldsymbol N}^{\Box}
 \leq p\prod_{\ell=1}^rN_\ell.
 \label{eq:sharp-analytic-anisotropic-count}
\end{equation}
\end{theorem}

\begin{proof}
\emph{Step 1: construct the tensor-degree polynomial.}
Replace the product in
\eqref{eq:sharp-analytic-product-factorization} by
\begin{equation}
 e^{i\sum_\ell m_\ell}\prod_{\ell=1}^rP_\ell(t_\ell).
 \label{eq:sharp-analytic-product-approximant}
\end{equation}
Because $t_\ell=z_\ell-m_\ell$ and $\deg P_\ell<N_\ell$, expansion of
$P_\ell(z_\ell-m_\ell)$ involves only
$1,z_\ell,\ldots,z_\ell^{N_\ell-1}$.  Multiplying these expansions over
$\ell$ gives only multi-indices satisfying
$0\leq\nu_\ell<N_\ell$.  After inserting the separated amplitude, the
range and adjoint range lie in the two spaces in
\eqref{eq:sharp-analytic-box-spans}.  Thus the approximating operator has
rank at most $d_{\boldsymbol N}^{\Box}$, not merely the raw term count
$p\prod_\ell N_\ell$.

\emph{Step 2: telescope the product error.}
For any chosen ordering of the factors, write
$E_\ell=e^{it_\ell}$.  The exact identity
\begin{equation}
 \prod_{j=1}^rE_{\pi(j)}-\prod_{j=1}^rP_{\pi(j)}
 =\sum_{j=1}^r
 \left(\prod_{h<j}P_{\pi(h)}\right)
 (E_{\pi(j)}-P_{\pi(j)})
 \left(\prod_{h>j}E_{\pi(h)}\right)
 \label{eq:sharp-analytic-product-telescoping}
\end{equation}
is obtained by replacing the factors one at a time.  Since
$|E_\ell|=1$, taking absolute values gives the bound $\Delta_\pi$ in
\eqref{eq:sharp-analytic-telescoping-delta}.  We may choose the best
ordering, which gives $\Delta$.

Multiplication by $|H|$ and integration over $X\times Y$ yield
\eqref{eq:sharp-analytic-anisotropic-error}.  Moreover,
$|P_\ell|\leq|e^{it}|+|P_\ell-e^{it}|\leq1+\delta_\ell$.  Substitution into
the telescoping sum gives
\[
 \sum_{j=1}^r\delta_{\pi(j)}
 \prod_{h<j}(1+\delta_{\pi(h)})
 =\prod_{\ell=1}^r(1+\delta_\ell)-1,
\]
proving \eqref{eq:sharp-analytic-correct-error-product}.  The zero operator
has rank zero and Hilbert--Schmidt error $C$, proving
\eqref{eq:sharp-analytic-product-zero-minimum}.

\emph{Step 3: transfer to the spectrum.}
If the error in \eqref{eq:sharp-analytic-product-zero-minimum} is at most
$\sqrt\epsilon$, the finite-rank approximation rule gives
$n(\sqrt\epsilon;K)\leq d_{\boldsymbol N}^{\Box}$.  The direct transfer
\eqref{eq:abstract-direct-count} proves
\eqref{eq:sharp-analytic-anisotropic-count}.
\end{proof}

The next corollary removes the unspecified polynomials from the anisotropic
theorem.

\begin{corollary}[Explicit anisotropic Taylor--Bessel certificate]
\label{cor:sharp-analytic-anisotropic-explicit}
For $N\geq1$ and $M\geq0$, define
\begin{equation}
 \widehat\chi_N(M):=
 \begin{cases}
 0,&M=0,\\[2mm]
 \displaystyle
 \min\left\{\dfrac{M^N}{N!},
       2\sum_{k=N}^{\infty}|J_k(M)|\right\},&M>0.
 \end{cases}
 \label{eq:sharp-analytic-widehat-chi}
\end{equation}
For $\boldsymbol N\in\mathbb N^r$, set
\begin{equation}
 \eta_{\boldsymbol N}
 :=\min\left\{1,
 \prod_{\ell=1}^r
 \left(1+\widehat\chi_{N_\ell}(M_\ell)\right)-1
 \right\}.
 \label{eq:sharp-analytic-anisotropic-eta}
\end{equation}
Then
\begin{equation}
 s_{d_{\boldsymbol N}^{\Box}+1}(K)
 \leq C\eta_{\boldsymbol N}.
 \label{eq:sharp-analytic-anisotropic-explicit-sv}
\end{equation}
If $K$ is a contraction and $0<\epsilon<1/2$, then
\begin{equation}
 C\eta_{\boldsymbol N}\leq\sqrt\epsilon
 \quad\Longrightarrow\quad
 \Lambda_\epsilon(K^*K)
 \leq d_{\boldsymbol N}^{\Box}
 \leq p\prod_{\ell=1}^rN_\ell.
 \label{eq:sharp-analytic-anisotropic-explicit-count}
\end{equation}
\end{corollary}

\begin{proof}
If $M_\ell=0$, take $P_\ell\equiv1$, which is exact.  If $M_\ell>0$,
choose whichever of the degree-$(N_\ell-1)$ Taylor polynomial and the
Chebyshev--Bessel polynomial
\eqref{eq:sharp-analytic-chebyshev-polynomial} gives the smaller displayed
error.  Then
$\delta_\ell\leq\widehat\chi_{N_\ell}(M_\ell)$.  Apply
\eqref{eq:sharp-analytic-product-zero-minimum} and then the singular-value
and plunge-count conclusions of
Theorem~\ref{thm:sharp-analytic-anisotropic-product}.
\end{proof}

\begin{remark}[What the sharpenings do and do not assert]
\label{rem:sharp-analytic-caveats}
The preceding bounds should be interpreted with the following qualifications.
\begin{enumerate}
\item The number $M$ is the exact half-oscillation for the fixed centered
interaction $\theta$.  Different admissible feature centers, or a different
finite separated representation of the same phase, can change $M$ and the
weighted measure $\sigma$.  Every such choice gives a valid certificate, so
one may minimize the resulting upper bounds.  The word ``exact'' in
Lemma~\ref{lem:sharp-analytic-half-oscillation} refers to optimal subtraction
of a scalar from that fixed $\theta$, not to an optimization over all possible
representations of the phase.

\item The dimensions in
\eqref{eq:sharp-analytic-total-degree-spans} and
\eqref{eq:sharp-analytic-box-spans} are ranks of finite Gram matrices and are
therefore computable whenever the corresponding inner products are known.
For example, if $r=2$ and
$\widetilde X_1=\widetilde X_2$ almost everywhere, the input monomials of
total degree less than $N$ span at most $N$ functions per amplitude term,
not $\binom{N+1}{2}$.  Further cancellation between coefficients can make
the rank of a particular approximant smaller still; the displayed span
dimensions are safe upper bounds, not claims of equality.

\item The weighted best-polynomial error is never larger than the centered
Taylor error, and the minimum in
\eqref{eq:sharp-analytic-chi} is never larger than either its Taylor or
Chebyshev--Bessel entry.  The Bessel entry by itself is not uniformly
superior for every $M$ and $N$.  Estimate
\eqref{eq:sharp-analytic-explicit-bessel-tail} shows its effective
$(M/2)^N/N!$ scale only once the stated denominator is positive.

\item The correct universal error for a product of separately approximated
factors is
$\prod_\ell(1+\delta_\ell)-1$, or the sharper ordered telescoping quantity
$\Delta$.  It is generally \emph{not} $\prod_\ell\delta_\ell$: that false
expression would vanish if one factor were exact even though another factor
still had a nonzero error.

\item Neither degree geometry dominates the other.  When all tensor degrees
are set equal to the total degree, $N_\ell=N$, the total-degree method has the
smaller raw rank $p\binom{N+r-1}{r}$ and can exploit cancellation in the
global half-oscillation $M$.  The tensor-degree method has raw rank
$p\prod_\ell N_\ell$, but it can assign very small degrees to weak or exact
interactions.  The final certificate should take the minimum of the original
factorial, centered best-polynomial, Chebyshev--Bessel, and anisotropic
product bounds.
\end{enumerate}
\end{remark}

\begin{remark}[What is, and is not, obtained by Fourier equivalence]
\label{rem:fourier-equivalent-transforms}
Some linear canonical and fractional Fourier transforms have kernels with a
quadratic phase and constant amplitude.  On parameter ranges where chirp
multiplications, reflections, and dilations give an exact unitary
factorization through the ordinary Fourier transform, any singular-value or
plunge estimate transfers by unitary equivalence, with the corresponding
rescaling of the windows.  That observation is useful, but it is a corollary
of the Fourier theorem rather than a new transform theorem.

By contrast, the nonlinear interaction $e^xe^y$ in
\eqref{eq:nonlinear-kernel-example} is not reduced by the chirp, reflection,
and dilation equivalences just described;
Proposition~\ref{prop:nonlinear-kernel-example} follows from the
finite-interaction theorem itself.  Hankel, Laplace, and Bargmann-type
transforms are not automatically covered either: their amplitudes, measures,
and often unbounded domains require separate approximation estimates.  We
therefore make no blanket claim for those transforms here.
\end{remark}

\section{Exact defect energy and high-cluster subtraction}
\label{sec:sharp-defect-energy}

This section strengthens the approximation certificates by retaining the
exact Hilbert--Schmidt energy of the defect.  For hard Fourier localization,
that energy has equivalent spatial and frequency translation formulas; it
also certifies an upper spectral cluster that can be subtracted from every
direct count.

\subsection{The exact defect trace for hard Fourier localization}
\label{sec:sharp-fourier-defect-trace}

The scalar $\mathcal D$ in \eqref{eq:sharp-m-D-definition} has several
equivalent exact geometric representations.  The spatial cross-boundary
identity below is the one-dimensional specialization of
\cite[Proposition~1.2]{HughesIsraelMayeli} under the Fourier normalization
used here.  We include the proof because it also yields the frequency-dual,
symmetric-difference, and rescaled forms used later.  These identities are
valid for arbitrary finite-measure windows; boundedness, interval structure,
and boundary regularity are not needed.

\begin{proposition}[Exact cross-boundary and symmetric-difference formulas]
\label{prop:sharp-fourier-defect-trace}
Let $A,B\subset\mathbb R$ be measurable with $|A|,|B|<\infty$, let $c>0$,
and set
\[
 C:=cB,
 \qquad
 K:=P_C\mathcal F P_A,
 \qquad
 S:=K^*K,
 \qquad
 m:=c|A||B|.
\]
Use the Fourier-transform convention
\[
 \widehat{\one_E}(u):=\int_Ee^{-2\pi iux}\,dx.
\]
Here and below $E-u:=\{x-u:x\in E\}$, with all sets understood modulo
null sets.
Then $K$ is Hilbert--Schmidt, $S$ is a positive trace-class contraction,
$\operatorname{Tr}S=m$, and
\begin{align}
 \mathcal D
 :=\operatorname{Tr}(S-S^2)
 &=\int_A\int_{A^c}
   \left|\int_Ce^{2\pi i(x-y)\xi}\,d\xi\right|^2dy\,dx
 \label{eq:sharp-defect-spatial-cross}\\
 &=\int_C\int_{C^c}
   \left|\widehat{\one_A}(\xi-\eta)\right|^2d\eta\,d\xi
 \label{eq:sharp-defect-frequency-cross}\\
 &=\frac12\int_{\mathbb R}
   |\widehat{\one_C}(u)|^2
   |A\mathbin{\triangle}(A-u)|\,du
 \label{eq:sharp-defect-A-symdiff}\\
 &=\frac12\int_{\mathbb R}
   |\widehat{\one_A}(v)|^2
   |C\mathbin{\triangle}(C-v)|\,dv.
 \label{eq:sharp-defect-C-symdiff}
\end{align}
The two correctly rescaled forms in terms of the undilated set $B$ are
\begin{align}
 \mathcal D
 &=\frac c2\int_{\mathbb R}
   |\widehat{\one_B}(z)|^2
   \left|A\mathbin{\triangle}
          \left(A-\frac zc\right)\right|\,dz,
 \label{eq:sharp-defect-A-scaled}\\
 \mathcal D
 &=\frac c2\int_{\mathbb R}
   |\widehat{\one_A}(v)|^2
   \left|B\mathbin{\triangle}
          \left(B-\frac vc\right)\right|\,dv.
 \label{eq:sharp-defect-B-scaled}
\end{align}
All identities are unchanged by null-set modifications of $A$ or $B$.
\end{proposition}

\begin{proof}
\emph{Step 1: Hilbert--Schmidt and trace-class facts.}
As an operator on $L^2(\mathbb R)$, $K$ has integral kernel
\[
 k(\xi,x)=\one_C(\xi)e^{-2\pi ix\xi}\one_A(x).
\]
Therefore Tonelli's theorem gives
\[
 \|K\|_{\mathcal S_2}^2
 =\int_C\int_A|e^{-2\pi ix\xi}|^2dx\,d\xi
 =|A||C|=c|A||B|=m.
\]
The Fourier transform and both projections are contractions, so $K$ is a
contraction.  Since $K$ is Hilbert--Schmidt, $S=K^*K$ is positive and trace
class, with $\operatorname{Tr}S=\|K\|_{\mathcal S_2}^2=m$.

\emph{Step 2: compute the spatial cross-boundary formula.}
Put
\[
 g_C(t):=\int_Ce^{2\pi it\xi}\,d\xi.
\]
The operator $S=P_A\mathcal F^{-1}P_C\mathcal F P_A$ has kernel
\[
 s(x,y)=\one_A(x)\one_A(y)g_C(x-y).
\]
Because $S$ is positive and Hilbert--Schmidt,
\begin{equation}
 \operatorname{Tr}(S^2)=\|S\|_{\mathcal S_2}^2
 =\int_A\int_A|g_C(x-y)|^2dy\,dx.
 \label{eq:sharp-trace-S-square}
\end{equation}
Plancherel's theorem gives
$\int_{\mathbb R}|g_C(t)|^2dt=|C|$.  Translation invariance then yields
\[
 m=|A||C|
 =\int_A\int_{\mathbb R}|g_C(x-y)|^2dy\,dx.
\]
Subtracting \eqref{eq:sharp-trace-S-square} from this identity, and using
nonnegativity to justify every decomposition by Tonelli, proves
\eqref{eq:sharp-defect-spatial-cross}.

\emph{Step 3: rewrite the spatial crossing as a symmetric difference.}
In \eqref{eq:sharp-defect-spatial-cross}, set $u=y-x$.  Since
$g_C(-u)=\widehat{\one_C}(u)$, Tonelli gives
\begin{align*}
 \mathcal D
 &=\int_{\mathbb R}|\widehat{\one_C}(u)|^2
   \left(\int_{\mathbb R}
     \one_A(x)\one_{A^c}(x+u)\,dx\right)du\\
 &=\int_{\mathbb R}|\widehat{\one_C}(u)|^2
   |A\setminus(A-u)|\,du.
\end{align*}
For any finite-measure set $E$, translation invariance gives
$|E|=|E-u|$, and hence
\[
 |E\setminus(E-u)|=|(E-u)\setminus E|
 =\frac12|E\mathbin{\triangle}(E-u)|.
\]
Applying this identity with $E=A$ proves
\eqref{eq:sharp-defect-A-symdiff}.

\emph{Step 4: repeat the computation on the frequency side.}
Let $\widetilde S:=KK^*$.  The nonzero eigenvalues of $K^*K$ and $KK^*$
agree, with multiplicity, so
\[
 \operatorname{Tr}(\widetilde S-\widetilde S^2)
 =\operatorname{Tr}(S-S^2)=\mathcal D.
\]
The kernel of
$\widetilde S=P_C\mathcal F P_A\mathcal F^{-1}P_C$ is
\[
 \widetilde s(\xi,\eta)
 =\one_C(\xi)\one_C(\eta)
   \widehat{\one_A}(\xi-\eta).
\]
Repeating Step 2 with $A$ and $C$ interchanged proves
\eqref{eq:sharp-defect-frequency-cross}.  Setting $v=\eta-\xi$ and using
$|\widehat{\one_A}(-v)|=|\widehat{\one_A}(v)|$ gives
\[
 \mathcal D
 =\int_{\mathbb R}|\widehat{\one_A}(v)|^2
   |C\setminus(C-v)|\,dv.
\]
The same equal-measure argument as in Step 3 proves
\eqref{eq:sharp-defect-C-symdiff}.

\emph{Step 5: rescale from $C$ to $B$.}
For $C=cB$ and $c>0$, a change of variables gives
\begin{equation}
 \widehat{\one_C}(u)=c\widehat{\one_B}(cu),
 \qquad
 |C\mathbin{\triangle}(C-v)|
 =c\left|B\mathbin{\triangle}
          \left(B-\frac vc\right)\right|.
 \label{eq:sharp-two-scaling-identities}
\end{equation}
Insert the first identity into
\eqref{eq:sharp-defect-A-symdiff}, then set $z=cu$; the factor $c^2$ from
the squared Fourier transform and the Jacobian $du=dz/c$ leave the factor
$c$ in \eqref{eq:sharp-defect-A-scaled}.  Inserting the second identity
directly into \eqref{eq:sharp-defect-C-symdiff} proves
\eqref{eq:sharp-defect-B-scaled}.  This completes all claimed identities.
\end{proof}

Combining Proposition~\ref{prop:sharp-fourier-defect-trace} with
\eqref{eq:sharp-zero-rank-two-counts} gives, without replacing translations
by a boundary-content majorant,
\begin{equation}
 \Lambda_\epsilon(A,cB)
 \leq
 \kappa\!\left(
   \frac{\mathcal D(A,cB)}{\epsilon(1-\epsilon)}
 \right),
 \label{eq:sharp-exact-Fourier-defect-count}
\end{equation}
where $\mathcal D(A,cB)$ may be evaluated by any of
\eqref{eq:sharp-defect-spatial-cross}--\eqref{eq:sharp-defect-B-scaled};
here $\mathcal D(A,cB)$ denotes the scalar $\mathcal D$ in
Proposition~\ref{prop:sharp-fourier-defect-trace}.

\subsection{Subtracting a certified high spectral cluster}
\label{sec:sharp-high-cluster}

The direct count $Q_\epsilon(S)$ includes both the transition spectrum and
the cluster at the upper endpoint.  Trace and defect information can certify
part of that upper cluster and subtract it from any direct upper bound.

\begin{theorem}[High-cluster subtraction]
\label{thm:sharp-high-cluster-subtraction}
Let $S$ be a positive trace-class contraction and put
\[
 m:=\operatorname{Tr}S,
 \qquad
 \mathcal D:=\operatorname{Tr}(S-S^2).
\]
For $0<\epsilon<1/2$, define
\begin{align}
 Q_\epsilon
 &:=\operatorname{rank}\one_{(\epsilon,\infty)}(S),
 &
 H_\epsilon
 &:=\operatorname{rank}\one_{[1-\epsilon,1]}(S),
 \label{eq:sharp-Q-H-definition}\\
 L_\epsilon
 &:=\left\lceil
       \left(m-\frac{\mathcal D}{\epsilon}\right)_+
      \right\rceil,
 &
 L_\epsilon^{\circ}
 &:=\max\left\{0,
       \left\lfloor m-\frac{\mathcal D}{\epsilon}\right\rfloor+1
      \right\}.
 \label{eq:sharp-high-cluster-lower-certificate}
\end{align}
Here $x_+:=\max\{x,0\}$.
Then
\begin{equation}
 H_\epsilon\geq L_\epsilon,
 \qquad
 \Lambda_\epsilon(S)=Q_\epsilon-H_\epsilon.
 \label{eq:sharp-cluster-two-facts}
\end{equation}
Consequently, every integer upper bound
$Q_\epsilon\leq\overline Q_\epsilon$ yields
\begin{equation}
 \Lambda_\epsilon(S)
 \leq\overline Q_\epsilon-L_\epsilon.
 \label{eq:sharp-high-cluster-subtraction}
\end{equation}
The open transition window gives the sharper unconditional endpoint rule
\begin{equation}
 \Lambda_\epsilon(S)
 \leq
 \max\{0,\overline Q_\epsilon-L_\epsilon^{\circ}\}
 \leq\overline Q_\epsilon-L_\epsilon.
 \label{eq:sharp-strict-high-cluster-subtraction}
\end{equation}
If $S=K^*K$ with $K$ Hilbert--Schmidt and $K_N$ is as in
Theorem~\ref{thm:sharp-residual-master}, then
\begin{equation}
 \Lambda_\epsilon(S)
 \leq
 \min\left\{
 N+\kappa\!\left(
       \frac{d_N^2}{\epsilon(1-\epsilon)}\right),
 \max\left\{0,
 N+\kappa\!\left(\frac{e_N^2}{\epsilon}\right)
 -L_\epsilon^{\circ}\right\}
 \right\}.
 \label{eq:sharp-master-with-subtraction}
\end{equation}
In particular, the zero approximant gives the fully trace-based certificate
\begin{equation}
 \Lambda_\epsilon(S)
 \leq
 \min\left\{
 \kappa\!\left(
       \frac{\mathcal D}{\epsilon(1-\epsilon)}\right),
 \max\left\{0,
 \kappa\!\left(\frac m\epsilon\right)-L_\epsilon^{\circ}
 \right\}
 \right\}.
 \label{eq:sharp-trace-only-hybrid}
\end{equation}
The outer maxima are necessary: $L_\epsilon^{\circ}$ is a conditional strict
lower certificate for the high cluster when the transition count is positive,
not an unconditional lower bound for $H_\epsilon$.
\end{theorem}

\begin{proof}
\emph{Step 1: prove a pointwise scalar inequality.}
For every $\lambda\in[0,1]$,
\begin{equation}
 \lambda
 \leq
 \frac{\lambda(1-\lambda)}{\epsilon}
 +\one_{[1-\epsilon,1]}(\lambda).
 \label{eq:sharp-pointwise-high-cluster}
\end{equation}
Indeed, if $0<\lambda<1-\epsilon$, then
$1-\lambda>\epsilon$, so
$\lambda(1-\lambda)/\epsilon>\lambda$.  At $\lambda=0$ both sides are
zero.  If $\lambda\in[1-\epsilon,1]$, the indicator is one and the right
side is at least one, hence at least $\lambda$.  Notice that
$\lambda=1-\epsilon$ belongs to the high cluster; this is required because
the transition interval is open at its upper endpoint.

\emph{Step 2: sum over the spectrum.}
Let $(\lambda_j)$ be the positive eigenvalues of $S$, repeated according to
multiplicity.  All relevant sums converge because $S$ is trace class and
$0\leq S-S^2\leq S$.  Summing
\eqref{eq:sharp-pointwise-high-cluster} gives
\[
 m=\sum_j\lambda_j
 \leq\frac1\epsilon\sum_j\lambda_j(1-\lambda_j)
      +H_\epsilon
 =\frac{\mathcal D}{\epsilon}+H_\epsilon.
\]
Therefore $H_\epsilon\geq m-\mathcal D/\epsilon$.  Since
$H_\epsilon$ is a nonnegative integer, it follows that
$H_\epsilon\geq L_\epsilon$.

\emph{Step 3: partition the direct spectral set with the correct endpoints.}
Because $0<\epsilon<1/2$,
\[
 (\epsilon,1]
 =(\epsilon,1-\epsilon)\,\dot\cup\,[1-\epsilon,1].
\]
As the spectrum of a contraction lies in $[0,1]$, the corresponding
spectral projections are orthogonal and add.  Taking their finite ranks
gives
$Q_\epsilon=\Lambda_\epsilon(S)+H_\epsilon$, which proves
\eqref{eq:sharp-cluster-two-facts}.  If
$Q_\epsilon\leq\overline Q_\epsilon$, subtracting
$H_\epsilon\geq L_\epsilon$ proves
\eqref{eq:sharp-high-cluster-subtraction}.

\emph{Step 4: exploit strictness when a transition eigenvalue exists.}
Put $r_\epsilon=m-\mathcal D/\epsilon$.  If
$\Lambda_\epsilon(S)>0$, then at least one eigenvalue lies in
$(\epsilon,1-\epsilon)$.  For that eigenvalue the scalar inequality
\eqref{eq:sharp-pointwise-high-cluster} is strict, because
$1-\lambda>\epsilon$ and hence
\[
 \lambda<\frac{\lambda(1-\lambda)}{\epsilon}.
\]
All other scalar gaps in \eqref{eq:sharp-pointwise-high-cluster} are
nonnegative.  Summing therefore gives the strict trace inequality
\[
 m<\frac{\mathcal D}{\epsilon}+H_\epsilon,
 \qquad\text{so}\qquad H_\epsilon>r_\epsilon.
\]
Since $H_\epsilon$ is a nonnegative integer,
$H_\epsilon\geq L_\epsilon^{\circ}$.  Thus, in the case
$\Lambda_\epsilon(S)>0$,
\[
 \Lambda_\epsilon(S)
 =Q_\epsilon-H_\epsilon
 \leq\overline Q_\epsilon-L_\epsilon^{\circ}.
\]
If instead $\Lambda_\epsilon(S)=0$, the same conclusion holds after the
right side is truncated below by zero.  This proves the first inequality in
\eqref{eq:sharp-strict-high-cluster-subtraction}.

To compare the two integer corrections, if $r_\epsilon<0$, then both
$L_\epsilon$ and $L_\epsilon^{\circ}$ are zero.  If
$r_\epsilon\geq0$ is nonintegral, both equal
$\lceil r_\epsilon\rceil$.  If $r_\epsilon$ is a nonnegative integer, then
$L_\epsilon^{\circ}=L_\epsilon+1$.  Hence
$L_\epsilon^{\circ}\geq L_\epsilon$.  Since
$\overline Q_\epsilon-L_\epsilon\geq0$, the second inequality in
\eqref{eq:sharp-strict-high-cluster-subtraction} follows.

\emph{Step 5: insert the residual-energy bounds.}
Equation \eqref{eq:sharp-direct-Q-bound} provides the integer upper bound
\[
 \overline Q_\epsilon
 =N+\kappa(e_N^2/\epsilon).
\]
Substitution into \eqref{eq:sharp-strict-high-cluster-subtraction} gives the
second entry of \eqref{eq:sharp-master-with-subtraction}; the first is the
weighted defect entry of \eqref{eq:sharp-direct-defect-master}.  Taking
$N=0$ and $K_0=0$ gives $e_0^2=m$ and $d_0^2=\mathcal D$, proving
\eqref{eq:sharp-trace-only-hybrid}.
\end{proof}

\subsection{Higher defect moments}
\label{sec:sharp-higher-defect-moments}

\begin{proposition}[Higher defect-moment certificates]
\label{prop:sharp-higher-defect-moments}
Let $S$ be a positive contraction, let $q>0$, and assume that
$(S-S^2)^q$ is trace class.  Define
\[
 \mathcal D_q:=\operatorname{Tr}\!\left((S-S^2)^q\right).
\]
Then, for every $0<\epsilon<1/2$,
\begin{equation}
 \Lambda_\epsilon(S)
 \leq
 \kappa\!\left(
   \frac{\mathcal D_q}{[\epsilon(1-\epsilon)]^q}
 \right).
 \label{eq:sharp-higher-defect-count}
\end{equation}
If $S$ is trace class, the hypothesis is automatic for every $q\geq1$, and
\begin{equation}
 \mathcal D_q\leq4^{1-q}\mathcal D.
 \label{eq:sharp-higher-defect-finite}
\end{equation}
Thus the family \eqref{eq:sharp-higher-defect-count} may be optimized over
all defect moments that are known or computable; no monotone improvement in
$q$ is asserted without further spectral information.
\end{proposition}

\begin{proof}
Let $t^2:=\epsilon(1-\epsilon)$ and put
\[
 E:=\one_{(\epsilon,1-\epsilon)}(S),
 \qquad
 A_q:=(S-S^2)^q.
\]
Functional calculus gives $A_qE\geq[t^2]^qE$.  Since $A_q$ is trace class,
so is its positive compression $A_qE$; the preceding lower bound therefore
forces $E$ to have finite rank.  Write
$L:=\operatorname{rank}E=\Lambda_\epsilon(S)$.

If $L>0$, the restriction of $S$ to $\operatorname{Ran}E$ is a
finite-dimensional positive contraction with eigenvalues
$\lambda_1,\ldots,\lambda_L\in(\epsilon,1-\epsilon)$.  Each satisfies
$[\lambda_j(1-\lambda_j)]^q>[t^2]^q$, and hence
\[
 L[t^2]^q
 <\sum_{j=1}^{L}[\lambda_j(1-\lambda_j)]^q
 =\operatorname{Tr}(A_qE)
 \leq\operatorname{Tr}(A_q)=\mathcal D_q.
\]
Thus the nonnegative integer $L$ is strictly smaller than
$\mathcal D_q/[t^2]^q$, and the definition
\eqref{eq:sharp-kappa-definition} gives
\eqref{eq:sharp-higher-defect-count}.  If $L=0$, the same conclusion is
immediate.

Finally, $0\leq S-S^2\leq\tfrac14I$.  For $q\geq1$, scalar functional
calculus gives
\[
 (S-S^2)^q\leq\left(\frac14\right)^{q-1}(S-S^2).
\]
If $S$ is trace class, then $0\leq S-S^2\leq S$ is trace class.  Taking
traces in the preceding inequality proves both the asserted automatic
trace-class property and \eqref{eq:sharp-higher-defect-finite}.
\end{proof}

\section{Synthesis: hybrid certificates, comparison, and limitations}
\label{sec:synthesis}

Several upper bounds are retained because they encode different information.
This section collects them in a single statement and then identifies the
precise extension beyond the regular-domain literature.

\subsection{A hybrid envelope}

For a null measurable set we use the convention
$\operatorname{diam}_{\rm ess}(E)=0$.  For bounded measurable $A,B$ and
$c>0$, define
\[
 m=c|A||B|,
 \qquad
 \zeta=\frac{\pi c}{2}
       \operatorname{diam}_{\rm ess}(A)
       \operatorname{diam}_{\rm ess}(B),
\]
and let
\begin{equation}
 \mathfrak B_{\rm fac}(\epsilon;A,B,c)
 :=\min\left\{N\in\mathbb N_0:
       \sqrt m\frac{\zeta^N}{N!}\le\sqrt\epsilon\right\}.
 \label{eq:synthesis-Bfac}
\end{equation}
At order $N=0$ we use the convention $\zeta^0=1$, including when
$\zeta=0$.  The partition envelope $\mathfrak B_{\rm mat}$ was defined in
\eqref{eq:partition-matrix-envelope}.  For brevity in this section write
\[
 \mathfrak B_{\rm fac}:=\mathfrak B_{\rm fac}(\epsilon;A,B,c),
 \qquad
 \mathfrak B_{\rm mat}:=\mathfrak B_{\rm mat}(\epsilon;A,B,c).
\]
When
$A=\dot\bigcup_{i=1}^M I_i$ and
$B=\dot\bigcup_{j=1}^K J_j$ are nonnull finite unions of bounded intervals
of positive length, put
\begin{equation}
 \mathfrak B_{\rm Sch}
 :=2\inf_{0<p\le1}\frac{[\epsilon(1-\epsilon)]^{-p/2}}p
   \sum_{i=1}^M\sum_{j=1}^K
   \left[8+\Csharp+G_*(c|I_i||J_j|p)\right]
 \label{eq:synthesis-Bsch}
\end{equation}
and, for this fixed component partition, define
\begin{equation}
 \mathfrak B_{\rm KF}
 :=\inf_{\substack{0<\epsilon_{ij}<1/2\\
       \sum_{i,j}\sqrt{\epsilon_{ij}(1-\epsilon_{ij})}
       \le\sqrt{\epsilon(1-\epsilon)}}}
       \sum_{i,j}\Lambda_{\epsilon_{ij}}(I_i,cJ_j).
 \label{eq:synthesis-BKF}
\end{equation}
Any explicit proved local upper bound may replace the exact local counts in
\eqref{eq:synthesis-BKF}, producing a possibly larger explicit certificate.

\begin{theorem}[Hybrid upper envelope]
\label{thm:hybrid-envelope}
For bounded measurable $A,B$, $c>0$, and $0<\epsilon<1/2$,
\begin{equation}
 \Lambda_\epsilon(A,cB)
 \le \mathfrak B_{\rm mat}
 \le \mathfrak B_{\rm fac}.
 \label{eq:synthesis-measurable-min}
\end{equation}
If $A$ and $B$ are the finite interval unions above, then also
\begin{equation}
 \Lambda_\epsilon(A,cB)
 \le
 \min\left\{
   \left\lfloor\mathfrak B_{\rm Sch}\right\rfloor,
   \mathfrak B_{\rm mat},
   \mathfrak B_{\rm KF}
 \right\}.
 \label{eq:synthesis-hybrid-min}
\end{equation}
\end{theorem}

\begin{proof}
Each entry follows from an independently valid estimate.

First, Corollary~\ref{cor:local-taylor-matrix} bounds the plunge count by
$\sum_{i,j}N_{ji}$ for every finite measurable partition and every integer
array satisfying the matrix-error constraint.  Taking the infimum over all
such choices proves the first inequality in
\eqref{eq:synthesis-measurable-min}.

Second, take the one-piece partitions $A_1=A$ and $B_1=B$.  At a common order
$N$, the scalar error matrix has one entry,
\[
 \mathbf E_{11}=\sqrt m\frac{\zeta^N}{N!}.
\]
The order $N=\mathfrak B_{\rm fac}$ is therefore feasible in the partition
optimization.  This proves
$\mathfrak B_{\rm mat}\le\mathfrak B_{\rm fac}$.

For interval unions, Theorem~\ref{thm:sch-componentwise} gives
$\Lambda_\epsilon(A,cB)\le\mathfrak B_{\rm Sch}$.  Since the left side is an
integer, it is at most $\lfloor\mathfrak B_{\rm Sch}\rfloor$.
The matrix entry was already proved.  Finally,
Theorem~\ref{thm:weighted-measurable-partition} bounds the count by the sum in
\eqref{eq:synthesis-BKF} for every feasible threshold array.  Taking its
infimum gives the third entry.  Since all three inequalities hold
simultaneously, their minimum is also a valid bound.
\end{proof}

The first inequality in \eqref{eq:synthesis-measurable-min} is a formal
dominance statement over the one-piece \emph{factorial certificate}; it is not
a statement that partitions improve every known estimate.  Likewise,
$\mathfrak B_{\rm KF}$ preserves exact local counts but remains computationally
implicit until those counts are bounded.  Formula
\eqref{eq:synthesis-hybrid-min} records simultaneous certified bounds; it is
not a universal optimizer.

\subsection{The strengthened hybrid envelope}
\label{sec:synthesis-strengthened-hybrid}

The factorial and matrix entries in the preceding envelope used a hard
operator-norm cutoff.  Residual energy, the exact Fourier defect trace,
optimized local polynomials, and the orthogonal two-level partition budgets
give additional independently valid entries.  We assemble them and verify
entry by entry that the result formally improves the certificates in
Theorem~\ref{thm:hybrid-envelope}.

For a finite bounded measurable partition and an order array
$\mathbf N=(N_{ji})$, set
\begin{equation}
 R(\mathbf N):=\sum_{j,i}N_{ji},
 \qquad
 H_{\rm best}(\mathbf N)
 :=\|\mathbf E^{\rm best}(\mathbf N)\|_{\rm F}^{2}
 =\sum_{j,i}\mathbf E^{\rm best}(\mathbf N)_{ji}^{2},
 \label{eq:synthesis-energy-partition-data}
\end{equation}
where $\|\cdot\|_{\rm F}$ is the scalar Frobenius norm and
$\mathbf E^{\rm best}$ is defined in \eqref{eq:best-poly-matrix}.  Define
the partition residual-energy envelope
\begin{equation}
 \mathfrak B_{\rm en}
 :=\inf_{\substack{A=\dot\bigcup_i A_i,\ B=\dot\bigcup_j B_j\\
                    \text{finite measurable partitions}\\
                    N_{ji}\in\mathbb N_0}}
 \left[
   R(\mathbf N)
   +\kappa\!\left(
       \frac{H_{\rm best}(\mathbf N)}{\epsilon}
     \right)
 \right].
 \label{eq:synthesis-energy-partition-envelope}
\end{equation}
The trivial one-piece partition is included in this infimum.

Next fix finite partitions $A=\dot\bigcup_iA_i$ and
$B=\dot\bigcup_jB_j$.  In the interval-union comparison below, these are
the component partitions fixed in \eqref{eq:synthesis-Bsch} and
\eqref{eq:synthesis-BKF}.  Put
$t_\epsilon=\sqrt{\epsilon(1-\epsilon)}$, and use the function $\eta$ from
\eqref{eq:eta-recalled-two-level}.  Let $\mathfrak B_{\rm row,2}$ be the
infimum of
\[
 \sum_{i,j}\Lambda_{\eta(\tau_{ij})}(A_i,cB_j)
\]
over positive arrays satisfying
\[
 \sum_i\tau_i^2\leq t_\epsilon^2,
 \qquad
 \sum_j\tau_{ij}\leq\tau_i\quad\text{for every }i.
\]
Define $\mathfrak B_{\rm col,2}$ by the Fourier-dual column constraints in
\eqref{eq:two-level-column-budget}, and put
\begin{equation}
 \mathfrak B_{\rm KF,2}
 :=\min\{\mathfrak B_{\rm row,2},\mathfrak B_{\rm col,2}\}.
 \label{eq:synthesis-BKF2}
\end{equation}
As with $\mathfrak B_{\rm KF}$, explicit proved local upper bounds may
replace the exact local counts, at the cost of a possibly larger certificate.

In the interval-union case, the real-valued Schatten expression admits a
strict integer correction:
\begin{equation}
 \mathfrak B_{\rm Sch}^{\#}
 :=\kappa(\mathfrak B_{\rm Sch})
 =\max\{0,\lceil\mathfrak B_{\rm Sch}\rceil-1\}.
 \label{eq:synthesis-Bsch-sharp}
\end{equation}

Finally, put $K=K_{A,B,c}$, $S=K^*K$, and
\[
 m=\operatorname{Tr}S=c|A||B|,
 \qquad
 \mathcal D=\operatorname{Tr}(S-S^2),
 \qquad
 W=(I-KK^*)^{1/2}.
\]
Let $\mathcal A$ be any nonempty collection of pairs $(N,K_N)$ such that
$N\in\mathbb N_0$, $\operatorname{rank}K_N\leq N$, and
$K-K_N\in\mathcal S_2$.  We require that $(0,0)\in\mathcal A$.  For such a
family define
\begin{align}
 \mathfrak B_{\rm dir}(\mathcal A)
 &:=\inf_{(N,K_N)\in\mathcal A}
 \left[N+\kappa\!\left(
   \frac{\|K-K_N\|_{\mathcal S_2}^2}{\epsilon}
 \right)\right],
 \label{eq:synthesis-Bdir-family}\\
 \mathfrak B_{\rm wdef}(\mathcal A)
 &:=\inf_{(N,K_N)\in\mathcal A}
 \left[N+\kappa\!\left(
   \frac{\|W(K-K_N)\|_{\mathcal S_2}^2}
        {\epsilon(1-\epsilon)}
 \right)\right].
 \label{eq:synthesis-Bwdef-family}
\end{align}
The zero approximant makes both infima finite.  Each is an infimum of
nonnegative integers and therefore is itself a nonnegative integer.

Set
\begin{align}
 \mathfrak B_{\rm def}
 &:=\kappa\!\left(
       \frac{\mathcal D}{\epsilon(1-\epsilon)}
    \right),
 &
 L_\epsilon^{\circ}
 &:=\max\left\{0,
       \left\lfloor m-\frac{\mathcal D}{\epsilon}\right\rfloor+1
    \right\},
 \label{eq:synthesis-def-strict-L}\\
 \overline Q_\epsilon^{\star}(\mathcal A)
 &:=\min\left\{
   \mathfrak B_{\rm dir}(\mathcal A),
   \mathfrak B_{\rm en},
   \mathfrak B_{\rm best}
 \right\},
 &
 \mathfrak B_{\rm bulk}^{\star}(\mathcal A)
 &:=\max\left\{0,
   \overline Q_\epsilon^{\star}(\mathcal A)-L_\epsilon^{\circ}
 \right\}.
 \label{eq:synthesis-Qstar-Bbulkstar}
\end{align}
The family $\mathcal A$ can contain whichever approximants are actually
available: Taylor, least-squares product polynomials, singular-value
truncations, core--tail approximants, or their union.  Enlarging
$\mathcal A$ can only improve the two family envelopes.  The weighted norm
in \eqref{eq:synthesis-Bwdef-family} is global; no blockwise orthogonality is
assumed after applying $W$.

\begin{proposition}[Strict integer Schatten correction]
\label{prop:synthesis-strict-schatten}
Let $c>0$, $0<\epsilon<1/2$, and let $A,B$ be nonnull finite unions of
bounded intervals of positive length.  Then
\begin{equation}
 \Lambda_\epsilon(A,cB)
 \leq\mathfrak B_{\rm Sch}^{\#}
 \leq\lfloor\mathfrak B_{\rm Sch}\rfloor.
 \label{eq:synthesis-strict-schatten-dominance}
\end{equation}
The second inequality is strict exactly when
$\mathfrak B_{\rm Sch}$ is an integer.
\end{proposition}

\begin{proof}
For fixed $p\in(0,1]$, denote the right side of the fixed-$p$ estimate in
Theorem~\ref{thm:sch-componentwise} by
\[
 \Phi(p):=
  \frac{2t_\epsilon^{-p}}p
  \sum_{i,j}\bigl[8+\Csharp+G_*(w_{ij}p)\bigr].
\]
The explicit two-branch formula \eqref{eq:sch-Gstar-formula} shows that
$\Phi$ is continuous on $(0,1]$.  Every summand is positive and the factor
$1/p$ diverges, so $\Phi(p)\to\infty$ as $p\downarrow0$.  Hence there is a
$\delta>0$ such that the infimum over $(0,1]$ equals the minimum of the
continuous function $\Phi$ on $[\delta,1]$.  Choose a minimizer $p_*$.  By
definition, $\Phi(p_*)=\mathfrak B_{\rm Sch}$.

Let $T=P_{A^c}Q_{cB}P_A$ and
$q=\Lambda_\epsilon(A,cB)=n(t_\epsilon;T)$.  If $q>0$, every counted
singular value is strictly larger than $t_\epsilon$, and therefore
\[
 q\,t_\epsilon^{p_*}
 <\sum_{s_k(T)>t_\epsilon}s_k(T)^{p_*}
 \leq\|T\|_{p_*}^{p_*}.
\]
The fixed-$p_*$ estimate in the proof of
Theorem~\ref{thm:sch-componentwise} now gives
$q<\Phi(p_*)=\mathfrak B_{\rm Sch}$.  The largest nonnegative integer
strictly smaller than a positive real number $x$ is
$\lceil x\rceil-1$; the case $q=0$ is immediate.  Thus
$q\leq\kappa(\mathfrak B_{\rm Sch})$.  Finally,
$\lceil x\rceil-1=\lfloor x\rfloor$ unless $x$ is an integer, in which case
it equals $\lfloor x\rfloor-1$.  This proves all assertions.
\end{proof}

\begin{theorem}[Strengthened hybrid upper envelope]
\label{thm:synthesis-strengthened-hybrid}
Let $A,B$ be bounded measurable sets, let $c>0$, and let
$0<\epsilon<1/2$.  For every approximant family $\mathcal A$ described
above,
\begin{equation}
 \boxed{
 \Lambda_\epsilon(A,cB)
 \leq
 \min\left\{
   \mathfrak B_{\rm wdef}(\mathcal A),
   \mathfrak B_{\rm bulk}^{\star}(\mathcal A)
 \right\}.}
 \label{eq:synthesis-strengthened-measurable}
\end{equation}
If $A$ and $B$ are the nonnull finite interval unions used in
\eqref{eq:synthesis-Bsch} and \eqref{eq:synthesis-BKF}, one additionally has
\begin{equation}
 \boxed{
 \Lambda_\epsilon(A,cB)
 \leq
 \min\left\{
   \mathfrak B_{\rm wdef}(\mathcal A),
   \mathfrak B_{\rm bulk}^{\star}(\mathcal A),
   \mathfrak B_{\rm KF,2},
   \mathfrak B_{\rm Sch}^{\#}
 \right\}.}
 \label{eq:synthesis-strengthened-interval}
\end{equation}
Moreover, the comparisons below involving only the bounded measurable
quantities hold in general; the two comparisons involving
$\mathfrak B_{\rm KF}$ or $\mathfrak B_{\rm Sch}$ apply in the
interval-union case.
\begin{align}
 \mathfrak B_{\rm bulk}^{\star}(\mathcal A)
 &\leq\overline Q_\epsilon^{\star}(\mathcal A)
 \leq\mathfrak B_{\rm best}
 \leq\mathfrak B_{\rm mat},
 &
 \overline Q_\epsilon^{\star}(\mathcal A)
 &\leq\mathfrak B_{\rm en}
 \leq\mathfrak B_{\rm fac},
 \notag\\
 \mathfrak B_{\rm wdef}(\mathcal A)
 &\leq\mathfrak B_{\rm def},
 &
 \mathfrak B_{\rm KF,2}
 &\leq\mathfrak B_{\rm KF},
 \notag\\
 \mathfrak B_{\rm Sch}^{\#}
 &\leq\lfloor\mathfrak B_{\rm Sch}\rfloor.
 \label{eq:synthesis-four-formal-dominances}
\end{align}
Consequently the strengthened bounded-window certificate is no larger than
\eqref{eq:synthesis-measurable-min}, and the strengthened interval-union
certificate is no larger than \eqref{eq:synthesis-hybrid-min}.
\end{theorem}

\begin{proof}
\emph{Step 1: prove the residual-energy partition entry.}
For every block $(j,i)$ and every $\delta>0$, the definition of
$\mathbf E^{\rm best}$ supplies a rank-at-most-$N_{ji}$ block approximant
whose Hilbert--Schmidt residual norm is at most
$\mathbf E^{\rm best}_{ji}+\delta$.  The block kernels are supported on the
pairwise disjoint rectangles $cB_j\times A_i$ in output--input coordinates.
Their Hilbert--Schmidt inner
products therefore vanish across distinct pairs.  On assembling the local
approximants, the total rank is at most $R(\mathbf N)$ and the squared
Hilbert--Schmidt residual tends, as $\delta\downarrow0$, to
$H_{\rm best}(\mathbf N)$.

Equivalently, the Hilbert--Schmidt Eckart--Young identity gives
\[
 \sum_{k>R(\mathbf N)}s_k(K_{A,B,c})^2
 \leq H_{\rm best}(\mathbf N).
\]
Apply Lemma~\ref{lem:sharp-residual-energy} to a best rank-$R(\mathbf N)$
singular-value truncation of $K_{A,B,c}$ at threshold $\sqrt\epsilon$,
and use the monotonicity of $\kappa$.  It follows that both
$Q_\epsilon(S)$ and $\Lambda_\epsilon(S)$ are at most
\[
 R(\mathbf N)+
 \kappa\!\left(H_{\rm best}(\mathbf N)/\epsilon\right).
\]
Taking the infimum proves the $\mathfrak B_{\rm en}$ direct-count entry and
justifies its inclusion in
$\overline Q_\epsilon^{\star}(\mathcal A)$ in
\eqref{eq:synthesis-Qstar-Bbulkstar}.

\emph{Step 2: prove the optimized matrix and two-level entries.}
For every array feasible in the definition of
$\mathfrak B_{\rm best}$, \eqref{eq:optimized-block-sv} gives
$s_{R(\mathbf N)+1}(K)\leq\sqrt\epsilon$.  Thus
$Q_\epsilon(S)=n(\sqrt\epsilon;K)\leq R(\mathbf N)$.  Taking the integer
infimum gives
\[
 Q_\epsilon(S)\leq\mathfrak B_{\rm best},
\]
which is stronger than the transition-only conclusion in
\eqref{eq:optimized-envelope-chain}.  The row and column parts of
Theorem~\ref{thm:two-level-partition-allocation} separately give
$\Lambda_\epsilon\leq\mathfrak B_{\rm row,2}$ and
$\Lambda_\epsilon\leq\mathfrak B_{\rm col,2}$.  Their minimum is therefore
valid.

\emph{Step 3: prove the weighted-defect and strict bulk entries.}
For each $(N,K_N)\in\mathcal A$, Theorem
\ref{thm:sharp-residual-master} gives
\begin{align*}
 Q_\epsilon(S)
 &\leq N+\kappa\!\left(
   \frac{\|K-K_N\|_{\mathcal S_2}^2}{\epsilon}\right),\\
 \Lambda_\epsilon(S)
 &\leq N+\kappa\!\left(
   \frac{\|W(K-K_N)\|_{\mathcal S_2}^2}
        {\epsilon(1-\epsilon)}\right).
\end{align*}
Taking the two infima proves
$Q_\epsilon(S)\leq\mathfrak B_{\rm dir}(\mathcal A)$ and
$\Lambda_\epsilon(S)\leq\mathfrak B_{\rm wdef}(\mathcal A)$.  Because
$(0,0)\in\mathcal A$ and
$\|WK\|_{\mathcal S_2}^2=\mathcal D$, the latter also proves
$\mathfrak B_{\rm wdef}(\mathcal A)\leq\mathfrak B_{\rm def}$.

Step~1 gives $Q_\epsilon(S)\leq\mathfrak B_{\rm en}$, and Step~2 gives
$Q_\epsilon(S)\leq\mathfrak B_{\rm best}$.  Hence
$Q_\epsilon(S)\leq\overline Q_\epsilon^{\star}(\mathcal A)$.
The three entries defining $\overline Q_\epsilon^{\star}(\mathcal A)$ are
finite nonnegative integers.  This was noted for
$\mathfrak B_{\rm dir}(\mathcal A)$; the other two are infima of nonempty
subsets of $\mathbb N_0$, since the one-piece Taylor construction makes their
defining sets nonempty.  Thus
$\overline Q_\epsilon^{\star}(\mathcal A)$ is an integer direct-count bound.
Apply the strict high-cluster subtraction
\eqref{eq:sharp-strict-high-cluster-subtraction} to obtain
\[
 \Lambda_\epsilon(S)
 \leq\mathfrak B_{\rm bulk}^{\star}(\mathcal A).
\]
Proposition~\ref{prop:synthesis-strict-schatten} gives the last
interval-union entry.  Since every listed inequality holds simultaneously,
their minimum is valid.

\emph{Step 4: prove formal dominance over the old envelope.}
The inequalities
$\mathfrak B_{\rm bulk}^{\star}\leq
 \overline Q_\epsilon^{\star}\leq
 \mathfrak B_{\rm best}\leq\mathfrak B_{\rm mat}$ and
$\overline Q_\epsilon^{\star}\leq\mathfrak B_{\rm en}$ follow from the
definitions, from $L_\epsilon^{\circ}\geq0$, and from
\eqref{eq:optimized-envelope-chain}.  The weighted-defect comparison follows
from the zero-approximant argument in Step~3.  For the Ky--Fan comparison,
start with any flat allocation feasible in \eqref{eq:synthesis-BKF}, set
$u_{ij}=\sqrt{\epsilon_{ij}(1-\epsilon_{ij})}$ and
$\tau_i=\sum_j u_{ij}$, $\tau_{ij}=u_{ij}$.  Then
\[
 \left(\sum_i\tau_i^2\right)^{1/2}
 \leq\sum_i\tau_i
 =\sum_{i,j}u_{ij}
 \leq t_\epsilon.
\]
Since $0<\epsilon_{ij}<1/2$, the inverse relation
\eqref{eq:eta-recalled-two-level} gives
$\eta(u_{ij})=\epsilon_{ij}$.  Thus the same exact local right side is
feasible for the row envelope.  Taking infima gives
$\mathfrak B_{\rm KF,2}\leq\mathfrak B_{\rm KF}$.  The comparison remains
valid if the same local upper surrogates are substituted on both sides; it
need not hold if only one side is enlarged.  The Schatten inequality is
Proposition~\ref{prop:synthesis-strict-schatten}.

It remains to prove
$\mathfrak B_{\rm en}\leq\mathfrak B_{\rm fac}$.  Choose the trivial
one-piece partition and the order $N=\mathfrak B_{\rm fac}$.  The Taylor polynomial is an admissible
local polynomial, so
\[
 H_{\rm best}(N)
 \leq m\left(\frac{\zeta^N}{N!}\right)^2
 \leq\epsilon.
\]
Because $\kappa(r)=0$ for $0\leq r\leq1$, this choice gives
$\mathfrak B_{\rm en}\leq N=\mathfrak B_{\rm fac}$.  Thus each old entry
has a no-larger new counterpart.  Adding further independently valid entries
can only decrease the minimum, completing the proof.
\end{proof}

The comparisons in \eqref{eq:synthesis-four-formal-dominances} are internal,
theorem-to-theorem
dominance statements.  They do not assert numerical dominance over every
regular-domain estimate in the literature.  The exact defect trace and the
high-cluster subtraction add information of a different kind, so neither is
universally ordered against the polynomial, partition, or Schatten entries.

\subsection{What happens to the boundary functions}

The spectral boundaries do not change: the exact defect route always uses
$\sqrt{\epsilon(1-\epsilon)}$, and the direct route always uses
$\sqrt\epsilon$.  What changes is the function that bounds the count after a
kernel approximation is inverted.

\begin{itemize}
 \item The mass bound has a hard zero regime.  If
       $c|A||B|\le\epsilon$, then
       $\Lambda_\epsilon(A,cB)=0$.  This follows already from the order-zero
       factorial certificate and is consistent with
       $\|S\|\leq\operatorname{Tr}S=c|A||B|$.

 \item For a factorial error $C\zeta^N/N!$, the exact integer boundary is the
       least admissible $N$.  When $C>\tau$ and $\zeta>0$, its closed
       surrogate is
       \[
          N_W=\left\lceil
          \frac{\log(C/\tau)}
          {W_0(\log(C/\tau)/(e\zeta))}
          \right\rceil.
       \]
       If $C\le\tau$, order zero suffices; if $C>\tau$ and $\zeta=0$, order
       one suffices.
       With fixed positive $C$ and $\zeta$ and
       $\tau=\sqrt\epsilon$, this is
       $O(\log(1/\epsilon)/\log\log(1/\epsilon))$, rather than a bare linear
       function of $\log(1/\epsilon)$ for this factorial certificate.

 \item For moment-controlled unbounded windows, the boundary is determined by
       the growth of the moments.  The rough hard window $A_*$ and the
       super-Gaussian soft mask both have
       $\log s_{N+1}\le-\tfrac12N\log N+O(N)$ and hence the same
       logarithmic-over-double-logarithmic certificate.  Other moment growth
       laws lead to other functions; no universal moment rate is asserted.

 \item For a contraction in the finite-interaction class with $r$ bounded
       separated interactions, a total-degree order $N$ has rank
       $p\binom{N+r-1}{r}$.  For fixed $p,r,C,\zeta>0$ and
       $\epsilon\downarrow0$, the Lambert inversion gives
       $N=O(\log(1/\epsilon)/\log\log(1/\epsilon))$ and therefore
       \begin{equation}
        \Lambda_\epsilon(K^*K)
        =O\!\left[
          \left(\frac{\log(1/\epsilon)}
                     {\log\log(1/\epsilon)}\right)^r
        \right].
        \label{eq:synthesis-analytic-rate}
       \end{equation}
       The power $r$ arises from counting multivariate interaction
       monomials; for the single Fourier interaction one has $r=1$.

 \item If the approximation error is exponential, stretched exponential, or
       algebraic rather than factorial, Corollary~\ref{cor:abstract-other-laws}
       gives, respectively, logarithmic, a power of a logarithm, or a power of
       $1/\tau$.  The spectral transfer is unchanged, while the resulting
       boundary function depends on the approximation class.
\end{itemize}

To justify the fixed-parameter order used above, let
$L=\log(1/\epsilon)$.  Then
$\rho=\log(C/\sqrt\epsilon)=L/2+O(1)$ and
$W_0(\rho/(e\zeta))\sim\log L$.  Hence
$N_W=O(L/\log L)$.  Finally,
$\binom{N+r-1}{r}=O_r(N^r)$, which proves
\eqref{eq:synthesis-analytic-rate}.

\subsection{Position relative to the literature}

The interval/prolate theory provides sharper structural and asymptotic
information in its natural regimes
\cite{LandauWidom,KarnikRombergDavenport,BonamiJamingKaroui,
AzimifardDeterminant}.  The one-dimensional componentwise Schatten result here
uses the block estimate from \cite{AzimifardIndependent}; that input is
cited, not relabeled as a new proof.  Kulikov and Dam Larsen obtain sharp
uniform bounds when one window is a finite union of parallelepipeds and the
other has finite upper Minkowski boundary content, together with a bound
containing one additional logarithmic factor when both boundaries satisfy
that condition \cite{KulikovLarsen}.  Hughes, Israel, and Mayeli
give the cross-boundary defect-trace identity used in
Proposition~\ref{prop:sharp-fourier-defect-trace} and derive trace bounds for
rough domains with finite positive-order translation perimeters
\cite{HughesIsraelMayeli}.  Section~\ref{sec:sharp-fourier-defect-trace}
records the one-dimensional frequency-dual, symmetric-difference, and scaled
forms needed here.  In dimensions $d\ge2$, Marceca,
Romero, and Speckbacher treat compact domains under maximally Ahlfors-regular
boundary assumptions \cite{MarcecaRomeroSpeckbacher}.  Structural work on
truncated Fourier operators provides additional background
\cite{KatsnelsonMachlufI}.

The set $A_*$ in Theorem~\ref{thm:rough-window} is unbounded, has infinitely
many interval components, has no finite global boundary-tube parameter, and
has infinite $\operatorname{Per}_\gamma$ for every $0<\gamma\le1$.  Hence it
is outside (i) the finite-interval-union theorem used here, (ii) the bounded
finite-upper-Minkowski-content regime of \cite{KulikovLarsen}, and (iii) the
finite-positive-order-perimeter regime of \cite{HughesIsraelMayeli}.
Nevertheless, all its polynomial moments are finite, so
Theorem~\ref{thm:moment-hard} gives an explicit deep-threshold rate.  This
establishes noncoverage at the level of the stated hypotheses for these three
regimes; it does not imply that no earlier conclusion applies.
Hilbert--Schmidt counting and qualitative results for arbitrary finite-measure
windows remain relevant.

The finite-interaction theorem is similarly precise about replacing Fourier.
It covers phases with finitely many bounded separated interactions and
finite-separated-rank amplitudes.  Separated representations also appear in
the numerical analysis of Fourier integral operators
\cite{CandesDemanetYing}, but that literature is not being cited as a prior
plunge-count theorem.  On parameter ranges where a linear canonical or
fractional Fourier transform admits an exact factorization into chirps,
reflections, dilations, and the ordinary Fourier transform, the estimates
transfer by unitary equivalence with rescaled windows.  Hankel, Laplace,
Bargmann, and arbitrary oscillatory transforms require their own amplitude and
tail analysis.

\subsection{Limitations and conclusion}

All principal results are upper bounds.  We do not prove matching lower bounds
for rough windows, optimal constants for finite partitions, or a universal
best partition.  The componentwise Schatten theorem remains restricted to
finite interval unions because its imported one-sided estimate is assembled
through interval tails.  The diameter theorem needs bounded essential hulls;
the moment theorem removes that assumption only when sufficiently high moments
are finite.  The partition results are finite: a countable decomposition would
require additional summability and convergence hypotheses.  Soft-mask
estimates require integrability; unbounded-mask estimates additionally require
the stated moment or truncation hypotheses.  The finite-interaction
approximation theorem requires essentially bounded centered features, finite
separated amplitude rank, and finite weighted $L^2$
amplitude norm; its plunge-count conclusions additionally require the operator
to be a contraction.

Within those limits, the framework separates three questions cleanly: the
spectral transfer, the approximation law, and the geometric assembly.  This
separation is what permits the same proof architecture to retain component
sizes, exploit sparse measurable windows, treat unbounded moment-controlled
sets, and pass to a stated nonlinear class beyond Fourier without overstating
its reach.

\section*{AI use disclosure}
Generative AI tools were used during manuscript preparation for language
editing, copy-editing, bibliographic cross-checking, and limited symbolic and
numerical consistency checks. Selected algebraic identities and
finite-dimensional calculations were also checked in Lean~4. These checks
covered only local statements and were used as supporting verification, not as
substitutes for mathematical proof; they do not constitute a formal
verification of the main theorems or the analytic arguments. The author
reviewed the resulting manuscript and accepts full responsibility for all
statements, derivations, citations, and conclusions.

\bibliographystyle{plain}
\bibliography{references}

@misc{AzimifardIndependent,
  author        = {Azimifard, A.},
  title         = {Scale-free {Hankel} factorization and explicit one-dimensional plunge bounds for time-frequency localization operators},
  year          = {2026},
  eprint        = {2607.23016v3},
  archiveprefix = {arXiv},
  primaryclass  = {math.FA},
  note          = {arXiv:2607.23016v3 [math.FA]},
  doi           = {10.48550/arXiv.2607.23016}
}

@misc{AzimifardDeterminant,
  author        = {Azimifard, A.},
  title         = {Uniform sine-kernel determinant asymptotics, tail-side quantiles, and prolate eigenvalue bounds},
  year          = {2026},
  eprint        = {2608.15808v2},
  archiveprefix = {arXiv},
  primaryclass  = {math.FA},
  note          = {arXiv:2608.15808v2 [math.FA]},
  doi           = {10.48550/arXiv.2608.15808}
}

@book{BhatiaMatrix,
  author    = {Bhatia, Rajendra},
  title     = {Matrix Analysis},
  series    = {Graduate Texts in Mathematics},
  volume    = {169},
  publisher = {Springer},
  address   = {New York},
  year      = {1997},
  doi       = {10.1007/978-1-4612-0653-8}
}

@article{BonamiJamingKaroui,
  author  = {Bonami, Aline and Jaming, Philippe and Karoui, Abderrazek},
  title   = {Non-asymptotic behavior of the spectrum of the sinc-kernel operator and related applications},
  journal = {Journal of Mathematical Physics},
  volume  = {62},
  number  = {3},
  year    = {2021},
  pages   = {033511},
  doi     = {10.1063/1.5140496},
  eprint  = {1804.01257}
}

@article{CandesDemanetYing,
  author  = {Cand{\`e}s, Emmanuel J. and Demanet, Laurent and Ying, Lexing},
  title   = {Fast Computation of {Fourier} Integral Operators},
  journal = {SIAM Journal on Scientific Computing},
  volume  = {29},
  number  = {6},
  year    = {2007},
  pages   = {2464--2493},
  doi     = {10.1137/060671139},
  eprint  = {math/0610051}
}

@misc{HughesIsraelMayeli,
  author        = {Hughes, Kevin and Israel, Arie and Mayeli, Azita},
  title         = {Trace bounds for limiting operators on rough domains},
  year          = {2026},
  eprint        = {2607.02996},
  archiveprefix = {arXiv},
  primaryclass  = {math.CA},
  note          = {arXiv:2607.02996 [math.CA]},
  doi           = {10.48550/arXiv.2607.02996}
}

@article{KarnikRombergDavenport,
  author  = {Karnik, Santhosh and Romberg, Justin and Davenport, Mark A.},
  title   = {Improved bounds for the eigenvalues of prolate spheroidal wave functions and discrete prolate spheroidal sequences},
  journal = {Applied and Computational Harmonic Analysis},
  volume  = {55},
  year    = {2021},
  pages   = {97--128},
  doi     = {10.1016/j.acha.2021.04.002},
  eprint  = {2006.00427}
}

@misc{KatsnelsonMachlufI,
  author        = {Katsnelson, Victor and Machluf, Ronny},
  title         = {The truncated {Fourier} operator. {I}},
  year          = {2009},
  eprint        = {0901.2555},
  archiveprefix = {arXiv},
  primaryclass  = {math.CA},
  note          = {arXiv:0901.2555 [math.CA]},
  doi           = {10.48550/arXiv.0901.2555}
}

@misc{KulikovLarsen,
  author        = {Kulikov, Aleksei and Larsen, Martin Dam},
  title         = {Sharp estimates for eigenvalues of localization operators with applications to area laws},
  year          = {2026},
  eprint        = {2603.23832v1},
  archiveprefix = {arXiv},
  primaryclass  = {math.SP},
  note          = {arXiv:2603.23832v1 [math.SP]},
  doi           = {10.48550/arXiv.2603.23832}
}

@article{LandauWidom,
  author  = {Landau, H. J. and Widom, H.},
  title   = {Eigenvalue distribution of time and frequency limiting},
  journal = {Journal of Mathematical Analysis and Applications},
  volume  = {77},
  number  = {2},
  year    = {1980},
  pages   = {469--481},
  doi     = {10.1016/0022-247X(80)90241-3}
}

@article{MarcecaRomeroSpeckbacher,
  author  = {Marceca, Felipe and Romero, Jos{\'e} Luis and Speckbacher, Michael},
  title   = {Eigenvalue estimates for {Fourier} concentration operators on two domains},
  journal = {Archive for Rational Mechanics and Analysis},
  volume  = {248},
  number  = {3},
  year    = {2024},
  pages   = {35},
  doi     = {10.1007/s00205-024-01979-9},
  eprint  = {2301.11685}
}

@book{SimonTrace,
  author    = {Simon, Barry},
  title     = {Trace Ideals and Their Applications},
  edition   = {2},
  series    = {Mathematical Surveys and Monographs},
  volume    = {120},
  publisher = {American Mathematical Society},
  address   = {Providence, RI},
  year      = {2005},
  doi       = {10.1090/surv/120}
}

\bigskip
\noindent
\begin{minipage}{\textwidth}
\textsc{Ahmadreza Azimifard}\\[2pt]
\textit{Harmonic Research \& Technologies, LLC}\\
New York, NY, USA\\
\textit{Email address:}
\href{mailto:afard@harmonicrt.com}{\texttt{afard@harmonicrt.com}}
\end{minipage}

\end{document}